\documentclass[11pt,a4paper,reqno,oneside]{amsart}

\usepackage[T1]{fontenc}
\usepackage[utf8]{inputenc}
\usepackage{lmodern}
\usepackage{microtype}
\usepackage{geometry}

\usepackage{amsthm}
\usepackage{amsmath}
\usepackage{amsfonts} 
\usepackage{amssymb}
\usepackage{mathtools}
\usepackage{tikz-cd}
\usepackage{esint}

\usepackage[unicode=true, pdflang={en-GB}, pdfusetitle]{hyperref}
\let\C\undefined

\usepackage[inline]{enumitem}
\usepackage{moreenum}
\usepackage[abbrev]{amsrefs}
\usepackage{thmtools}
\usepackage[capitalize]{cleveref}
\usepackage{constants}

\numberwithin{equation}{section}

\newtheorem{proposition}{Proposition}[section]
\newtheorem{theorem}[proposition]{Theorem}
\newtheorem{lemma}[proposition]{Lemma}
\newtheorem{corollary}[proposition]{Corollary}
\theoremstyle{definition}

\newtheorem{remark}[proposition]{Remark}

\DeclareMathOperator{\id}{id}
\DeclareMathOperator{\dist}{dist}
\DeclareMathOperator{\tr}{tr}

\newcommand{\sm}{\textnormal{sm}}
\newcommand{\lip}{\textnormal{lip}}
\newcommand{\ext}{\textnormal{ext}}

\newcommand{\restr}[1]{\vert_{#1}}
\newcommand{\Deriv}{\mathrm{D}}
\newcommand{\defeq}{\coloneqq}

\newcommand{\Nset}{\mathbb{N}}
\newcommand{\Zset}{\mathbb{Z}}
\newcommand{\Rset}{\mathbb{R}}
\newcommand{\Sset}{\mathbb{S}}
\newcommand{\Tset}{\mathbb{T}}
\newcommand{\Bset}{\mathbb{B}}

\newcommand{\dif}{\,\mathrm{d}}

\newcommand{\compose}{\,\circ\,}
\newcommand{\manifold}[1]{\mathcal{#1}}
\newcommand{\lifting}[1]{\smash{\widetilde{#1}}}

\DeclarePairedDelimiter{\brk}{(}{)}
\DeclarePairedDelimiter{\abs}{\lvert}{\rvert}
\DeclarePairedDelimiter{\norm}{\lVert}{\rVert}
\DeclarePairedDelimiter{\floor}{\lfloor}{\rfloor}
\DeclarePairedDelimiterX{\intvc}[2]{[}{]}{#1,#2}
\DeclarePairedDelimiterX{\intvl}[2]{(}{]}{#1,#2}
\DeclarePairedDelimiterX{\intvr}[2]{[}{)}{#1,#2}
\DeclarePairedDelimiterX{\intvo}[2]{(}{)}{#1,#2}

\newcommand\stSymbol[1][]{%
\nonscript\;#1\vert
\allowbreak
\nonscript\;
\mathopen{}}
\DeclarePairedDelimiterX\set[1]\{\}{%
\renewcommand\st{\stSymbol[\delimsize]}
#1
}
\providecommand{\st}{\stSymbol}

\renewcommand{\PrintDOI}[1]{%
  \href{http://dx.doi.org/#1}{doi:#1}%
}
\BibSpec{book}{%
    +{}  {\PrintPrimary}                {transition}
    +{,} { \textit}                     {title}
    +{.} { }                            {part}
    +{:} { \textit}                     {subtitle}
    +{,} { \PrintEdition}               {edition}
    +{}  { \PrintEditorsB}              {editor}
    +{,} { \PrintTranslatorsC}          {translator}
    +{,} { \PrintContributions}         {contribution}
    +{,} { }                            {series}
    +{,} { \voltext}                    {volume}
    +{,} { }                            {publisher}
    +{,} { }                            {organization}
    +{,} { }                            {address} 
    +{,} { \PrintDateB}                 {date}
    +{,} { }                            {status}
    +{,} { \PrintDOI}                   {doi}
    +{}  { \parenthesize}               {language}
    +{}  { \PrintTranslation}           {translation}
    +{;} { \PrintReprint}               {reprint}
    +{.} { }                            {note}
    +{.} {}                             {transition}
    +{}  {\SentenceSpace \PrintReviews} {review}
}

\title%
[%
Analytical obstructions to the weak approximation of Sobolev mappings%
]%
{%
Analytical obstructions to the weak approximation of Sobolev mappings into manifolds%
}

\author{Antoine Detaille}
\address[A. Detaille]{
	Universite Claude Bernard Lyon 1\\ CNRS\\ Centrale Lyon\\ INSA Lyon\\ Université Jean Monnet\\ ICJ UMR5208\\
	69622 Villeurbanne\\
	France}
\curraddr[A. Detaille]{
	ETH Z\"urich\\
	Departement Mathematik\\
	R\"amistrasse 101\\
	8092 Z\"urich \\
	Switzerland
}
\email{antoine.detaille@math.ethz.ch}
\author{Jean Van Schaftingen}

\begin{document}

	\begin{abstract}
		For any integer $ p \geq 2 $, we construct a compact Riemannian manifold $ \mathcal{N} $ such that if $ \dim \mathcal{M} > p $, there is a map in the Sobolev space of mappings $ W^{1,p} (\mathcal{M}, \mathcal{N})$ which is not a weak limit of smooth maps into $ \mathcal{N} $
		due to a mechanism of analytical obstruction.
		For $ p = 4n - 1 $, the target manifold can be taken to be the sphere $ \mathbb{S}^{2n} $
		thanks to the construction by Whitehead product of maps with nontrivial Hopf invariant, generalizing the result by Bethuel for $ p = 4n -1 = 3$.
		The results extend to higher order Sobolev spaces $ W^{s,p} $, with $ s \in \mathbb{R} $, $s \geq 1 $, $ sp \in \mathbb{N}$, and $ sp \ge 2 $.
	\end{abstract}

\address[J. Van Schaftingen]{
UCLouvain, Institut de Recherche en Math\'ematique et Physique, Chemin du Cyclotron 2 bte L7.01.01, 1348 Louvain-la-Neuve, Belgium}

\email{Jean.VanSchaftingen@UCLouvain.be}

\thanks{This work was initiated during several visits of A. Detaille at the Institut de Recherche en Mathématique et Physique of UCLouvain, and he warmly thanks them for their hospitality.
	Part of this work was completed during a visit of A. Detaille at CUNEF Universidad (Madrid).
	He is grateful to them for their support and hospitality, and to the Université Claude Bernard Lyon 1 for supporting this stay through a funding \emph{Aide à la mobilité doctorale}.\newline
	\indent J. Van Schaftingen was supported by the Projet de Recherche T.0229.21 ``Singular Harmonic Maps and Asymptotics of Ginzburg--Landau Relaxations'' of the Fonds de la Recherche Scientifique--FNRS}

\maketitle

\setcounter{tocdepth}{1}
\tableofcontents
\setcounter{tocdepth}{5}

\section{Introduction}

\subsection{The weak approximation problem}
Given \(p \in \intvr{1}{\infty}\) and Riemannian manifolds \(\manifold{M}\) and \(\manifold{N}\), the \emph{Sobolev space of mappings} from \( \manifold{M} \) to \( \manifold{N} \) is defined as
\[
 W^{1, p} \brk{\manifold{M}, \manifold{N}}
 \defeq
 \set{u \in W^{1, p} \brk{\manifold{M}, \Rset^\nu} \st  \text{\( u\brk{x} \in \manifold{N} \) for a.e.\ \( x \in \manifold{M} \)}}\text{,}
\]
where \(\manifold{N}\) is isometrically embedded into \(\Rset^\nu\), without loss of generality by the Nash isometric embedding theorem~\cites{Nash54, Nash56}.
The Sobolev space \(W^{1, p} \brk{\manifold{M}, \manifold{N}}\) naturally inherits the distance on \(W^{1, p} \brk{\manifold{M}, \Rset^{\nu}} \), with respect to which it enjoys completeness properties.

Sobolev spaces of mappings are the natural functional analytic framework for partial differential equations and calculus of variations problems arising in many contexts such as harmonic maps \citelist{\cite{Helein_Wood_2008}\cite{Eells_Lemaire_1978}},  the modeling of ordered media in condensed matter physics \citelist{\cite{Mermin_1979}\cite{Bethuel_Chiron_2007}}, the study of Cosserat models in elasticity~\cite{Ericksen_Truesdell_1958}, or the description of attitudes of a square or a cube and meshing of domains for numerical  graphics~\cite{Huang_Tong_Wei_Bao_2011}.

A classical fundamental property of linear Sobolev spaces is the \emph{density of smooth maps} in \(W^{1, p} \brk{\manifold{M}, \Rset^\nu}\)~\cites{DenyLions1954, Meyers_Serrin_1964}.
This question of approximation happens to be much more delicate for spaces of mappings \(W^{1, p} \brk{\manifold{M}, \manifold{N}}\).
Indeed, whereas the definition and the classical theory of linear spaces yield for every \(u \in W^{1, p} \brk{\manifold{M}, \manifold{N}}\) a sequence \(\brk{u_k}_{k \in \Nset}\) in \(C^{\infty} \brk{\manifold{M}, \Rset^\nu}\) that strongly converges to \(u\), there is no reason to hope that the standard linear convolution construction would provide a map satisfying the nonlinear manifold constraint.

The \emph{strong approximation problem} asks whether
\begin{equation}
\label{eq_shee2ee9Ieghooja6wei4bee}
 H^{1, p}_S \brk{\manifold{M}, \manifold{N}} = W^{1, p} \brk{\manifold{M}, \manifold{N}}\text{,}
\end{equation}
where
\(
 H^{1, p}_S \brk{\manifold{M}, \manifold{N}}\)
 is the strong closure of \(C^\infty \brk{\manifold{M}, \manifold{N}}\) in \(W^{1, p} \brk{\manifold{M}, \manifold{N}}\).
When \(p \ge \dim \manifold{M}\),
Sobolev mappings are continuous, or almost --- more precisely, they are continuous for \( \dim \manifold{M} > p \) and \( \operatorname{VMO} \) for \( \dim\manifold{M} = p \) --- and the strong approximation property~\eqref{eq_shee2ee9Ieghooja6wei4bee} holds; see~\cites{Schoen_Uhlenbeck_1983, Brezis_Nirenberg_1995}.
When \(p < \dim \manifold{M}\), the answer depends on the  \emph{topology of the manifolds} \( \manifold{M} \) and \( \manifold{N} \).
For instance, when \(\pi_{1} \brk{\manifold{M}} \simeq \dotsb \simeq \pi_{\floor{p - 1}} \brk{\manifold{M}} \simeq \set{0}\),
which includes the sphere \(\manifold{M} = \Sset^m\) with \(m \ge p\), Bethuel~\cite{Bethuel_1991} has proved that the strong approximation property~\eqref{eq_shee2ee9Ieghooja6wei4bee} holds if and only if
\(
 \pi_{\floor{p}} \brk{\manifold{N}} \simeq \set{0}
\), where \(\pi_{\ell} \brk{\manifold{N}}\) denotes the \(\ell^{\text{th}}\) homotopy group of \(\manifold{N}\).
(In the general case,~\eqref{eq_shee2ee9Ieghooja6wei4bee} holds if and only if there is no obstruction to the extension of a continuous map from a \(\floor{p}\)-dimensional triangulation of \(\manifold{M}\) into \(\manifold{N}\) to a continuous map from \(\manifold{M}\) into \(\manifold{N}\)~\cite{Hang_Lin_2003_II}.)

When smooth maps fail to be strongly dense in the space of Sobolev mappings, a natural approach is to investigate approximation with respect to \emph{weaker} but still useful other notions of convergence.
One can thus consider
the \emph{sequential weak closure} of smooth maps 
\( H^{1, p}_W \brk{\manifold{M}, \manifold{N}}
\) as the set of mappings \(u \in W^{1, p} \brk{\manifold{M}, \manifold{N}}\) that are almost everywhere limits of a sequence \(\brk{u_k}_{k \in \Nset}\) in \(C^\infty \brk{\manifold{M}, \manifold{N}}\) that is bounded in \(W^{1, p} \brk{\manifold{M}, \manifold{N}}\).
This set \( H^{1, p}_W \brk{\manifold{M}, \manifold{N}}
\) is quite natural in partial differential equations and calculus of variations, where sequences of approximate solutions or minimizers typically converge weakly rather than strongly.
Moreover, weak convergence may sometimes be sufficient to extend desirable properties from smooth to Sobolev mappings; see~\cite{Mironescu_VanSchaftingen_2021_APDE}*{Theorem~3} for one instance of such phenomenon.

The \emph{weak approximation problem} is to determine whether
\begin{equation}
\label{eq_cahkoh3me3ohBeeQu7iengeb}
 H^{1, p}_W \brk{\manifold{M}, \manifold{N}} = W^{1, p} \brk{\manifold{M}, \manifold{N}}\text{.}
\end{equation}
It follows immediately from the definitions that
\(
  H^{1, p}_S \brk{\manifold{M}, \manifold{N}} \subseteq H^{1, p}_W \brk{\manifold{M}, \manifold{N}}\),
so that the strong approximation property~\eqref{eq_shee2ee9Ieghooja6wei4bee} implies the weak one~\eqref{eq_cahkoh3me3ohBeeQu7iengeb}.

When \(p \not \in \Nset\), Bethuel has proved~\cite{Bethuel_1991}*{Theorem~3} that \(H^{1, p}_S \brk{\manifold{M}, \manifold{N}} = H^{1, p}_W \brk{\manifold{M}, \manifold{N}}\) and that the strong and weak approximation properties~\eqref{eq_shee2ee9Ieghooja6wei4bee} and~\eqref{eq_cahkoh3me3ohBeeQu7iengeb} are hence equivalent.
In the only remaining interesting case where \( p \in \Nset \) and the strong approximation fails,
weak approximation \emph{may still hold}.
If \(\pi_{p} \brk{\manifold{N}} \not \simeq \set{0}\), then
\(H^{1, p}_S \brk{\manifold{M}, \manifold{N}} \subsetneq H^{1, p}_W \brk{\manifold{M}, \manifold{N}}\)~\citelist{\cite{Bethuel_1991}*{Theorem~5}\cite{Hang_Lin_2003_III}*{Theorem~5.5}}.
In a first instance, Bethuel has proved in~\cite{Bethuel_1990} that
\[
	H^{1,2}_{S}\brk{\Bset^{3}, \Sset^{2}} \subsetneq H^{1,2}_{W}\brk{\Bset^{3}, \Sset^{2}} = W^{1,2}\brk{\Bset^{3}, \Sset^{2}}\text{.}
\]
More generally, Bethuel~\cite{Bethuel_1991}*{Theorem~6}, and finally Hajłasz~\cite{Hajlasz_1994} have proved that if \(p \in \Nset = \set{1, 2, \dotsc}\) and if
\(
\pi_1 \brk{\manifold{N}} \simeq \dotsb \simeq \pi_{p - 1} \brk{\manifold{N}}\simeq \set{0}\),
then the weak approximation property~\eqref{eq_cahkoh3me3ohBeeQu7iengeb} holds. 
In particular,~\eqref{eq_cahkoh3me3ohBeeQu7iengeb} always holds when \(p = 1\).
(By lack of reflexivity for \(p = 1\), the space \(H^{1,1}_W \brk{\manifold{M}, \manifold{N}}\) need not be the sequential closure of \(C^\infty \brk{\manifold{M}, \manifold{N}}\) with respect to the weak topology of \(W^{1,1} \brk{\manifold{M}, \Rset^\nu}\) induced by its dual space which Hang has shown to be \(W^{1,1} \brk{\manifold{M}, \manifold{N}}\)~\cite{Hang_2002}.)

Further relevant research on this question
includes the work of Pakzad and Rivière~\cite{pakzad_riviere_2003} who showed that, when \( \manifold{M} \) is simply connected and \(p = 2\),~\eqref{eq_cahkoh3me3ohBeeQu7iengeb} holds for more general target manifolds \( \manifold{N} \) than those covered by Haj\l asz's result, including some non-simply connected \( \manifold{N} \), which are not handled by~\cite{Hajlasz_1994};
the adaptation of the methodology to give a new proof of~\eqref{eq_cahkoh3me3ohBeeQu7iengeb} when \(p = 1\)~\cite{pakzad_2003};
the development of the concept of \emph{scan}, in order to provide a new tool for understanding strong and weak approximation problems, in the case of \( W^{1,3}\brk{\manifold{M}, \Sset^{2}} \) and in the more general case  of a topological obstruction induced by any nontorsion part of \( \pi_{p}\brk{\manifold{N}} \)~\citelist{\cite{Hardt_Riviere_2003}\cite{Hardt_Riviere_2008}}.

\subsection{Analytical obstructions to the weak approximation}

All these partial results suggested that no local obstruction might arise for weak density when \(p \in \Nset\), that is,
when \(\manifold{M}\) is topologically sufficiently simple, then \( H^{1,p}_{W}\brk{\manifold{M}, \manifold{N}} = W^{1,p}\brk{\manifold{M}, \manifold{N}} \) regardless of \( \manifold{N} \) \citelist{\cite{Hang_Lin_2003_II}*{Conjecture 7.1}\cite{Bethuel_1991}}.
However, in a groundbreaking contribution, Bethuel \cite{Bethuel_2020} has shown the presence of \emph{analytical} obstructions to the weak approximation problem: if \(\dim \manifold{M} \ge 4\), then
\begin{equation}
\label{eq_iehon4gaitho3temeihaeFeM}
  H^{1, 3}_W \brk{\manifold{M}, \Sset^2}
  \subsetneq W^{1, 3} \brk{\manifold{M}, \Sset^2}\text{.}
\end{equation}
Whereas the role of analytical obstructions had already been well-understood for the lifting~\citelist{\cite{Bourgain_Brezis_Mironescu_2000}\cite{Bethuel_Chiron_2007}\cite{Mironescu_VanSchaftingen_2021_APDE}} and for the extension of traces~\citelist{\cite{Bethuel_Demengel_1995}\cite{Bethuel_2014}\cite{Mironescu_VanSchaftingen_2021_AFST}} for Sobolev mappings, Bethuel's result was the first and only known instance of analytical obstruction for the weak approximation problem for compact manifolds --- noncompact manifolds can exhibit analytical obstructions provoked by the geometry of their ends \cite{Bousquet_Ponce_VanSchaftingen_2018}.

The main result of the present work is that analytic obstructions to the weak approximation actually occur for \emph{every} \(p \in \Nset \setminus \set{1}\).

\begin{theorem}
\label{theorem_main}
For every \(p \in \Nset \setminus \set{1}\), there exists a compact manifold \(\manifold{N}\)
such that if \(\dim \manifold{M} > p\), then 
\[
 H^{1, p}_W \brk{\manifold{M}, \manifold{N}}
  \subsetneq W^{1, p} \brk{\manifold{M}, \manifold{N}}\text{.}
\]
\end{theorem}

In fact the conclusion of \ref{theorem_main}
is slightly stronger: we have
\[
	H^{1, p}_W \brk{\manifold{M}, \manifold{N}}
	\subsetneq
	\smash{\overline{H^{1, p}_W \brk{\manifold{M}, \manifold{N}}}}
	\subseteq
	W^{1, p} \brk{\manifold{M}, \manifold{N}}
\]
(\cref{theorem_non_strongly_closed}) and \(H^{1, p}_W \brk{\manifold{M}, \manifold{N}}\) is a meagre subset of \(\smash{\overline{H^{1, p}_W \brk{\manifold{M}, \manifold{N}}}}\) in the sense of Baire category theory for the strong topology (\cref{theorem_meagre}).

In particular, Theorem~\ref{theorem_main} is the \emph{first} instance of the failure of the weak approximation property when \( \manifold{M} \) is a ball for \( p \neq 3 \).

The manifold \(\manifold{N}\) is defined explicitly, depending on \(p\); it retracts onto the \(p\)-dimensional skeleton of a \(\brk{p + 1}\)-dimensional torus.
Moreover, one can show that \( \manifold{N} \) can be obtained as
\begin{equation}
\label{eq_ciix9tievuof8aiSahWie4ho}
	\manifold{N} \simeq \brk{\Tset^{p+1} \setminus \Bset^{p+1}} \times \Sset^{m - 1} \cup_{\partial} \Sset^{p} \times \Bset^m\text{,}
\end{equation}
where \( m \in \Nset \) is a parameter involved in the construction.
There is a trade-off in the choice of the parameter \( m \).
On the one hand, one can choose it to be minimal, that is, take \( m = 1 \), in order to obtain a target of dimension as small as possible.
In this case, it can be observed that the target manifold \( \manifold{N} \) is actually a connected sum of two \(\brk{p+1}\)-dimensional tori \( \Tset^{p + 1} = \brk{\Sset^1}^{p + 1}\):
\[
	\manifold{N} \simeq \Tset^{p+1}\,\#\, \Tset^{p+1}\text{.}
\]
On the other hand, one can try to choose \( m \) in order to have as much information as possible on the topology of the resulting target \( \manifold{N} \).
In this case, this leads to taking \( m \) large in \eqref{eq_ciix9tievuof8aiSahWie4ho}.
For instance, choosing \( m > p \), one can ensure that \(\pi_2 \brk{\manifold{N}} \simeq \dotsb \simeq \pi_{p - 1} \brk{\manifold{N}} \simeq \set{0}\) while \(\pi_{1} \brk{\manifold{N}}\) is abelian and \(\pi_{p} \brk{\manifold{N}}\) is finitely generated as a \(\pi_1 \brk{\manifold{N}}\)-module but not as a \(\Zset\)-module.
The precise description of the target \( \manifold{N} \) that we construct, as well as the proof of its topological properties, will be given in Proposition~\ref{proposition_manifold} and Remark~\ref{rmk:target_form}.

Let us explain how Theorem~\ref{theorem_main} fits into the existing weak approximation results.
So far, there are three strategies to prove the weak density.
The first one is to deduce it from the strong approximation; this works when \(\pi_{p} \brk{\manifold{N}} \simeq \set{0}\), which will not be the case in the setting of Theorem~\ref{theorem_main}.
A second strategy relies on a controlled almost retraction \cite{Hajlasz_1994} (see also \cite{Bethuel_Zheng_1988}); for topological reasons, the restriction \(\pi_{1} \brk{\manifold{N}} \simeq  \dotsb \simeq \pi_{p - 1} \brk{\manifold{N}}\simeq \set{0}\) is essential.
The last class of methods is based first on the construction and analysis of a topological singular set and of minimal connections, and second on the elimination of the topological singularities via a dipole construction; see e.g.~\citelist{\cite{BrezisCoronLieb1986}\cite{AlmgrenBrowderLieb1988}\cite{Bethuel_1990}\cite{AlbertiBaldoOrlandi2003}\cite{Canevari_Orlandi_2019}} for mappings into the sphere, and~\cite{pakzad_riviere_2003} for more general targets.
Although there is no evidence that the assumption \(\pi_{1} \brk{\manifold{N}} \simeq \dotsb \simeq \pi_{p - 1} \brk{\manifold{N}} \simeq \set{0}\) should be essential, it simplifies the setting considerably, since one then only has to consider the charges in \(\pi_{p} \brk{\manifold{N}}\) without having to take into account their interplay with lower-dimensional phenomena.
In a situation like ours, where the only nontrivial lower homotopy group is \(\pi_1 \brk{\manifold{N}}\), one could naturally try to use the universal covering \(\lifting{\manifold{N}}\)
for which \(\pi_1 \brk{\lifting{\manifold{N}}} \simeq \dotsb \simeq \pi_{p - 1} \brk{\lifting{\manifold{N}}} \simeq \set{0}\).
Although from the point of view of homotopy theory,
the \(\brk{p-1}\)-dimensional skeleton of \(\lifting{\manifold{N}}\) is homotopically equivalent to a point, there is no reason to believe that there is in general a reasonable quantitative control on the resulting homotopy that could be used for analytical constructions.
Actually, in our situation, the non-simple connectedness of \(\manifold{N}\) can be exploited to thwart this weak approximation scheme.
More details about the phenomenon at work here will be given after Proposition~\ref{proposition_topological_properties_skeleton}, where the precise topological properties of the target manifold we construct are proved.

Examples of topological obstructions to the weak approximation had been provided when the target manifold \(\manifold{N}\) is not compact \cite{Bousquet_Ponce_VanSchaftingen_2018};
whereas the latter result was based on a target with a nonuniform geometry, our result shows that analytical obstructions do arise even with a uniform geometry. A complete noncompact manifold can be obtained by either considering the universal covering of \(\manifold{N}\), which is not compact since \(\pi_1\brk{\manifold{N}}\) is infinite, or simply \(\manifold{N} \times \Rset\).

Even though~\cref{theorem_main} covers all the integer exponents \( p \ge 2 \), the resulting manifold \(\manifold{N}\) is not as simple as the sphere \(\Sset^2\) in Bethuel’s result \eqref{eq_iehon4gaitho3temeihaeFeM}.
Using a variant of our construction, we also recover Bethuel's counterexample, and show that it is actually part of an infinite family.

\begin{theorem}
	\label{theorem_counterexample_hopf}
	For every \(n \in \Nset \), if \(\dim \manifold{M} > 4 n-1\), then 
	\[
	H^{1, 4n - 1}_W \brk{\manifold{M}, \Sset^{2n}}
	\subsetneq W^{1, 4n-1} \brk{\manifold{M}, \Sset^{2n}}\text{.}
	\]
\end{theorem}

As for \cref{theorem_main}, one can prove that
\[
	H^{1, 4n-1}_W \brk{\manifold{M}, \Sset^{2n}}
	\subsetneq
	\smash{\overline{H^{1, 4n-1}_W \brk{\manifold{M}, \Sset^{2n}}}}
	\subseteq
	W^{1, 4n-1} \brk{\manifold{M}, \Sset^{2n}}
\]
and that \(H^{1, 4n-1}_W \brk{\manifold{M}, \Sset^{2n}}\) is a meagre subset of \(\smash{\overline{H^{1, 4n-1}_W \brk{\manifold{M}, \Sset^{2n}}}}\).

\subsection{Analytical vs.\ topological obstructions}
The obstructions in Theorems~\ref{theorem_main} and~\ref{theorem_counterexample_hopf} are called \emph{analytical}, as opposed to the other already known \emph{topological} obstructions. This states that, if \(p \in \Nset\) and if  density~\eqref{eq_cahkoh3me3ohBeeQu7iengeb} holds, then one has \cite{Hang_Lin_2003_II}
\begin{equation}
	\label{eq_ceig4JicohCe1eeda3xeiw1i}
	\tr_{\manifold{M}^{p-1}} \brk{C \brk{\manifold{M}, \manifold{N}}}
	= \tr_{\manifold{M}^{p-1}}  \brk{C \brk{\manifold{M}^{p}, \manifold{N}}}\text{.}
\end{equation}
In other words, if all Sobolev mappings are weak limits of smooth mappings,
then the restriction of a continuous map from the \(p\)-dimensional component \(\manifold{M}^{p}\) of a triangulation of \(\manifold{M}\) to its \(\brk{p-1}\)-dimensional component \(\manifold{M}^{p - 1}\) is the restriction of a continuous map defined on the entire manifold.
If \( \manifold{M} \) is a ball, a cube, a sphere, or more generally if \(\pi_{1} \brk{\manifold{M}} \simeq \dotsb \simeq \pi_{p - 1} \brk{\manifold{M}} \simeq \set{0}\), then the condition~\eqref{eq_ceig4JicohCe1eeda3xeiw1i} is satisfied independently of \(\manifold{N}\).

The difference between topological and analytical obstructions can be formalized by considering the strong closure
 \(\smash{\overline{H^{1, p}_W \brk{\manifold{M}, \manifold{N}}}}\) of \(H^{1, p}_W \brk{\manifold{M}, \manifold{N}}\) in \(W^{1, p} \brk{\manifold{M}, \manifold{N}}\).
This set can be characterised as the set of mappings in \(W^{1, p} \brk{\manifold{M}, \manifold{N}}\) that are connected by a path to \(C^\infty \brk{\manifold{M}, \manifold{N}}\)~\citelist{\cite{Hang_Lin_2003_III}*{Theorem~5.5}}, which is weakly sequentially closed and open~\cite{Hang_Lin_2003_II}*{Proposition~4.1} (see also \cite{White_1988}).
The topological obstruction can be refined~\citelist{\cite{Bethuel_1991}*{Theorem~5 and its proof}\cite{Hang_Lin_2003_III}*{Theorem~5.5}}  into the equivalence between~\eqref{eq_ceig4JicohCe1eeda3xeiw1i}
and equality in the second inclusion in
\begin{equation}
\label{eq_baWooMah0OoJ1yu8laingoon}
H^{1, p}_W \brk{\manifold{M}, \manifold{N}}
\subseteq
\smash{\overline{H^{1, p}_W \brk{\manifold{M}, \manifold{N}}}}
\subseteq W^{1, p} \brk{\manifold{M}, \manifold{N}}\text{.}
\end{equation}
The analytical obstructions in Bethuel's work and ours exhibit a failure of equality in the first inclusion in~\eqref{eq_baWooMah0OoJ1yu8laingoon}; the proofs even show that \(H^{1, p}_W \brk{\manifold{M}, \manifold{N}}\) is meager relatively to \(\smash{\overline{H^{1, p}_W \brk{\manifold{M}, \manifold{N}}}}\) with respect to the strong topology (see~\cite{Monteil_VanSchaftingen_2019}*{Theorem~3.3}).

\subsection{About the proof}

As our main constructions take place in dimension \( n = p+1 \), from now on, we switch notation and use the dimension \( n \) as our main variable.
For a given integer \( n \geq 3 \), the basic construction in our proof of~\cref{theorem_main}
defines \(
  u \colon \Rset^{n} \setminus \Sigma
  \to \lifting{\manifold{N}}_{0}
\)
as the (singular) retraction of \(\Rset^{n} \setminus \Sigma\) to \(\lifting{\manifold{N}}_{0} \), where \(\lifting{\manifold{N}}_{0} \) is the \(\brk{n-1}\)-dimensional component of the decomposition of \(\Rset^{n}\) into cubes with vertices in \(\Zset^{n}\) and \(\Sigma \defeq \brk{\Zset +1/2}^{n}\) is the corresponding \(0\)-dimensional dual complex.
More specifically, \( u \) is defined on every cube centered at the point \(\sigma \in \Sigma\) as \(u \brk{x}  \defeq \sigma + \brk{x - \sigma}/\brk{2 \abs{x - \sigma}_{\infty}}\).

For every \(\ell \in \Nset\), we have by periodicity of \( u \)
\[
  \int_{\intvc{0}{\ell}^{n}} \abs{\Deriv u}^{n-1} =
  \ell^{n} \int_{\intvc{0}{1}^{n}} \abs{\Deriv u}^{n-1}\text{.}
\]
The relaxed energy of a Sobolev map \(w\) on a domain \( \manifold{M} \) is defined, following \cite{Bethuel_Brezis_Coron_1988},  as
\begin{multline}
\label{eq_oochohL1iezal0eephoo5phe}
 \mathcal{E}^{1, p}_{\mathrm{rel}} \brk{w, \manifold{M}}\\
 \defeq
 \inf \set[\bigg]{\liminf_{k \to \infty} \mathcal{E}^{1,p}\brk{u_{k}, \manifold{M}} \st u_{k} \to w \text{ a.e.} \text{ and }   u_{k} \in W^{1,p}\brk{\manifold{M}, \lifting{\manifold{N}}_{0}} \cap C\brk{\manifold{M}, \lifting{\manifold{N}}_{0}}}\text{,}
\end{multline}
where the Sobolev energy is defined by
\[
	\mathcal{E}^{1,p}\brk{w, \manifold{M}} \defeq \int_{\manifold{M}} \abs{\Deriv w}^{p}\text{.}
\]
Although our definition of the relaxed energy is not identical to the
usual definition where the infimum is taken over sequences of smooth maps, they are nonetheless equivalent when the target is a smooth manifold.
Indeed, any continuous Sobolev mapping with values into a smooth manifold may always be \emph{strongly} approximated by smooth maps by a classical regularization and reprojection argument.
The equivalence then follows from a diagonal argument, since the convergence in measure is metrizable.
The advantage of working with the alternative definition is that it makes sense and is tractable also when the target is only a skeleton, which is instrumental to us in the intermediate steps of the proof, for instance here when working with \( \smash{\widetilde{\manifold{N}}_{0}} \).

If we can show that the relaxed energy \( \mathcal{E}^{1, n-1}_{\mathrm{rel}} \brk{u, \intvc{0}{\ell}^{n}} \) of \( u \) on \( \intvc{0}{\ell}^{n} \) grows faster than the Sobolev energy \( \mathcal{E}^{1,n-1}\brk{u, \intvc{0}{\ell}^{n}} \) as \( \ell \to \infty \),
the obstruction to the weak approximation will follow from a nonlinear uniform boundedness principle~\cite{Hang_Lin_2003_III}*{Theorem~9.6} (see also~\cite{Monteil_VanSchaftingen_2019} and \cref{thm:NonlinUnifBddPrinciple} below), which is a kind of nonlinear counterpart of the classical Banach--Steinhaus theorem in functional analysis: if there exists a sequence \( \brk{u_{\ell}}_{\ell \in \Nset} \) in \( H^{1,n-1}_W
	\brk{\Bset^m, \manifold{N}} \) such that
	\begin{equation}
	\label{eq_d537ac00f9c8ab35b}
		\sup_{\ell \in \Nset} \frac{\mathcal{E}^{1,n-1}_{\mathrm{rel}} \brk{u_{\ell}, \Bset^{m}}}{\mathcal{E}^{1,n-1} \brk{u_{\ell}, \Bset^{m}}}
		=
		+\infty\text{,}
	\end{equation}
	then there exists a map \( v \in W^{1,n-1} \brk{\Bset^m, \manifold{N}} \setminus H^{1,n-1}_{W}\brk{\Bset^m, \manifold{N}} \).

In our situation, the maps \( u_{\ell} \) will be taken to be \( u\restr{\intvc{0}{\ell}^{n}} \), or more precisely restrictions of scaled copies of \(u\) to the domain \( \Bset^{n} \) --- since both the Sobolev and the relaxed energies have the same scaling, this rescaling does not affect the validity of \eqref{eq_d537ac00f9c8ab35b}.

In the proof of the nonlinear uniform boundedness principle, the map \( v \) is essentially constructed as follows.
By a scaling argument, using the superlinear growth assumption of the relaxed energy with respect to the Sobolev energy, we construct maps \( v_{\ell} \in W^{1,n-1}\brk{B_{r_{\ell}}, \manifold{N}} \), for some radius \( r_{\ell} > 0 \) depending on \( \ell \), such that
\[
	\sum_{\ell \in \Nset} \mathcal{E}^{1,n-1}\brk{v_{\ell}, B_{r_{\ell}}}
	<
	+\infty
\]
while
\[
	\sum_{\ell \in \Nset} \mathcal{E}^{1,n-1}_{\mathrm{rel}}\brk{v_{\ell}, B_{r_{\ell}}}
	=
	+\infty\text{.}
\]
The counterexample \( v \) is then obtained by gluing together (translated copies of) the \( v_{\ell} \) on disjoint balls.
Technicalities that we leave aside in this short exposition include notably how to make some appropriate surgery on the maps \( v_{\ell} \) to make them compatible on the boundary so that they can be glued together without destroying the required conclusion, and how to ensure that the radii \( r_{\ell} \) can be chosen sufficiently small so that we indeed have at our disposal disjoint balls on which to insert all the \( v_{\ell} \) --- this relies on the fact that the Sobolev exponent is less than the dimension of the domain.

In order to achieve \eqref{eq_d537ac00f9c8ab35b}, we will compare \(\mathcal{E}^{1, n-1}_{\mathrm{rel}} \brk{u, \intvc{0}{\brk{2q+1}\ell}^{n}}\) to \(\mathcal{E}^{1, n-1}_{\mathrm{rel}} \brk{u, \intvc{0}{\ell}^{n}}\) for a suitable \(q \in \Nset\).
A simple additivity and translation argument shows that 
\begin{equation}
\label{eq_OogaiSh3neiy7ahl6eeF7eiV}
 \mathcal{E}^{1, n-1}_{\mathrm{rel}} \brk{u, \intvc{0}{\brk{2q+1}\ell}^{n}} \ge
 \brk{2q+1}^{n} \mathcal{E}^{1, n-1}_{\mathrm{rel}} \brk{u, \intvc{0}{\ell}^{n}}\text{;}
\end{equation}
we are going to strengthen this inequality.
For this purpose, we consider a sequence \(\brk{u_k}_{k \in \Nset}\) in \(W^{1,n-1}\brk{\intvc{0}{\brk{2q+1}\ell}^{n}, \lifting{\manifold{N}}_{0}} \cap C\brk{\intvc{0}{\brk{2q+1}\ell}^{n}, \lifting{\manifold{N}}_{0}} \) realizing the infimum in~\eqref{eq_oochohL1iezal0eephoo5phe}.
By a classical Fatou and Fubini--Tonelli argument, the sequence converges weakly on the boundary \(\partial Q\) of some cube \(Q\) with the same center as \(\intvc{0}{\brk{2q+1}\ell}^{n}\) and edge length between \(\brk{2q - 1}\ell\) and \(\brk{2q+1}\ell\), with
\begin{equation}
\label{eq_ahd7Iphee8deeYoos4daeth7}
 \liminf_{k \to \infty} \int_{\partial Q} \abs{\Deriv u_{k}}^{n-1}
 \le \frac{C}{\ell} \liminf_{k \to \infty} \int_{\intvc{0}{\brk{2q+1}\ell}^{n}} \abs{\Deriv u_{k}}^{n-1}\text{.}
\end{equation}
The sequence \(\brk{u_{k} \restr{\partial Q}}_{k \in \Nset}\) is thus a sequence of maps homotopic to a constant converging to \(u\restr{\partial Q}\) which is not homotopic to a constant.
Using the fact that \(\dim \partial Q = n-1\), we show in Section~\ref{section_bubbling} that, when \(k\) is sufficiently large, there is a finite family of disjoint balls such that \(u_{k}\) and \(u\) are homotopic outside these balls. (Although such bubbling phenomena are ubiquitous in the analysis of Sobolev and harmonic mappings, we had to develop a statement providing all the information we need in Proposition~\ref{proposition_bubble_decomposition}.)
The gap in homotopy classes between \(u_k\) and \(u\) needs thus to be compensated by the homotopical charge borne on the small balls.
In geometric terms, every singular point should be engulfed in the image of \(u_{k}\) on some singular ball.
This can be quantified by the isoperimetric theorem.
Namely, the isoperimetric theorem asserts that, if \(u_{k} = b\) on \(\partial B_{\rho} \brk{a}\cap \partial Q\), the volume engulfed in the image of the small ball \( B_{\rho}\brk{a} \cap \partial Q\) is controlled by
\begin{equation*}
   \brk[\Big]{\int_{B_\rho \brk{a}\cap \partial Q} \mathcal{J} u_{k}}^{\frac{n-1}{n}} \le  \brk[\bigg]{\frac{1}{\brk{n-1}^{\frac{n-1}{2}}}\int_{B_\rho \brk{a}\cap \partial Q} \abs{\Deriv u_{k}}^{n-1}}^{\frac{n-1}{n}}
\end{equation*}
(where \(\mathcal{J} u_k = \abs{\det \brk{\Deriv u_k}}\) is the Jacobian of \( u_k \), and the above inequality follows from the arithmetico-geometric inequality).
In turn, translating the engulfed volume in terms of the number of enclosed singularities, this estimate implies that a small disk containing \(\ell^{n}\) singularities yields a total contribution to the energy \(\mathcal{E}^{1, n-1} \brk{u_k, \partial Q}\) of order of \(\ell^{n-1}\), and thus, by \eqref{eq_ahd7Iphee8deeYoos4daeth7}, a total contribution of order \(\ell^{n}\) to the energy \(\mathcal{E}^{1,n-1}\brk{u_{k}, \intvc{0}{\brk{2q+1}\ell}^{n}} \).

\begin{figure}
\includegraphics{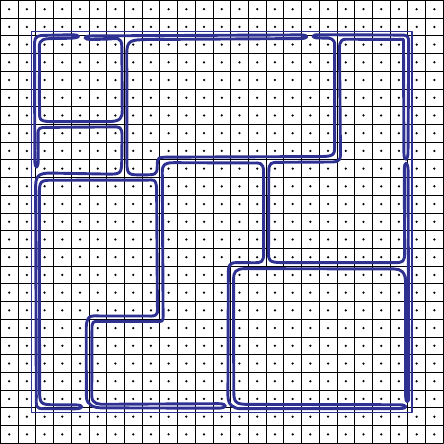}
\caption{On a generic square, a smooth map \( u_k \) approximating \( u \) should take at most points a value close to \( u \) while engulfing at the other points all the singularities through the creation of bubbles.}
\end{figure}

In order to transform this idea into an improvement of~\eqref{eq_OogaiSh3neiy7ahl6eeF7eiV}, we localize the estimate on the number of engulfed singularities.
More specifically, we consider \(\Pi\) a suitable retraction from \(\lifting{\manifold{N}}_{0}\) to \(\intvc{0}{\ell}^{n} \cap \lifting{\manifold{N}}_{0}\), and rely on the intermediate maps \( \Pi \compose u_{k} \).
The role of this map \( \Pi \) is to decouple the contribution of the singularities coming from points inside \( \intvc{0}{\ell}^{n} \) and inside a cube far away from \( \intvc{0}{\ell}^{n} \).
Namely, as \( \Pi = \id \) on \( \intvc{0}{\ell}^{n} \cap \lifting{\manifold{N}}_{0} \), using \( \Pi \compose u_{k} \) as a competitor in the computation of the relaxed energy, we get
\[
\liminf_{k \to \infty} \int_{\intvc{0}{\ell}^{n} \cap \partial Q} \abs{\Deriv \brk{\Pi \compose u_{k}}}^{n-1} \ge \mathcal{E}^{1, n-1}_{\mathrm{rel}} \brk{u, \intvc{0}{\ell}^{n} \cap \partial Q}\text{.}
\]
Meanwhile, on the region where \( u_{k} \in \intvc{q \ell}{\brk{q + 1}\ell}^{n} \), if \( q \) is sufficiently large, \( \abs{\Deriv u_{k}}^{n-1} - \abs{\Deriv \brk{\Pi \compose u_{k}}}^{n-1} \) is bounded from below, up to a multiplicative constant, by \( \abs{\Deriv u_{k}}^{n-1} \) alone, as a consequence of the fact that \( \Deriv\Pi \) is small on \( \intvc{q \ell}{\brk{q + 1}\ell}^{n} \).
To exploit this inequality, we rely on a conical estimate for the Brouwer degree, which essentially asserts that the degree of a map is estimated in terms of the energy of its gradient only on the region where the map takes its values into a cone as illustrated on \cref{figure_cone_estimate}; see Section~\ref{subseq:Conical_Degree}.
For singularities coming from points inside \(\intvc{q \ell}{\brk{q + 1}\ell}^{n} \), choosing the direction of the cone ``away'' from \( \intvc{0}{\ell}^{n} \), we ensure that the degree of \( u_{k} \) with respect to these points can be estimated by the integral of \( \abs{\Deriv u_{k}}^{n-1} \) only on a region where the estimate
\begin{equation}
\label{eq_d10b2d748184331f}
	\abs{\Deriv u_{k}}^{n-1} \lesssim \abs{\Deriv u_{k}}^{n-1} - \abs{\Deriv \brk{\Pi \compose u_{k}}}^{n-1}
\end{equation}
holds.
The second ingredient that enters into play at this stage, as we already mentioned, is a precise bubbling statement established in Section~\ref{section_bubbling}, which ensures that the difference of topological charge between \( u \) and \( u_{k} \) comes from the contribution of the degrees of \( u_{k} \) on a finite number of balls, where a bubble forms.
But, by our conical degree estimate and \eqref{eq_d10b2d748184331f}, on a ball \(B_{\rho} \brk{a} \) where a bubble forms, the number of singularities coming from the cube \(\intvc{q \ell}{\brk{q + 1}\ell}^{n}\) --- which is \( \ell^{n} \) --- can be controlled when \(q\) is sufficiently large by
\[
  \brk[\Big]{\int_{B_\rho \brk{a}\cap \partial Q} \abs{\Deriv u_{k}}^{n-1} - \int_{B_{\rho} \brk{a}\cap \partial Q} \abs{\Deriv \brk{\Pi \compose u_{k}}}^{n-1}}^{\frac{n-1}{n}}\text{.}
\]

\begin{figure}
\includegraphics{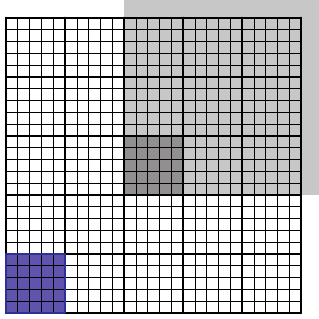}
\caption{We can estimate from below the excess energy on the cube \(\intvc{0}{\ell}^{n}\) in the lower-left corner by the degrees with respect to the points of the grid contained in the central cube thanks to an estimate of the degree that only takes into accounts the value of the map in the upper-right cone.
}
\label{figure_cone_estimate}
\end{figure}

Combining both these decoupled contributions, one gets 
\begin{equation*}
 \mathcal{E}^{1, n-1}_{\mathrm{rel}} \brk{u, \intvc{0}{\brk{2q+1}\ell}^{n}} \ge
 \brk{2q+1}^{n} \mathcal{E}^{1, n-1}_{\mathrm{rel}} \brk{u, \intvc{0}{\ell}^{n}} + c \ell^{n}
\end{equation*}
with \(c > 0\), from which it follows that 
\begin{equation*}
 \lim_{\ell \to \infty} \frac{\mathcal{E}^{1, n-1}_{\mathrm{rel}} \brk{u, \intvc{0}{\ell}^{n}}}{\ell^{n}}= \infty\text{.}
\end{equation*}

If \(\lifting{\manifold{N}}_{0}\) were compact and a manifold, we would have our conclusion by the uniform boundedness principle.
In order to remedy these issues, we first consider the space \(\manifold{N}_{0} = \lifting{\manifold{N}}_{0}/\Zset^{n}\). By construction, \( \lifting{\manifold{N}}_{0}\) is the universal covering of \(\manifold{N}_{0} \), and the latter is compact. 
Since \(p = n-1 \ge 2\), the results of the analysis above can be transferred to get a mapping \(v \colon \Rset^{n} \to \manifold{N}_{0} \) such that
\begin{equation*}
 \lim_{\ell \to \infty} \frac{\mathcal{E}^{1, n-1}_{\mathrm{rel}} \brk{v, \intvc{0}{\ell}^{n}}}{\ell^{n}}= \infty\text{.}
\end{equation*}
While \(\manifold{N}_{0}\) is a compact simplicial complex, it is not yet a manifold. 
In order to fix this last issue, we define explicitly a manifold containing \(\manifold{N}_{0}\) and retracting to \(\manifold{N}_{0}\). (Alternatively, this could be done in a more abstract fashion, constructing a simplicial embedding of \(\manifold{N}_{0}\) in a Euclidean space, giving it a tubular neighborhood, and endowing the boundary of the latter with a smooth structure.)
As we already explained earlier in this introduction, there is some freedom in the choice of the manifold we define, leading to several possible targets exhibiting an obstruction.
In particular, obstructions arise when the target is \( \Tset^{n}\,\#\, \Tset^{n} \), or more generally a connected sum of this with another compact \( n \)-dimensional Riemannian manifold; see Remark~\ref{rmk:target_connected_sum}.

\medskip 

Concerning Theorem~\ref{theorem_counterexample_hopf}, a quite delicate part of Bethuel’s proof is the construction of mappings having a prescribed number of singularities of the same degree, starting from the spaghetton map on \(\Sset^3\) and then defining a map on \(\Rset^4\) by a Gordian cut through a suitable deformation.

The presentation is substantially simplified and the transfer  to higher dimension is eased by noting that there is a periodic map \(u \in W^{1, 4n- 1}_{\mathrm{loc}} \brk{\Rset^{4n}, \Sset^{2n}}\) constructed by Whitehead product.
This map can be constructed thanks to periodic smooth maps \(\Rset^{2n} \to \Sset^{2n}\) which are constant on the boundary of unit cubes and have Brouwer degree \(1\) in these cubes, so that the resulting homotopy class is the Whitehead product of the \(\brk{2n}\)-dimensional homotopy classes and has Hopf invariant \(2\).
When \(n = 1\), we have a short construction of a map having the same properties as Bethuel's. 
Having Hopf invariant \(2\) is essential for covering the full range of \(n \in \mathbb{N} \setminus \set{0}\), since maps having Hopf invariant \(1\) only exist when \(n \in \set{1, 2, 4}\) \cite{Adams_1960}.

At a more technical level, this periodic character of the maps we construct and the arrangement of the singularities on a regular grid eliminate the need of relying on the notion of scans introduced by Hardt and Rivière~\citelist{\cite{Hardt_Riviere_2003}\cite{Hardt_Riviere_2008}} to analyze the concentration, and allow using instead more elementary arguments.

The lower estimate on the relaxed energy is performed thanks to a branched transport argument, as in~\cite{Bethuel_2020}.

\subsection{Higher order spaces}

We close this introduction with a short word on extensions of the results in this paper to higher order Sobolev spaces of mappings.
Just as we defined the space of first order Sobolev mappings, we may define the space of \( W^{s,p} \) mappings from \( \manifold{M} \) to \( \manifold{N} \), with \( 1 \le p < \infty\) and \( 0 < s < \infty \) real numbers, as
\[
	W^{s, p} \brk{\manifold{M}, \manifold{N}}
	\defeq
	\set{u \in W^{s, p} \brk{\manifold{M}, \Rset^\nu} \st  \text{\( u\brk{x} \in \manifold{N} \) for a.e.\ \( x \in \manifold{M} \)}}\text{.}
\]
Higher order Sobolev spaces of mappings arise, for instance, in the study of biharmonic maps (\( s = 2 \)) or more generally polyharmonic maps (\( s \geq 2 \) integer); see e.g.~\citelist{\cite{Eells_Sampson_1966}\cite{Urakawa_2019}\cite{Gastel_2016}\cite{DaLio_Riviere_2011}} and the references therein.

The problem of strong approximation has been addressed for the whole range \( 0 < s < \infty \): the counterpart of Bethuel's theorem for \( W^{s,p} \) states that, when \( \pi_{1}\brk{\manifold{M}} \simeq \dotsb \simeq \pi_{\floor{sp - 1}}\brk{\manifold{M}} \simeq \set{0} \), then
\begin{equation} 
\label{eq_higher_order_strong_density_problem}
	H^{s,p}_{S} \brk{\manifold{M}, \manifold{N}} = W^{s,p} \brk{\manifold{M}, \manifold{N}}
\end{equation} 
if and only if \( \pi_{\floor{sp}}\brk{\manifold{N}} \simeq \set{0} \); see~\citelist{\cite{Bousquet_Ponce_VanSchaftingen_2015}\cite{Brezis_Mironescu_2015}\cite{Detaille2023}}.
(In the general case,~\eqref{eq_higher_order_strong_density_problem} holds if and only if there is no obstruction to the extension of a continuous map from a \(\floor{sp}\)-dimensional triangulation of \(\manifold{M}\) into \(\manifold{N}\) to a continuous map from \(\manifold{M}\) into \(\manifold{N}\), as in~\cite{Hang_Lin_2003_II} for \( s = 1 \).)

Similarly, one may consider the weak approximation problem for spaces of \( W^{s,p} \) mappings.
As above, we define the sequential weak closure of smooth maps
\(H^{s, p}_W \brk{\manifold{M}, \manifold{N}}\) as
the set of mappings \(u \in W^{s, p} \brk{\manifold{M}, \manifold{N}}\) for which there is a sequence \(\brk{u_k}_{k \in \Nset}\) in \(C^\infty \brk{\manifold{M}, \manifold{N}}\)
converging almost everywhere to \(u\) and such that \(\brk{\mathcal{E}^{s,p}\brk{u_{k}, \manifold{M}}}_{k \in \Nset}\) is bounded,
where the Sobolev energy is now defined by
\begin{align*}
	\mathcal{E}^{s,p}\brk{w, \manifold{M}}
	&\defeq
	\int_{\manifold{M}} \abs{\Deriv^{s}w}^{p}
	&& \text{if \( s \in \Nset \),}
\intertext{and}
	\mathcal{E}^{s,p}\brk{w, \manifold{M}}
	&\defeq
\iint_{\manifold{M} \times \manifold{M}} \frac{\abs{\Deriv^{\floor{s}}w\brk{x} - \Deriv^{\floor{s}}w\brk{y}}^{p}}{\dist_{\manifold{M}}\brk{x,y}^{\dim{\manifold{M}}+\brk{s-\floor{s}}p}} \dif x \dif y	
	&&
	\text{if \( s \notin \Nset \).}
\end{align*}
One may then ask when
\begin{equation}
\label{eq_cahkoh3me3ohBeeQu7iengeb_bis}
	H^{s, p}_W \brk{\manifold{M}, \manifold{N}} = W^{s, p} \brk{\manifold{M}, \manifold{N}}\text{.}
\end{equation}

Following~\cite{Bethuel_1991}*{Proof of Theorem~3}, \( H^{s,p}_{S}\brk{\manifold{M}, \manifold{N}} = H^{s,p}_{W}\brk{\manifold{M}, \manifold{N}} \) when \( sp \notin \Nset \), as when \( s = 1 \), so that the only interesting case is once again when \( sp \in \Nset \) and strong density fails.

In this setting, some partial results have been obtained for \( s \geq 1 \).
In~\cite{Bousquet_Ponce_VanSchaftingen_2013}, Haj\l asz's theorem was extended to the full range \( s \geq 1 \), proving that, if \( sp \in \Nset \) and \( \pi_{1}\brk{\manifold{N}} \simeq \dotsb \simeq \pi_{sp-1}\brk{\manifold{N}} \simeq \set{0} \), then the weak approximation property~\eqref{eq_cahkoh3me3ohBeeQu7iengeb_bis} holds.
In~\cite{Hardt_Riviere_2015}, Hardt and Rivière proved that \( H^{2,2}_{W}\brk{\Bset^{5}, \Sset^{3}} = W^{2,2}\brk{\Bset^{5}, \Sset^{3}} \).
Although their work is restricted to second order Sobolev maps and the target \( \manifold{N} = \Sset^{3} \) for technical reasons, the key feature in their analysis is the fact that the relevant homotopy group \( \pi_{4}\brk{\Sset^{3}} \simeq \Zset/2\Zset \) is finite.

To the best of our knowledge, no analytical obstruction to the weak approximation was known until now for \( s \neq 1 \).
In Section~\ref{section_higher_order}, we prove that both our main results admit a natural counterpart for any \( s \geq 1 \).
Namely, we prove the following theorems.

\begin{theorem}
	\label{theorem_counterexample_higher_order}
	For every \( 1 \leq s < \infty \) and \( 1 \leq p < \infty \) such that \( sp \in \Nset \setminus \set{0,1} \), there exists a manifold \(\manifold{N}\)
	such that if \(\dim \manifold{M} > sp\), then 
	\[
	H^{s, p}_W \brk{\manifold{M}, \manifold{N}}
	\subsetneq W^{s, p} \brk{\manifold{M}, \manifold{N}}\text{.}
	\]
\end{theorem}

It will appear in the proof of \cref{theorem_counterexample_higher_order} that the target manifold \( \manifold{N} \) only depends on the number \( sp  \ge 2 \).

Again we even have
	\[
	H^{s, p}_W \brk{\manifold{M}, \manifold{N}}
	\subsetneq
	\smash{\overline{H^{s, p}_W \brk{\manifold{M}, \manifold{N}}}}
	\subseteq
	W^{s, p} \brk{\manifold{M}, \manifold{N}}\text{,}
	\]
with \(H^{s, p}_W \brk{\manifold{M}, \manifold{N}}\) meagre in \(\smash{\overline{H^{s, p}_W \brk{\manifold{M}, \manifold{N}}}}\).

\begin{theorem}
	\label{theorem_counterexample_hopf_higher_order}
	For every \(n \in \Nset \setminus \set{0}\), if \(\dim \manifold{M} > 4 n-1 \), for every \( 1 \leq s < \infty \) and \( 1 \leq p < \infty \) such that \( sp = 4n-1 \), then 
	\[
	H^{s, p}_W \brk{\manifold{M}, \Sset^{2n}}
	\subsetneq W^{s, p} \brk{\manifold{M}, \Sset^{2n}}\text{.}
	\]
\end{theorem}

As before,
	\[
	H^{s, p}_W \brk{\manifold{M}, \Sset^{2n}}
	\subsetneq
	\smash{\overline{H^{s, p}_W \brk{\manifold{M}, \Sset^{2n}}}}
	\subseteq
	W^{s, p} \brk{\manifold{M}, \Sset^{2n}}\text{,}
	\]
with \(H^{s, p}_W \brk{\manifold{M}, \Sset^{2n}}\) meagre in \(\smash{\overline{H^{s, p}_W \brk{\manifold{M}, \Sset^{2n}}}}\).

Our results are restricted to \( s \geq 1 \) since our proof relies on the Gagliardo--Nirenberg inequality~\eqref{eq_tuLooTadaoya9Ceilie2ieta} to use the fact that
\(H^{s, p}_W \brk{\manifold{M}, \manifold{N}}
\subseteq H^{1, sp}_W \brk{\manifold{M}, \manifold{N}}\).
Since this procedure is not available for \( 0 < s < 1 \), the study of this case is not a byproduct of our method, although several of the ideas developed in this work could be useful.

\section{Bubbling of sequences of Sobolev mappings}
\label{section_bubbling}
In this section, we take as a target space a set \( \manifold{N} \subset \Rset^{\nu} \) which has the property of being a \emph{uniform Lipschitz neighborhood retract}, that is, there exists \( \iota > 0 \) such that there is a Lipschitz-continuous retraction \( \Pi_{\manifold{N}} \colon \manifold{N} + B_{\iota} \to \manifold{N} \).
This includes the particular case where \( \manifold{N} \) is an embedded smooth compact manifold, but also more general cell complexes.
For instance, one may take \( \manifold{N} \) to be  a skeleton of \( \Rset^{\nu} \), which will be crucial for us in the sequel.

The goal of this section is to prove the following \emph{bubbling proposition}.

\begin{proposition}
\label{proposition_bubble_decomposition}
Let \(\manifold{M}\) be a compact manifold and let  \(p = \dim \manifold{M}\).
For every \(\varepsilon \in \intvo{0}{\infty}\) and \(M \in \intvo{0}{\infty}\), there exists \( \delta \in \intvo{0}{\infty} \) such that, given \(\rho_0 \in \intvo{0}{\delta}\) and
\(u \), \( v \in W^{1, p}\brk{\manifold{M}, \manifold{N}} \cap C\brk{\manifold{M}, \manifold{N}}\) satisfying
\begin{gather}
\label{eq_dei4ahxeehei4Uigoov6Soo1}
 \int_{\manifold{M}} \abs{\Deriv u}^p \le M\text{,}\\
\label{eq:ControlGradientv}
 \int_{\manifold{M}} \abs{\Deriv v}^p \le M\text{,}\\
 \label{eq_vohzohehie2aeWa0Ie8vuesh}
 \sup_{a \in \manifold{M}} \int_{B_{\rho_0} \brk{a}} \abs{\Deriv u}^p \le \delta^p\text{,}\\
\intertext{and}
\label{eq_ek3eith6shahp2eeleiweeSh}
  \int_{\manifold{M}} \abs{u - v} \le \rho_0^p \delta\text{,}
\end{gather}
there exist
\(w \in W^{1, p}\brk{\manifold{M}, \manifold{N}} \cap C\brk{\manifold{M}, \manifold{N}}\), \( J \in \Nset\), 
\( a_1, \dotsc, a_J \in \manifold{M} \),
\( \rho_1, \dotsc, \rho_J \in \intvo{0}{\rho_0} \), 
and \( b_1, \dotsc, b_J \in \manifold{N}\) such that
\begin{enumerate}[label=(\roman*)]
\item
\label{item:BallsDisjoint}
the balls \(B_{\rho_j} \brk{a_j}\) are disjoint,
\item
\label{item:Radii}
\(\sum_{j = 1}^J \rho_j \le \rho_0\),
\item
\label{item:ImageAnnulusw} \(b_j \in u\brk{\partial B_{\rho_j}\brk{a_j}}\),
\item
\label{item:BoundaryConditionw}
\(w\restr{\partial B_{\rho_j/2} \brk{a_j}} = b_j\),
\item
\label{item:Locationbj}
\(w \brk{B_{\rho_j} \brk{a_j} \setminus B_{\rho_j/4} \brk{a_j}}\subseteq B_{\varepsilon} \brk{b_j}\),
\item
\label{item:Outside}
\(w = v\) on \(\manifold{M} \setminus \bigcup_{j = 1}^J B_{\rho_j} \brk{a_j}\),
\item
\label{item:annuli}
\(\displaystyle \sum_{j = 1}^J \int_{B_{\rho_j} \brk{a_j} \setminus B_{\rho_j/4} \brk{a_j}} \abs{\Deriv w}^p \le \varepsilon^p\),
\item
  \label{item:wDilationv}
  for every \(x \in B_{\rho_j/4} \brk{a_j}\),
  \(w \brk{x} = v \brk{4x}\) in exponential coordinates,
\item
\label{item:v_homotopic_w}
\(w\) is homotopic to \(v\),
\item
  \label{item:u_homotopic_w0}
  \(w_0\) is homotopic to \(u\), where
  \begin{equation}
  \label{eq:defw0}
  w_0 \defeq
  \begin{cases}
    w & \text{in \(\manifold{M} \setminus \bigcup_{j = 1}^J B_{\rho_j/2} \brk{a_j}\),}\\
    b_j & \text{in \(B_{\rho_j/2} \brk{a_j}\).}
  \end{cases}
  \end{equation}

\end{enumerate}
\end{proposition}

\begin{figure}
\includegraphics{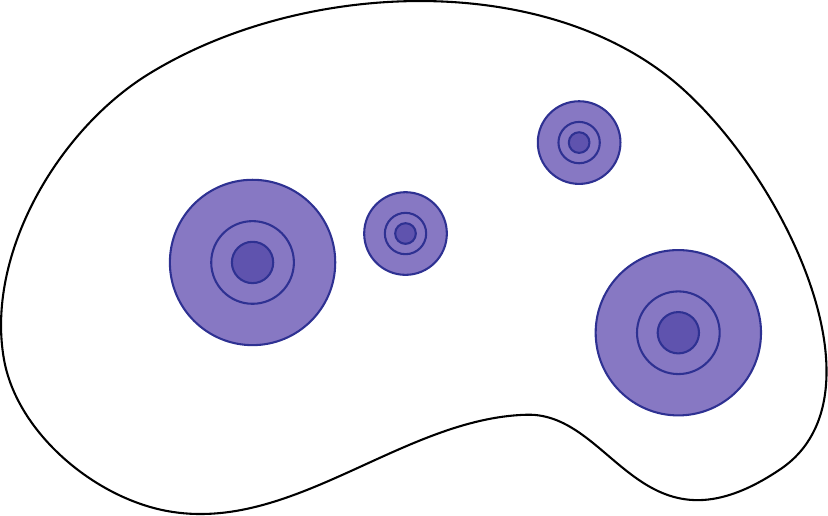}
\caption{Given \(u\) and \(v\), the map \(w\) is constructed so that it coincides with \(v\) outside the larger balls,
it is constant on the intermediate sphere and it is a rescaling of \(v\) on the smaller balls;
the map \(w_0\) defined similarly outside the intermediate balls and constant inside those is homotopic to \( u \).}
\end{figure}

Results similar to~\cref{proposition_bubble_decomposition}
are well established for sequences of harmonic maps~\citelist{\cite{Sacks_Uhlenbeck_1981}*{Proposition~4.3 \& Theorem~4.4}\cite{Nakauchi_Takakuwa_1995}*{Theorem~2}}.
Also in works related to density questions in Sobolev spaces of mappings, bubbling constructions have proved their usefulness in numerous places.
Examples include, but not only,~\citelist{\cite{Hardt_Riviere_2008}*{Proposition~3.4}\cite{Bethuel_2020}*{Remark~1}}, and also the study of the relaxed energy, see for instance~\cite{Giaquinta_Modica_Soucek_1998_II}*{Theorem~3.1.5.1}.
The estimate~\eqref{eq:MorreySobolevu} in the proof generalizes the Courant--Lebesgue lemma~\cite{Jost_1984}*{Lemma~3.1}.
The main conceptual difference between most of the above quoted results and our proposition is that, most often, bubbling statements are formulated as limiting results,
stating that, if a sequence \( \brk{u_{k}}_{k \in \Nset} \) converges weakly in \( W^{1,p} \) to a limiting map \( u \), then the gradients \( \abs{\Deriv{u}_{k}}^{p} \) converge weakly as measures to \( \abs{\Deriv{u}}^{p} \) plus a weighted sum of Dirac masses, which account for the formation and concentration of bubbles around isolated points.

Our statement in Proposition~\ref{proposition_bubble_decomposition}, on the contrary, is concerned with two \emph{fixed} maps, stating that, if they are taken to be sufficiently close and with suitable control on their energy, then they are homotopic to each other, upon removing a finite number of bubbles which are located in small balls.
This quite sharp --- and inevitably more complex --- statement will be needed to locate precisely enough the bubbles when applying Proposition~\ref{proposition_bubble_decomposition}.

\medskip

Before we move to the technical ingredients required in the proof of Proposition~\ref{proposition_bubble_decomposition} and then to the proof itself, we give an informal sketch of the argument, in order to allow the reader to have in mind the main lines of the proof.

The first step is to choose a finite collection of disjoint balls that contains the region of \( \manifold{M} \) where the energy of \( v \) is concentrated.
Letting these balls grow exponentially and merge for a well-chosen time, we may further take them so that the energy of \( u \) and \( v \) as well as the integral distance between \( u \) and \( v \) are controlled on their boundary.
Thanks to the Morrey--Sobolev embedding \( W^{1,p}\brk{\partial\Bset^{p}} \hookrightarrow C^{0}\brk{\partial\Bset^{p}} \), \( u \) and \( v \) both take their values on a common small ball on the boundary of these enlarged balls.

The map \( w \) is then constructed from \( v \) as follows.
Outside of the selected balls \( B_{\rho_j}\brk{a_j} \), the map \( v \) is left unchanged.
The values of \( v \) inside the balls are then shrunk by a linear change of variable in exponential coordinates to make them fit into the four times smaller balls \( B_{\rho_j/4}\brk{a_j} \).
Finally, since \( v \) takes values in a small ball \( B_{\varepsilon}\brk{b_j} \subseteq \manifold{N} \) on \( \partial B_{\rho_j}\brk{a_j} \), we may use an interpolation between the values of \( v \) and \( b_j \) plus a reprojection procedure to fill in the annulus \( B_{\rho_j}\brk{a_j} \setminus B_{\rho_j/4}\brk{a_j} \) with values in \( B_{\varepsilon}\brk{b_j} \) so that \( w \) is constantly equal to \( b_j \) on \(\partial B_{\rho_j/2}\brk{a_j} \).
This already shows assertions~\ref{item:BallsDisjoint} to~\ref{item:v_homotopic_w} in Proposition~\ref{proposition_bubble_decomposition}.

The proof that \( u \) is homotopic to the map \( w_0 \) defined through~\eqref{eq:defw0} relies essentially on a VMO (vanishing mean oscillation) criterion for homotopy.
Indeed, since the balls \( B_{\rho_j}\brk{a_j} \) have been chosen to contain the region of energy concentration of \( v \), the map \( w_0 \) can be proved to have small energy on all balls of radius \( \rho_0 \).
On the other hand, the map \( u \) satisfies the same property by assumption.
This combined with the integral estimate of the distance between \( u \) and \( v \) allows us to apply the VMO criterion for homotopy to deduce that \( u \) and \( w_0 \) are homotopic as continuous maps.

The following proposition will allow us to perform the growing of balls procedure mentioned at the beginning of the above sketch.

\begin{proposition}[Growing balls]
\label{proposition_growing_balls}
Given a Riemannian manifold \(\manifold{M}\) and given balls \(B_{\rho_1} \brk{a_1}, \dotsc, B_{\rho_J}\brk{a_J} \subset \manifold{M}\),
there exist balls 
\begin{equation*}
B_{\rho_1\brk{t}} \brk{a_1\brk{t}},
\dotsc,
B_{\rho_{J\brk{t}}\brk{t}}\brk{a_{J\brk{t}}\brk{t}}
\end{equation*}
such that for every \(t \in \intvr{0}{\infty}\),
\begin{enumerate}[label=(\roman*)]
  \item 
  \label{it_eeb1iejoo1zuingaeth8gi0W}
  the balls \(\Bar{B}_{\rho_1\brk{t}} \brk{a_1\brk{t}}, \dotsc, \Bar{B}_{\rho_{J\brk{t}}\brk{t}}\brk{a_{J\brk{t}}\brk{t}}\) are disjoint,
 \item 
 \label{it_eetheibi0benaecuK6siRe0u}
 \(\displaystyle
  \bigcup_{j = 1}^J B_{\rho_j} \brk{a_j}
  \subseteq \bigcup_{j = 1}^{J\brk{t}} B_{\rho_j\brk{t}} \brk{a_j\brk{t}}\text{,}
 \)
 \item 
 \label{it_deetheiZenoroogae5cahqu3}
 \(\displaystyle \sum_{j = 1}^{J\brk{t}} \rho_j\brk{t} \le e^{t} \sum_{j = 1}^J \rho_j\)\text{,}
  \item
 \label{it_growing_balls_integral_formula}
 if \(f \colon \manifold{M} \to \intvc{0}{\infty}\) is Borel-measurable, then 
 \begin{equation*}
   \int_0^\infty \brk[\bigg]{\;\sum_{j = 1}^{J\brk{t}} \rho_j \brk{t}
   \int_{\partial B_{\rho_j \brk{t}} \brk{a_j \brk{t}}} f} \dif t 
   \le \int_{\manifold{M}}f \text{.}
 \end{equation*}
\end{enumerate}
\end{proposition}

Even though there is no upper bound on \(t\) appearing in the statement of Proposition~\ref{proposition_growing_balls}, 
if \( \manifold{M} \) is bounded, 
then for \(t\) sufficiently large, one will inevitably have \( \partial B_{\rho_j \brk{t}} \brk{a_j} = \emptyset \), and the statement will give no information about such \(t\).
Also, the conclusions are the most useful when \( \rho_j \brk{t} \) is sufficiently small to have \( \partial B_{\rho_j \brk{t}} \brk{a_j \brk{t}} \) diffeomorphic to a sphere.

\begin{proof}[Proof of Proposition~\ref{proposition_growing_balls}]
The proof follows the Euclidean case~\cite{Sandier_Serfaty_2007}*{Theorem 4.2} (see also~\citelist{\cite{Sandier_1998}\cite{Jerrard_1999}}).
Roughly speaking, one defines for \(t \in \intvo{0}{T_1}\)
\begin{align*}
 a_j \brk{t} &= a_j&
 &\text{and}&
 \rho_j \brk{t}= e^t \rho_j\text{,}
\end{align*}
in such a way that the closed balls \(\Bar{B}_{\rho_1\brk{t}} \brk{a_1\brk{t}}, \dotsc, \Bar{B}_{\rho_{J\brk{t}}\brk{t}}\brk{a_{J\brk{t}}\brk{t}}\) are disjoint for \(t < T_1\) and not for \(t = T_1\).
The assertions~\ref{it_eeb1iejoo1zuingaeth8gi0W},~\ref{it_eetheibi0benaecuK6siRe0u}, and~\ref{it_deetheiZenoroogae5cahqu3} are immediate.
For~\ref{it_growing_balls_integral_formula}, we have by the coarea formula and a change of variable 
\[
\begin{split}
 \int_{B_{\rho_j\brk{T_1}}\brk{a_j} \setminus B_{\rho_j \brk{0}}\brk{a_j}}
 f
 &= \int_{\rho_j}^{\rho_j e^{T_1}} \brk[\bigg]{\int_{\partial B_{r} \brk{a_j}} f } \dif r \\
 &= \int_{0}^{T_1} \rho_j \brk{t} \brk[\bigg]{\int_{\partial B_{\rho_j \brk{t}} \brk{a_j}} f } \dif t\text{.}
\end{split}
\]
To continue the construction,
we apply Lemma~\ref{lemma_ball_merging} sufficiently many times to the collection \(\Bar{B}_{\rho_1\brk{t}} \brk{a_1\brk{t}}, \dotsc, \Bar{B}_{\rho_{J\brk{t}}\brk{t}}\brk{a_{J\brk{t}}\brk{t}}\) to get a disjoint collection.
Repeating this procedure at most \(J\) times, we get the required collection of families of balls.
\end{proof}

\begin{lemma}[Merging balls]
\label{lemma_ball_merging}
If \(\manifold{M}\) is a complete Riemannian manifold, and if 
\(\Bar{B}_{\rho_0} \brk{a_0} \cap \Bar{B}_{\rho_1} \brk{a_1} \ne \emptyset\), then there exists a ball \(B_\rho \brk{a}\) such that 
\[
 B_\rho \brk{a} \supseteq B_{\rho_0} \brk{a_0} \cup B_{\rho_1} \brk{a_1}
\]
and 
\[
 \rho \le \rho_0 + \rho_1\text{.}
\]
\end{lemma}
\begin{proof}
This is again a classical argument (see for example \cite{Sandier_Serfaty_2007}*{Lemma 4.1}).
If \(B_{\rho_0} \brk{a_0} \subseteq B_{\rho_1} \brk{a_1}\) or \(B_{\rho_1} \brk{a_1} \subseteq B_{\rho_0} \brk{a_0}\) we can take \(a=a_1\) and \(\rho = \rho_1\), or \(a = a_0\) and \(\rho = \rho_0\).
Otherwise, we take a point \(a\) on a minimising geodesic from \(a_0\) to \(a_1\) such that 
\(
 \rho_0 + d \brk{a_0, a} = 
 \rho_1 + d \brk{a_1, a}
\)
and 
\(
 \rho = \brk{\rho_0 + d \brk{a_0, a_1} + \rho_1}/{2}
\).
\end{proof}

As already mentioned, we will also use a VMO criterion for homotopy, whose statement is as follows.

\begin{proposition}
\label{proposition_small_homotopy}
There exist \(\theta \in \intvo{0}{\infty}\) and \(\rho_* \in \intvo{0}{\infty} \) such that,
if \(u_0 \), \( u_1 \in C \brk{\manifold{M}, \manifold{N}} \cap W^{1, p} \brk{\manifold{M}, \manifold{N}}\) and \(\rho \in \intvo{0}{\rho_*}\) satisfy the condition that for every \(a \in \manifold{M}\),
\begin{equation*}
\frac{1}{\rho^p} \int_{B_\rho\brk{a}} \abs{u_0 - u_1}\le \theta
\end{equation*}
and 
\begin{equation*}
  \int_{B_{\rho} \brk{a}} \abs{\Deriv u_0}^p + \abs{\Deriv u_1}^p \le \theta^p\text{,}
\end{equation*}
then \(u_0\) and \(u_1\) are homotopic.
\end{proposition}

Even though the statement of Proposition~\ref{proposition_small_homotopy} might not be found under this exact form in the literature, it relies on classical arguments that go back to the work of Schoen and Uhlenbeck~\cite{Schoen_Uhlenbeck_1983}, and Brezis and Nirenberg~\cite{Brezis_Nirenberg_1995}.

\begin{proof}[Proof of Proposition~\ref{proposition_small_homotopy}]
Define 
\begin{equation*}
  u_j^r \brk{x} \defeq \fint_{B_r \brk{x}} u_j\text{.}
\end{equation*}
Since \(\manifold{M}\) is a Riemannian manifold, its injectivity radius is positive; we assume that it is \(2 \rho_*\) with \(\rho_* \in \intvo{0}{\infty}\); in particular, for every \(a \in \manifold{M}\), the exponential map at \(a\) is uniformly controlled on 
\(B_{\rho_*} \brk{a}\).
Therefore, by the Poincaré--Wirtinger inequality, for every \( 0 < r \leq \rho < \rho_* \), we have
\begin{equation*}
 \dist \brk{u_j^r \brk{x}, \manifold{N}}
 \le \C \fint_{B_r \brk{x}} \fint_{B_r \brk{x}} \abs{u_j \brk{y} - u_j \brk{z}} \dif y \dif z
 \le  \C \brk[\bigg]{\int_{B_{r} \brk{a}} \abs{\Deriv u_j}^p}^\frac{1}{p}
 \le \C \theta\text{,}
\end{equation*}
whereas by the triangle inequality
\begin{equation*}
  \abs{u_0^\rho \brk{x} - u_1^\rho \brk{x}}
  \le \frac{\C}{\rho^p} \int_{B_\rho\brk{x}} \abs{u_1 - u_0}
  \leq
  \C\theta\text{.}
\end{equation*}
Hence, if \(\theta\) was chosen sufficiently small, we get the required homotopy.
Indeed, one first goes from \( u_{0} \) to \( u_{0}^{\rho} \) via the \( u_{0}^{r} \), then from \( u_{0}^{\rho} \) to \( u_{1}^{\rho} \) via linear interpolation plus reprojection, and finally from \( u_{1}^{\rho} \) to \( u_{1} \) via the \( u_{1}^{r} \).
More precisely, one may e.g.\ use \( H \colon \intvc{0}{1} \times \manifold{M} \to \manifold{N} \) defined by
\[
	H\brk{t,x} = \begin{cases}
		\Pi_{\manifold{N}} \brk{u_{0}^{3t\rho}\brk{x}} & \text{if \( 0 \leq t \leq 1/3 \),} \\
		\Pi_{\manifold{N}}\brk{\brk{3t-1}u_{1}^{\rho}\brk{x} + \brk{2-3t}u_{0}^{\rho}\brk{x}} & \text{if \( 1/3 \leq t \leq 2/3 \),} \\
		\Pi_{\manifold{N}}\brk{u_{1}^{3\brk{1 - t}\rho}\brk{x}} & \text{if \( 2/3 \leq t \leq 1 \),}
	\end{cases}
\]
where \( \Pi_{\manifold{N}} \) is the Lipschitz-continuous retraction onto \( \manifold{N} \).
\resetconstant
\end{proof}

\begin{proof}[Proof of~\cref{proposition_bubble_decomposition}]
Since Proposition~\ref{proposition_bubble_decomposition} depends on quite a few parameters --- either part of the statement, or specific to the proof --- that may depend on each other, we adopt the following convention.
All constraints for the various parameters involved in the proof will be displayed inside boxes.
Then, at the end of the proof, we give the exact relations that explain how and in which order to choose those parameters so that they satisfy the required constraints.

We again assume that the injectivity radius of \( \manifold{M} \) is \(2 \rho_*\) with \(\rho_* \in \intvo{0}{\infty}\), which implies a uniform control on the exponential map on any ball 
\(B_{\rho_*} \brk{a}\), and we assume that
\[
  \boxed{\rho_0 \le \rho_*\text{.}}
\]
	
Given \(\eta \in \intvo{0}{\infty}\) and \(\rho \in \intvo{0}{\infty}\) to be chosen later on,
we consider the set 
\begin{equation}
\label{eq_cha5piephae2nee8geiZaeng}
  A \defeq 
  \set[\Big]{a \in \manifold{M} \st \int_{B_{\rho} \brk{a}} \abs{\Deriv v}^p \ge \eta^p}\text{,}
\end{equation}
and a maximal subset \(A_* \subseteq A\) such that the balls
\(\set{B_\rho \brk{a}}_{a \in A_*}\) are disjoint.
In particular, we have 
by disjointness and by~\eqref{eq:ControlGradientv}
\begin{equation}
\label{eq_ingaix2Oojeifohleixae1AB}
 \# A_* \le \frac{1}{\eta^p} \sum_{a \in A_*}
 \int_{B_\rho \brk{a}} \abs{\Deriv v}^p
 \le \frac{1}{\eta^p} \int_{\manifold{M}} \abs{\Deriv v}^p
 \le \frac{M}{\eta^p}\text{.}
\end{equation}
On the other hand, by maximality of \(A_*\),
\(
 A \subseteq \bigcup_{a \in A_*} B_{2\rho} \brk{a}
\),
so that for every \(a \in \manifold{M} \setminus \bigcup_{a' \in A_*} B_{2\rho} \brk{a'}\), \(a \not \in A\). (If \(A_* = \emptyset\), the following arguments remain valid albeit unnecessarily complicated in this trivial case.)

\begin{figure}
\includegraphics{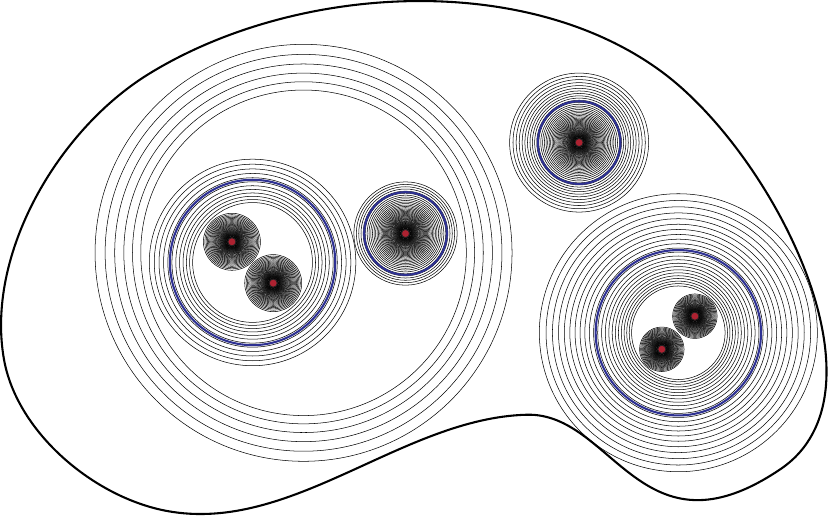}
\caption{Thanks to an averaging argument, the spheres \( \partial B_{\rho_j} \brk{a_j} \) (in blue) can be chosen out of the growing balls generated by the balls \( B_{\rho} \brk{a} \) (in red)
so that \( u \) and \( v \) have a small Sobolev energy and are at small average distance on them.}
\end{figure}

Applying Proposition~\ref{proposition_growing_balls} to \(\brk{B_{2 \rho} \brk{a}}_{a \in A_*}\),
we get a finite collection
\begin{equation*}
B_{\rho_1\brk{t}} \brk{a_1\brk{t}}, \dotsc, B_{\rho_{J\brk{t}}\brk{t}}\brk{a_{J\brk{t}}\brk{t}}
\end{equation*}
for \(t \in \intvo{0}{\infty}\);
by Proposition~\ref{proposition_growing_balls}~\ref{it_growing_balls_integral_formula}, it satisfies for every \( T_{\ast} \in \intvo{0}{\infty} \)
\begin{gather*}
\int_0^{T_*} \brk[\bigg]{\sum_{j = 1}^{J\brk{t}} \rho_j \brk{t}
   \int_{\partial B_{\rho_j \brk{t}} \brk{a_j \brk{t}}} \abs{\Deriv u}^p}
   \dif t \le\int_{\manifold{M}} \abs{\Deriv u}^p \le M\text{,}\\
  \int_0^{T_*}
  \brk[\bigg]{\sum_{j = 1}^{J\brk{t}} \rho_j \brk{t}
   \int_{\partial B_{\rho_j \brk{t}} \brk{a_j \brk{t}}} \abs{\Deriv v}^p}
   \dif t \le \int_{\manifold{M}} \abs{\Deriv v}^p \le  M\text{,}\\
\intertext{and}
  \int_0^{T_*} \brk[\bigg]{\sum_{j = 1}^{J\brk{t}} \rho_j \brk{t}
   \int_{\partial B_{\rho_j \brk{t}} \brk{a_j \brk{t}}} \abs{u - v}}
   \dif t \le \int_{\manifold{M}} \abs{u - v} \le  \rho_0^p \delta\text{,}
\end{gather*}
in view of~\eqref{eq_dei4ahxeehei4Uigoov6Soo1},~\eqref{eq:ControlGradientv}, and~\eqref{eq_ek3eith6shahp2eeleiweeSh}.
There exists thus \(t_* \in \intvo{0}{T_*}\) such that, if we set \(J \defeq J\brk{t_*}\), and for \(j \in \set{1, \dotsc, J}\),
\(
  a_j \defeq a_j \brk{t_*}\) and \(\rho_j \defeq \rho_j \brk{t_*}\),
then we have
\begin{gather}
\label{eq:ControluBoundary}
 \sum_{j = 1}^{J} \rho_j
   \int_{\partial B_{\rho_j} \brk{a_j}} \abs{\Deriv u}^p
   \le  \frac{3 M}{T_*}\text{,} \\
\label{eq:ControlvBoundary}
 \sum_{j = 1}^{J} \rho_j
   \int_{\partial B_{\rho_j} \brk{a_j}} \abs{\Deriv v}^p
   \le \frac{3 M}{T_*}\text{,}\\
\intertext{and}
\label{eq:ControluMinusvBoundary}
	\sum_{j = 1}^{J} \rho_j
	\int_{\partial B_{\rho_j} \brk{a_j}} \abs{u-v}
	\leq  \frac{3 \rho_0^p \delta}{T_*}\text{.}
\end{gather}

We observe now that 
\begin{equation}
\label{eq_poi9OoZaisheuJ8ahbie5yee}
A \subseteq
\bigcup_{a \in A_*} B_{2\rho} \brk{a}\subseteq 
  \bigcup_{j = 1}^J B_{\rho_j} \brk{a_j}
\end{equation}
and, by~\eqref{eq_ingaix2Oojeifohleixae1AB},
\begin{equation}
\label{eq_OophieG0oekoorohzai3rooj}
 \max_{j \in \set{1, \dotsc, J}} \rho_j \le  \sum_{j = 1}^{J} \rho_j
 \le \frac{2 e^{T_*} \rho M}{\eta^p} \le \rho_0 \le \rho_* \text{,}
\end{equation}
so that \ref{item:Radii} holds,
provided 
\begin{equation}
\label{eq_gei4Iba2phoheiN9siiFai0u}
  \boxed{2 \rho M e^{T_*} \le \rho_0 \eta^p\text{.}}
\end{equation}

By the Morrey--Sobolev embedding, we have for every \(j \in \set{1, \dotsc, J}\) and \(x\), \(y \in \partial B_{\rho_j} \brk{a_j}\), in view of~\eqref{eq:ControluBoundary},~\eqref{eq_OophieG0oekoorohzai3rooj}, and by the choice of \( \rho_* \) in terms of the injectivity radius,
\begin{equation}
	\label{eq:MorreySobolevu}
  \abs{u \brk{x} - u \brk{y}}
  \le \Cl{cst_eiPh4toh4ze6buN4dae8nohd} \brk[\bigg]{\int_{\partial B_{\rho_j}\brk{a_j}} \abs{\Deriv u}^p}^{\frac{1}{p}}
  \rho_j^\frac{1}{p}
  \le \frac{\Cl{cst_UaF3thoh5eengoubaiyeephe} M^{1/p}}{T_*^{1/p}}\text{.}
\end{equation}
We also have for every \(j \in \set{1, \dotsc, J}\) and every \(x\), \(y \in \partial B_{\rho_j} \brk{a_j}\), by~\eqref{eq:MorreySobolevu},~\eqref{eq:ControlvBoundary},~\eqref{eq:ControluMinusvBoundary}, by~\eqref{eq_OophieG0oekoorohzai3rooj} and the conditions on \(\rho_*\), since \( \rho \leq \rho_j \),
\begin{equation}
\label{eq:distutov}
\begin{split}
  &\abs{u \brk{x} - v \brk{y}}\\
  &\quad \le \sup_{z \in \partial B_{\rho_j} \brk{a_j}} \abs{u \brk{x} - u \brk{z}}
  + \inf_{z \in \partial B_{\rho_j} \brk{a_j}}
  \abs{u \brk{z} - v \brk{z}}
  + \sup_{z \in \partial B_{\rho_j} \brk{a_j}} \abs{v \brk{z} - v \brk{y}}
  \\
  &\quad \le \Cr{cst_eiPh4toh4ze6buN4dae8nohd} \brk[\bigg]{\int_{\partial B_{\rho_j}\brk{a_j}} \abs{\Deriv u}^p}^{\frac{1}{p}}
  \rho_j^\frac{1}{p}
  + \frac{\C}{\rho_j^{p - 1}} \int_{\partial  B_{\rho_j} \brk{a_j}} \abs{u - v}
  + \Cr{cst_eiPh4toh4ze6buN4dae8nohd} \brk[\bigg]{\int_{\partial B_{\rho_j}\brk{a_j}} \abs{\Deriv v}^p}^{\frac{1}{p}}
  \rho_j^\frac{1}{p}\\
  &\quad \le \frac{2 \Cr{cst_UaF3thoh5eengoubaiyeephe} M^{1/p}}{T_*^{1/p}} + \frac{\Cl{cst_BaesokiekooxohGaimai1Ush} \rho_0^p \delta }{\rho^p T_*}\text{.}
\end{split}
\end{equation}

Using the fact that \( \manifold{N} \) is a uniform Lipschitz neighborhood retract, we fix \(\tau \in \intvo{0}{\varepsilon}\) such that for every \(b \in \manifold{N}\), \(\Pi_{\manifold{N}}\) is well-defined on \(\Bar{B}_\tau \brk{b}\),
and such that for every \(z \in \Bar{B}_{\tau} \brk{b}\) and \(t \in \intvc{0}{1}\),
\[
  \abs{b - \Pi_{\manifold{N}} \brk{\brk{1 - t} z + t b}} \le \varepsilon\text{.}
\]
If
\begin{equation}
\label{eq_keHocai7vooyae4Cibipa1zo}
 \boxed{
 \frac{2 \Cr{cst_UaF3thoh5eengoubaiyeephe} M^{1/p}}{T_*^{1/p}}  + \frac{\Cr{cst_BaesokiekooxohGaimai1Ush} \rho_0^p \delta }{\rho^p T_*}
 \le \tau
  \text{,}
 }
\end{equation}
taking \(b_j \in u \brk{\partial B_{\rho_j} \brk{a_j}}\), for every \(x \in \partial B_{\rho_j} \brk{a_j}\),
we have in view of~\eqref{eq:MorreySobolevu} and~\eqref{eq:distutov},
\begin{align*}
 \abs{u \brk{x} - b_j} &\le \tau&
 &\text{and} &
 \abs{v \brk{x} - b_j} & \le \tau
 \text{.}
\end{align*}
We can thus define now, in exponential coordinates, by~\eqref{eq_OophieG0oekoorohzai3rooj} and the choice of \(\rho_*\),
\begin{equation*}
 w \brk{x}
 \defeq
 \begin{cases}
  v \brk{x} & \text{if \(x \in \manifold{M} \setminus \bigcup_{j = 1}^J B_{\rho_j} \brk{a_j}\),}\\
    \Pi_{\manifold{N}} \brk{\brk{2\frac{\abs{x}}{\rho_j} - 1} v \brk{\tfrac{\rho_j}{\abs{x}} x} + \brk{2 - 2 \frac{\abs{x}}{\rho_j}} b_j} &\text{if \(x \in B_{\rho_j}\brk{a_j} \setminus B_{\rho_j/2}\brk{a_j}\),}\\
    \Pi_{\manifold{N}} \brk{\brk{2 - 4\frac{\abs{x}}{\rho_j}} v \brk{\frac{\rho_j}{\abs{x}} x} + \brk{ 4 \frac{\abs{x}}{\rho_j} - 1} b_j} &\text{if \(x \in B_{\rho_j/2}\brk{a_j} \setminus B_{\rho_j/4} \brk{a_j}\),}\\
    v \brk{4x} &\text{if \(x \in B_{\rho_j/4}\brk{a_j}\).}
 \end{cases}
\end{equation*}

Since \(v\) is continuous, we have
\(
\operatorname{tr}_{\partial B_{\rho_j} \brk{a_j}} v
= v \vert_{\partial B_{\rho_j} \brk{a_j}}
\),
so that \(w \in W^{1, p} \brk{\manifold{M}, \manifold{N}}\).
Properties~\ref{item:ImageAnnulusw},~\ref{item:BoundaryConditionw},~\ref{item:Locationbj}, \ref{item:Outside}, and~\ref{item:wDilationv} clearly hold by construction of \(w\).
Moreover, the map defined in exponential coordinates as 
\[
 H \brk{t, x}
\defeq
\begin{cases}
   v \brk{x} & \text{if \(x \in \manifold{M} \setminus \bigcup_{j = 1}^J B_{\rho_j} \brk{a_j}\),}\\
   \Pi_{\manifold{N}} \brk[\big]{\brk{2\frac{\abs{x}}{\rho_j} - 1}v \brk{\tfrac{\rho_j}{\abs{x}} x} + \brk{2 - 2 \frac{\abs{x}}{\rho_j}} b_j} &\text{if \(x \in B_{\rho_j}\brk{a_j} \setminus B_{\brk{1 - t/2}\rho_j}\brk{a_j}\),}\\
   \begin{aligned}
   \Pi_{\manifold{N}} \brk[\big]{\brk{5 - 3t - 4\tfrac{\abs{x}}{\rho_j}}v \brk{\tfrac{\rho_j}{\abs{x}} x}\\
   + \brk{4 \tfrac{\abs{x}}{\rho_j} + 3t - 4} b_j}
   \end{aligned}
   &\text{if \(x \in B_{\brk{1 - t/2}\rho_j}\brk{a_j} \setminus B_{\brk{1 - 3t/4}\rho_j}\brk{a_j}\),}\\
   v \brk{x/\brk{1 - 3t/4}} &\text{if \(x \in B_{\brk{1 - 3t/4}\rho_j}\brk{a_j}\),}
\end{cases}
\]
can be checked to be a homotopy between the maps \(w\) and \(v\), so that~\ref{item:v_homotopic_w} holds.

Via an integration of~\eqref{eq:distutov}, we have 
for every \(j \in \set{1, \dotsc, J}\),
\[
\begin{split}
\int_{\partial B_{\rho_j} \brk{a_j}} \abs{v - b_j}^p
 &\le \C \rho_j^p \int_{\partial B_{\rho_j}\brk{a_j}} \abs{\Deriv u}^p
+ \C \rho_j^{p - 1}  \brk[\bigg]{\frac{1}{\rho_j^{p - 1}} \int_{\partial B_{\rho_j} \brk{a_j}} \abs{u - v}}^p\\
&\qquad\qquad + \C \rho_j^p \int_{\partial B_{\rho_j} \brk{a_j}} \abs{\Deriv v}^p\text{,}
 \end{split}
\]
and thus, by the conditions on \(\rho_*\) and \(\rho \le \rho_j\)
\begin{equation*}
\begin{split}
	\int_{B_{\rho_j} \brk{a_j} \setminus B_{\rho_j/4} \brk{a_j}} \abs{\Deriv w}^p
    & \le \C\rho_{j} \int_{\partial B_{\rho_j} \brk{a_j}} \abs{\Deriv v}^p + \frac{\C}{\rho_j^{p - 1}} \int_{\partial B_{\rho_j} \brk{a_j}} \abs{v - b_j}^p\\
	&\le
	\C
	\rho_j
	\int_{\partial B_{\rho_j } \brk{a_j}} \abs{\Deriv v}^p
	+ \Cl{cst_middle_eelopeheezo4vai9iQuoop8y}\brk[\bigg]{\frac{\rho_j}{\rho^p} \int_{\partial B_{\rho_j} \brk{a_j}} \abs{u - v}}^p\\
	&\qquad +
	 \C \rho_j \int_{\partial B_{\rho_j} \brk{a_j}} \abs{\Deriv u}^p\text{.}
\end{split}
\end{equation*}
Therefore, in view of~\eqref{eq:ControluBoundary},~\eqref{eq:ControlvBoundary},
and~\eqref{eq:ControluMinusvBoundary}, we find
\begin{equation}
\label{eq_och4coopahd6ii2Suish6Eim}
 \sum_{j = 1}^J
 \int_{B_{\rho_j} \brk{a_j} \setminus B_{\rho_j/4} \brk{a_j}} \abs{\Deriv w}^p
 \le \frac{\Cl{cst_oL2aitaiH6Ri1thie8eV7oTh} M}{T_*}
 + \Cl{cst_eiGua5ieQuu0opahzahghaep} \brk[\Big]{\frac{\rho_0^p \delta}{\rho^pT_*}}^p\text{.}
\end{equation}
In particular, if
\begin{equation}
\label{eq_Oejei6Phathuigh5ahw2Iet6}
  \boxed{\frac{\Cr{cst_oL2aitaiH6Ri1thie8eV7oTh} M}{T_*} + \Cr{cst_eiGua5ieQuu0opahzahghaep}
 \brk[\Big]{\frac{\rho_0^p \delta}{\rho^pT_*}}^p
 \le \varepsilon^p\text{,}}
\end{equation}
then \ref{item:annuli} follows from \eqref{eq_och4coopahd6ii2Suish6Eim}.

We now prove that \( u \) and \( w_{0} \) are homotopic, relying on Proposition~\ref{proposition_small_homotopy}.
Since \(w_0 = w\) on \( B_{\rho_j} \brk{a_j} \setminus B_{\rho_j/4} \brk{a_j} \) and \( w_0 = b_{j} \) on \( B_{\rho_j/4} \brk{a_j} \),
we have by \eqref{eq_och4coopahd6ii2Suish6Eim}
\begin{equation}
\label{eq_uFohghie2Saeneen2Noh1oot}
 \sum_{j = 1}^J
 \int_{B_{\rho_j} \brk{a_j}} \abs{\Deriv w_0}^p
\le
 \sum_{j = 1}^J
 \int_{B_{\rho_j} \brk{a_j} \setminus B_{\rho_j/4} \brk{a_j}} \abs{\Deriv w}^p
 \le \frac{\Cr{cst_oL2aitaiH6Ri1thie8eV7oTh} M}{T_*}
 + \Cr{cst_eiGua5ieQuu0opahzahghaep} \brk[\Big]{\frac{\rho_0^p \delta}{\rho^pT_*}}^p\text{.}
\end{equation}
Given \(a \in \manifold{M}\), we have, either
\(B_{\rho/2} \brk{a} \subseteq \bigcup_{j = 1}^J B_{\rho_j} \brk{a_j}\), and thus, by \eqref{eq_uFohghie2Saeneen2Noh1oot},
\[
 \int_{B_{\rho/2} \brk{a}} \abs{\Deriv w_0}^p
 \le \frac{\Cr{cst_oL2aitaiH6Ri1thie8eV7oTh} M}{T_*}
 + \Cr{cst_eiGua5ieQuu0opahzahghaep} \brk[\Big]{\frac{\rho_0^p \delta}{\rho^pT_*}}^p\text{,}
\]
or there exists \(a' \in \manifold{M} \setminus \bigcup_{j = 1}^J B_{\rho_j} \brk{a_j}\)
such that \(B_{\rho/2} \brk{a} \subseteq B_{\rho} \brk{a'}\), and thus, by \eqref{eq_uFohghie2Saeneen2Noh1oot} again
\[
\begin{split}
  \int_{B_{\rho/2} \brk{a}} \abs{\Deriv w_0}^p
  &\le
   \int_{B_{\rho} \brk{a'} \setminus \bigcup_{j = 1}^J B_{\rho_j} \brk{a_j}} \abs{\Deriv v}^p
   +  \sum_{j = 1}^J
 \int_{B_{\rho_j} \brk{a_j}} \abs{\Deriv w_0}^p\\
 &\le \eta^p + \frac{\Cr{cst_oL2aitaiH6Ri1thie8eV7oTh} M}{T_*}  + \Cr{cst_eiGua5ieQuu0opahzahghaep}
  \brk[\Big]{\frac{\rho_0^p \delta}{\rho^pT_*}}^p
\end{split}
\]
in view of~\eqref{eq_poi9OoZaisheuJ8ahbie5yee},~\eqref{eq_cha5piephae2nee8geiZaeng} and~\eqref{eq_och4coopahd6ii2Suish6Eim}.
In particular, if
\begin{equation}
\label{eq_neeFahj1eFu6chee3aiGhaeHbis}
 \boxed{\eta^p + \frac{\Cr{cst_oL2aitaiH6Ri1thie8eV7oTh} M}{T_*} + \Cr{cst_eiGua5ieQuu0opahzahghaep}
 \brk[\Big]{\frac{\rho_0^p \delta}{\rho^pT_*}}^p
 \le \theta^p\text{,}}
\end{equation}
then
\[
 \int_{B_{\rho/2} \brk{a}} \abs{\Deriv w_0}^p \le \theta^p\text{.}
\]

In order to verify the assumptions of Proposition~\ref{proposition_small_homotopy}, we actually need to work with a modified version of \( u \), constructed from \( u \) by a similar process to the one used to construct \( w_0 \) from \( v \).
More specifically, given \(\sigma_1, \dotsc, \sigma_J \) such that \(0 < \sigma_j < \rho_j/4\), we define the map \(u_0 \colon \manifold{M} \to \manifold{N}\) by
\begin{equation*}
 u_0 \brk{x}
 \defeq
 \begin{cases}
  u \brk{x} & \text{if \(x \in \manifold{M} \setminus \bigcup_{j = 1}^J B_{\rho_j} \brk{a_j}\),}\\
    \Pi_{\manifold{N}} \brk[\big]{\brk{2\tfrac{\abs{x}}{\rho_j} - 1} u \brk{\tfrac{\rho_j}{\abs{x}} x} + \brk[\big]{2 - 2 \frac{\abs{x}}{\rho_j}} b_j} &\text{if \(x \in B_{\rho_j}\brk{a_j} \setminus B_{\rho_j/2}\brk{a_j}\),}\\
    b_j & \text{if \(x \in B_{\rho_j/2} \brk{a_j}\setminus B_{2 \sigma_j} \brk{a_j}\),}\\
    \Pi_{\manifold{N}} \brk[\big]{\brk{2 - \frac{\abs{x}}{\sigma_j}} u \brk{\tfrac{\rho_j}{\abs{x}} x} + \brk[\big]{ \frac{\abs{x}}{\sigma_j} - 1} b_j} &\text{if \(x \in B_{2\sigma_j}\brk{a_j} \setminus B_{\sigma_j} \brk{a_j}\),}\\
    u \brk{\tfrac{\rho_j}{\sigma_j}x} &\text{if \(x \in B_{\sigma_j}\brk{a_j}\).}
 \end{cases}
\end{equation*}
We observe that the map \(u\) is homotopic to \(u_0\).
Therefore, it suffices to show that \( w_0 \) is homotopic to \( u_0 \).

As we did for \( v \) above, we have for every \(j \in \set{1, \dotsc, J}\) that, by~\eqref{eq:MorreySobolevu},
\[
 \int_{\partial B_{\rho_j} \brk{a_j}} \abs{u \brk{x} - b_j}^p
 \le \C \rho_j^p \int_{\partial B_{\rho_j} \brk{a_j}} \abs{\Deriv u}^p\text{,}
\]
and thus, by~\eqref{eq_OophieG0oekoorohzai3rooj} and the conditions on \(\rho_*\),
\[
 \int_{B_{\rho_j} \brk{a_j} \setminus B_{\sigma_j} \brk{a_j}} \abs{\Deriv u_0}^p
 \le \C \rho_j \int_{\partial B_{\rho_j} \brk{a_j}} \abs{\Deriv u}^p\text{.}
\]
Therefore, by~\eqref{eq:ControluBoundary},
\[
 \sum_{j = 1}^J
 \int_{B_{\rho_j} \brk{a_j} \setminus B_{\sigma_j} \brk{a_j}} \abs{\Deriv u_0}^p
 \le \frac{\Cl{cst_shoo4Keenohn6caifoo7ieta} M}{T_*}\text{.}
\]
We also have, by~\eqref{eq_OophieG0oekoorohzai3rooj}, the conditions on \(\rho_*\), and by the assumption \eqref{eq_vohzohehie2aeWa0Ie8vuesh},
\[
 \int_{B_{\sigma_j} \brk{a_j}} \abs{\Deriv u_0}^p
 \le \Cl{cst_Zoh1mooxoyoom6chahFohH4o} \int_{B_{\rho_j} \brk{a_j}} \abs{\Deriv u}^p\le \Cr{cst_Zoh1mooxoyoom6chahFohH4o} \delta^p \text{,}
\]
with \(\Cr{cst_Zoh1mooxoyoom6chahFohH4o} \ge 1\).
(In the case where \( \manifold{M} \) is flat, one can take \( \Cr{cst_Zoh1mooxoyoom6chahFohH4o} = 1 \);
the constant comes from the bound on the geometry of \( \manifold{M} \) on scales below \(\rho_*\).)

Given \(a \in \manifold{M}\), we consider two cases.
If \(B_{\rho/2} \brk{a} \subseteq B_{\rho_j} \brk{a_j}\), then
\[
  \int_{B_{\rho/2} \brk{a}} \abs{\Deriv u_0}^p
  \le \frac{\Cr{cst_shoo4Keenohn6caifoo7ieta} M}{T_*}
  + \Cr{cst_Zoh1mooxoyoom6chahFohH4o} \delta^p\text{,}
\]
provided \eqref{eq_gei4Iba2phoheiN9siiFai0u} holds. 
Otherwise, since \(\rho \le \rho_j/2\) and \(\sigma_j \le \rho_j/4\),
\(B_{\rho/2} \brk{a} \cap B_{\sigma_j} \brk{a_j} = \emptyset\), and thus by~\eqref{eq_vohzohehie2aeWa0Ie8vuesh},
\[
   \int_{B_{\rho/2} \brk{a}} \abs{\Deriv u_0}^p
   \le  \frac{\Cr{cst_shoo4Keenohn6caifoo7ieta} M}{T_*} +
   \delta^p\text{,}
\]
provided \(2 \rho \le \rho_0\),
which follows from the condition~\eqref{eq_gei4Iba2phoheiN9siiFai0u},
assuming without loss of generality that 
\begin{equation}
\label{eq_Na2Fie7eish1Echee3ce0too}
\boxed{
\eta^{p} \leq M \text{.}} 
\end{equation}
If
\begin{equation}
\label{eq_neeFahj1eFu6chee3aiGhaeH}
  \boxed{\frac{\Cr{cst_shoo4Keenohn6caifoo7ieta} M}{T_*} +
   \Cr{cst_Zoh1mooxoyoom6chahFohH4o} \delta^p \le \theta^p\text{,}}
\end{equation}
we then have
\begin{equation*}
  \int_{B_{\rho/2} \brk{a}} \abs{\Deriv u_0}^p
  \le \theta^p\text{.}
\end{equation*}

Finally, since \(u_0 = u\) and \( w_0 = v \) outside of the balls \( B_{\rho_j}\brk{a_j} \), we find that
\begin{equation*}
\begin{split}
 \int_{B_{\rho/2} \brk{a}}
 \abs{u_0 - w_0}
 &\le
 \int_{\manifold{M}} \abs{u - v}
 + \Cl{cst_ez3Eeshae3foaphiavai4IZu} \rho^p \varepsilon
 + \sum_{j = 1}^J \Cl{cst_oos0iene8fov3joog7zoo8Lo} \sigma_j^p\\
 &\le \rho_0^p \delta + \Cr{cst_ez3Eeshae3foaphiavai4IZu}\rho^p \varepsilon
 + \Cr{cst_oos0iene8fov3joog7zoo8Lo}
 \sum_{j = 1}^J \sigma_j^p\text{.}
\end{split}
\end{equation*}
We note that here \( \Cr{cst_oos0iene8fov3joog7zoo8Lo} \) depends on \( u \) and \( v \) through their \( L^{\infty} \) norm.
If
\[
 \rho_0^p \delta + \Cr{cst_ez3Eeshae3foaphiavai4IZu}\rho^p \varepsilon
 + \Cr{cst_oos0iene8fov3joog7zoo8Lo}
 \sum_{j = 1}^J \sigma_j^p \le \frac{\rho^p \theta}{2^p} \text{,}
\]
we have
\begin{equation}
\label{eq_wai0zehechuho7och2Zeefah}
\int_{B_{\rho/2} \brk{a}}
 \abs{u_0 - w_0} \le \frac{\rho^p \theta}{2^p}\text{.}
\end{equation}
If
\begin{equation}
\label{eq_ooc7phoo4ohneeng4Ioz3ied}
\boxed{
  \frac{\rho_0^p}{\rho^p} \delta + \Cr{cst_ez3Eeshae3foaphiavai4IZu} \varepsilon \le \frac{\theta}{2^{p + 1}}\text{,}}
\end{equation}
then \(\sigma_1, \dotsc, \sigma_J\) can be chosen sufficiently small, depending on \(u\) and \(v\), so that~\eqref{eq_wai0zehechuho7och2Zeefah} holds.
By Proposition~\ref{proposition_small_homotopy},
\(w_0\) is homotopic to \(u_0\) and thus to \(u\), proving~\ref{item:u_homotopic_w0}.

It remains to show that \(\eta\), \(\rho\), and \(T_*\) satisfying the conditions~\eqref{eq_gei4Iba2phoheiN9siiFai0u}, \eqref{eq_keHocai7vooyae4Cibipa1zo}, \eqref{eq_Oejei6Phathuigh5ahw2Iet6}, \eqref{eq_neeFahj1eFu6chee3aiGhaeHbis}, \eqref{eq_Na2Fie7eish1Echee3ce0too},~\eqref{eq_neeFahj1eFu6chee3aiGhaeH}, and ~\eqref{eq_ooc7phoo4ohneeng4Ioz3ied} can be found.
We first assume without loss of generality that
\begin{align*}
 \varepsilon &\le \frac{ \theta}{2^{p + 2} \Cr{cst_ez3Eeshae3foaphiavai4IZu}}&
 &\text{ and} &
 \rho_0 &\le \rho_*\text{.}
\end{align*}
We next choose successively
\begin{gather*}
\begin{aligned}
 \eta &\defeq \min \brk[\Big]{M^{1/p}, \frac{\theta}{3^{1/p}}, \frac{\varepsilon}{2^{1/p}}} \text{,}&
 T_* &\defeq
 \max \brk[\Big]{
 \frac{\brk{4\Cr{cst_UaF3thoh5eengoubaiyeephe}}^{p}M}{\tau^{p}}, 
 \frac{\Cr{cst_oL2aitaiH6Ri1thie8eV7oTh} M}{\eta^p},
 \frac{2\Cr{cst_shoo4Keenohn6caifoo7ieta} M}{\theta^p}}\text{,}&
\end{aligned}\\
\begin{aligned}
  \lambda &\defeq \frac{\eta^p}{2M  e^{T_*}}\text{,}&
  \rho &\defeq \lambda \rho_0\text{,}&
\delta
&\defeq
\min 
  \brk[\Big]{ 
    \frac{T_* \lambda^p \tau}{2 \Cr{cst_BaesokiekooxohGaimai1Ush}},
  \frac{T_* \lambda^p \eta}{\brk{\Cr{cst_eiGua5ieQuu0opahzahghaep}}^{1/p}},
  \frac{\theta}{\brk{2 \Cr{cst_Zoh1mooxoyoom6chahFohH4o}}^{1/p}}, \frac{\theta\lambda^p}{2^{p + 2}}}\text{,}
\end{aligned}
\end{gather*}
and check that all the conditions are satisfied.
\resetconstant
\end{proof}

\section{Analytical obstruction for integer exponent}
\label{section_analytical_obstruction}

This section is devoted to the proof of Theorem~\ref{theorem_main}.
We start by explaining some tools that will be crucial to construct a suitable family of Sobolev mappings with values into a skeleton of \( \Rset^{n} \) and prove that their relaxed energy grows superlinearly with respect to their Sobolev energy.
We then explain the procedure to transfer these constructions, first to a compact skeleton, and then to a compact manifold without boundary, yielding the proof of Theorem~\ref{theorem_main} via the nonlinear uniform boundedness principle.

\subsection{Conical joint estimate on the Brouwer degree}
\label{subseq:Conical_Degree}

Our goal in this section is to establish an analytic estimate of the joint Brouwer degrees of a map \( f \in C^{1}\brk{\Sset^{n-1}, \Rset^{n}} \) with respect to a finite number of points in \( \Rset^{n} \) that are avoided by \( f \).

We first recall that Brouwer's topological degree of a map \(f \in C^1 \brk{\Sset^{n-1}, \Sset^{n-1}}\) 
can be computed by the formula
\begin{equation}
\label{eq_degree_formula}
\deg \brk{f}
=
\frac{\displaystyle
  \int_{\Sset^{n-1}}
  \brk{\det \Deriv f}\, w \circ f }
  {\displaystyle
  \int_{\Sset^{n-1}} w}\text{,}
\end{equation}
for any weight function \(w \in C \brk{\Sset^{n-1}, \Rset}\) with \(\int_{\Sset^{n-1}} w \ne 0\).
(The determinant in~\eqref{eq_degree_formula} is computed on the tangent space with the orientation induced by the canonical orientation on the ambient space \(\Rset^{n}\).)

Actually, formula~\eqref{eq_degree_formula} is still valid for computing the degree of a continuous map which is not \( C^{1} \), but merely \( W^{1,n-1} \), see \citelist{\cite{Brezis_Coron_1983}\cite{Brezis_Nirenberg_1995}}.
Therefore, from now on, we will work with continuous and \( W^{1,n-1} \) maps --- this notably avoids some technical issues when working with smooth maps between cell complexes.
For every open set \(G \subset \Sset^{n-1}\), taking \(w = w_{k} \), where \( \brk{w_{k}}_{k \in \Nset} \) is a sequence approximating the characteristic function of \(G\), and letting \( k \to \infty \), we get the estimate
\begin{equation}
	\label{eq:estimate_degree_sphere}
 \abs{\deg \brk{f}}
 \le
 \frac{1}{\mathcal{H}^{n-1} \brk{G}}
 \int_{f^{-1} \brk{G}} \abs{\Deriv f}^{n-1}\text{.}
\end{equation}

Given \(f \in W^{1,n-1} \brk{\Sset^{n-1}, \Rset^{n}\setminus \set{0}} \cap C \brk{\Sset^{n-1}, \Rset^{n}\setminus \set{0}} \), we have
\begin{equation}
	\label{eq:degre_sphere_vs_origin}
 \deg f \defeq \deg \brk{f/\abs{f}}\text{.}
\end{equation}
We say that a set \(\manifold{C} \subseteq \Rset^{n}\) is a cone whenever, for every \(t \in \intvo{0}{\infty}\) and \(x \in \manifold{C}\), \(tx \in \manifold{C}\).
If \(\manifold{C}\subseteq \Rset^{n}\) is an open cone, then
it follows from~\eqref{eq:estimate_degree_sphere} and~\eqref{eq:degre_sphere_vs_origin} that
\begin{equation}
	\label{eq:estimate_degree_origin}
  \abs{\deg \brk{f}}
  \le
  \frac{1}{\mathcal{H}^{n-1} \brk{\manifold{C} \cap \Sset^{n-1}}}
  \int_{f^{-1} \brk{\manifold{C}}} \frac{\abs{\Deriv f}^{n-1}}{\abs{f}^{n-1}}\text{,}
\end{equation}
since \(
 \abs{\Deriv \brk{f/\abs{f}}} \le \abs{\Deriv f}/\abs{f}\) everywhere on \(\Sset^{n - 1}\).

We will use the degree with respect to a point \(\sigma \in \Rset^{n}\), defined for \(f \in C \brk{\Sset^{n-1}, \Rset^{n} \setminus \set{\sigma}}\) as
\begin{equation*}
 \deg_\sigma \brk{f} \defeq \deg \brk{f - \sigma}\text{.}
\end{equation*}
We are now going to give a joint estimate on the degrees with respect to a finite set of points \(\Sigma \subseteq \Rset^{n}\).
We start by observing that~\eqref{eq:estimate_degree_origin} implies that
\begin{equation}
	\label{eq:joint_estimate_with_sum}
\begin{split}
 \sum_{\sigma \in \Sigma} \abs{ \deg_\sigma \brk{f}}
  &\le
 \frac{1}{\mathcal{H}^{n-1} \brk{\manifold{C} \cap \Sset^{n-1}}}
 \sum_{\sigma \in \Sigma}\int_{f^{-1} \brk{\manifold{C} + \sigma}} \frac{\abs{\Deriv f}^{n-1}}{\abs{f - \sigma}^{n-1}}\\
 &\le
 \frac{1}{\mathcal{H}^{n-1} \brk{\manifold{C} \cap \Sset^{n-1}}}
 \brk[\Big]{\int_{f^{-1} \brk{\manifold{C} + \Sigma}} \abs{\Deriv f}^{n-1}}
  \sup_{x \in \Sset^{n-1}}
 \sum_{\sigma \in \Sigma}
 \frac{1}{\abs{f \brk{x} - \sigma}^{n-1}}\text{.}
\end{split}
\end{equation}
We now estimate the sum that appears on the last line of~\eqref{eq:joint_estimate_with_sum}.

\begin{proposition}
	\label{proposition_rearrangement_inequality}
If \(\Sigma \subseteq \Zset^{n}\), \(y \in \Rset^n\), and if \(\dist \brk{y, \Sigma} \ge 1/2\), then
\begin{equation*}
  \sum_{\sigma \in \Sigma}
 \frac{1}{\abs{y - \sigma}^{n-1}}
 \le C \brk{\# \Sigma}^{\frac{1}{n}}\text{,}
\end{equation*}
where the constant \( C > 0 \) depends only on \( n \).
\end{proposition}
\begin{proof}
We have
\begin{equation*}
\begin{split}
 \sum_{\sigma \in \Sigma}
 \frac{1}{\abs{y - \sigma}^{n-1}}
  &= \int_0^\infty \# \brk{\Sigma \cap B_r \brk{y}} \frac{n-1}{r^{n}}\dif r\\
  &\le \Cl{cst_AeFii7kae9phohg2naeruiYi} \int_{0}^\infty \min \brk{\# \Sigma, r^{n}} \frac{1}{r^{n}} \dif r\\
  &= \Cr{cst_AeFii7kae9phohg2naeruiYi} 
  \brk[\bigg]{\int_{0}^{\brk{\# \Sigma}^{1/n}} \dif r +
  \int_{\brk{\# \Sigma}^{1/n}}^\infty \frac{\# \Sigma}{r^{n}} \dif r}\\
  &= \Cr{cst_AeFii7kae9phohg2naeruiYi}\brk[\big]{1 + \tfrac{1}{n-1}} \brk{\# \Sigma}^{\frac{1}{n}}\text{.}
\end{split}
\end{equation*}
In the first inequality above, we have used the fact that
\begin{equation*}
 \# \brk{\Sigma \cap B_r \brk{y}}
 \le \mathcal{L}^{n} \brk{B_{r + \sqrt{n}/2} \brk{y}}
 = \mathcal{L}^{n} \brk{B_1} \brk{r + \sqrt{n}/2}^{n}
 \le \C r^{n}
\end{equation*}
provided \( r \ge 1/2 \), whereas \( \# \brk{\Sigma \cap B_r \brk{y}} = 0 \) when \(r < 1/2\).
\resetconstant
\end{proof}

Injecting Proposition~\ref{proposition_rearrangement_inequality} in~\eqref{eq:joint_estimate_with_sum}, we obtain the following conical joint estimate on the degrees.

\begin{proposition}
\label{proposition_conical_joint_degree_estimate}
Given a finite set \(\Sigma \subseteq \Zset^{n}\),  an open cone \( \manifold{C} \subset \Rset^{n} \), and \( f \in W^{1,n-1} \brk{\Sset^{n-1}, \Rset^{n} \setminus \brk{B_{1/2} + \Sigma}} \cap C\brk{\Sset^{n-1}, \Rset^{n} \setminus \brk{B_{1/2} + \Sigma}} \), one has 
\begin{equation*}
\brk[\bigg]{
\sum_{\sigma \in \Sigma} \abs{ \deg_\sigma \brk{f}}}^{1- \frac{1}{n}}
\le
\frac{C}{\mathcal{H}^{n-1} \brk{\manifold{C} \cap \Sset^{n-1}}}
\int_{f^{-1} \brk{\manifold{C} + \Sigma}} \abs{\Deriv f}^{n-1}\text{,}
\end{equation*}
where the constant \( C > 0 \) depends only on \( n \).
\end{proposition}
\begin{proof}
Without loss of generality, we assume that for every \(\sigma \in \Sigma\), 
\(\deg_\sigma \brk{f} \ne 0\).
We then have by \eqref{eq:joint_estimate_with_sum} and Proposition~\ref{proposition_rearrangement_inequality},
\begin{equation*}
\begin{split}
 \sum_{\sigma \in \Sigma} \abs{ \deg_\sigma \brk{f}}
&\le
\frac{\Cl{cst:Prop32} \brk{\# \Sigma}^\frac{1}{n}}{\mathcal{H}^{n-1} \brk{\manifold{C} \cap \Sset^{n-1}}}
\int_{f^{-1} \brk{\manifold{C} + \Sigma}} \abs{\Deriv f}^{n-1} \\
&\le \frac{\Cr{cst:Prop32} }{\mathcal{H}^{n-1}
\brk{\manifold{C} \cap \Sset^{n-1}}} \brk[\bigg]{
\sum_{\sigma \in \Sigma} \abs{ \deg_\sigma \brk{f}}}^{\frac{1}{n}}
\int_{f^{-1} \brk{\manifold{C} + \Sigma}} \abs{\Deriv f}^{n-1}\text{,}
\end{split}
\end{equation*}
and the conclusion follows.
\resetconstant
\end{proof}

\subsection{Lower bounds on energies on spheres}

In this section, we combine our conical joint degrees estimate with our bubbling construction to give a lower estimate on the Sobolev energy of approximating sequences in different homotopy classes.

We first get a lower bound in a general setting in terms of joint degree differences between two continuous Sobolev maps.

\begin{proposition}
\label{proposition_sphere_joint_degree_bubbled}
Given a finite set \(\Sigma \subseteq \Zset^{n}\), open cones \(\manifold{C}_1, \dotsc, \manifold{C}_I \subseteq \Rset^{n}\),
and open sets \(G_1, \dotsc, G_{I} \subseteq \Rset^{n}\)
such that for every \(i \in \set{1,\dotsc, I}\),
\(
  G_i \cap \brk{\manifold{C}_i + \Sigma} = \emptyset
\),
there exists a constant \(C\) such that, given \(u \in W^{1, n-1} \brk{\Sset^{n-1}, \Rset^{n} \setminus \brk{B_{1/2}+\Sigma}} \cap C \brk{\Sset^{n-1}, \Rset^{n} \setminus \brk{B_{1/2}+\Sigma}}\) satisfying
\begin{equation*}
  u \brk{\Sset^{n-1}} \subseteq \bigcup_{i = 1}^{I} G_i\text{,}
\end{equation*}
for every \(M \in \intvo{0}{\infty}\), there exists \(\eta \in \intvo{0}{\infty}\) such that,
if \(v \in W^{1, n-1} \brk{\Sset^{n-1}, \Rset^{n} \setminus \brk{B_{1/2}+\Sigma}} \cap C \brk{\Sset^{n-1}, \Rset^{n} \setminus \brk{B_{1/2}+\Sigma}}\),
if
\begin{equation*}
  \int_{\Sset^{n-1}} \abs{\Deriv v}^{n-1} \le M\text{,}
\end{equation*}
and if
\begin{equation*}
 \int_{\Sset^{n-1}} \abs{u - v} \le \eta\text{,}
\end{equation*}
then
\begin{equation*}
 \brk[\bigg]{\sum_{\sigma \in \Sigma}
 \abs{\deg_\sigma \brk{u} - \deg_{\sigma} \brk{v}}}^{1 - \frac{1}{n}}
 \le
 C
 \sum_{i = 1}^{I} \int_{u^{-1} \brk{G_i} \cap v^{-1} \brk{\manifold{C}_i + \Sigma}} \abs{\Deriv v}^{n-1}\text{.}
\end{equation*}
Moreover, the constant \( C > 0 \) depends only on \( n \) and on \( \mathcal{H}^{n-1}\brk{\manifold{C}_{i} \cap \Sset^{n-1}} \).
\end{proposition}
\begin{proof}
We choose \(\varepsilon\) sufficiently small so that
for every \(x \in \Sset^{n-1}\), there exists
\(i \in \set{1, \dotsc, I}\) for which
\(B_{2\varepsilon} \brk{u \brk{x}} \subseteq G_i\) (in other words, \(2\varepsilon\) is the Lebesgue number of the covering of the compact set \(u \brk{\Sset^{n-1}}\) by \(G_1, \dotsc, G_I\)).

Without loss of generality, we may also assume that 
\[
	\int_{\Sset^{n-1}} \abs{\Deriv u}^{n-1}
	\leq
	M\text{.}
\]
We take \(\rho_0\) such that for every \(x \in \Sset^{n-1}\),
\(u \brk{B_{\rho_0} \brk{x}} \subseteq B_{\varepsilon} \brk{u \brk{x}}\).
Invoking Lebesgue's lemma, we may furthermore assume that \( \rho_0 \) has been chosen so that \( \rho_0 \leq \delta \) and
\[
	\sup_{a \in \Sset^{n-1}} \int_{B_{\rho_{0}} \brk{a}} \abs{\Deriv u}^{n-1}
	\leq
	\delta^{n-1}\text{,}
\]
where \( \delta \in \intvo{0}{\infty} \) is given by Proposition~\ref{proposition_bubble_decomposition}.
Finally, we let \( \eta = \rho_0^{n-1}\delta \).
Let us note that \( \rho_{0} \), and therefore \( \eta \), depend on \( u \) via the use of Lebesgue's lemma.
On the other hand, the constants that will appear in the proof, and therefore the final constant \( C \), do not depend on \( u \).

Let \(w\) be given by Proposition~\ref{proposition_bubble_decomposition}.
Defining
\(
  w_j \defeq w \restr{B_{\rho_j/2} \brk{a_j}}
\),
we observe that for every \(\sigma \in \Sigma\),
\begin{equation*}
\begin{split}
 \deg_\sigma \brk{v} - \deg_{\sigma} \brk{u} = \deg_{\sigma} \brk{w_0}
 + \sum_{j = 1}^J \deg_{\sigma} \brk{w_j} - \deg_{\sigma} \brk{u}
 = \sum_{j = 1}^J \deg_{\sigma} \brk{w_j}\text{.}
\end{split}
\end{equation*}
It follows thus by the triangle inequality and sublinearity that
\begin{equation*}
 \brk[\bigg]{
 \sum_{\sigma \in \Sigma} \abs{
  \deg_\sigma \brk{v} - \deg_{\sigma} \brk{u}}}^{1 - \frac{1}{n}}
\le \brk[\bigg]{\sum_{j = 1}^J \sum_{\sigma \in \Sigma} \abs{\deg_{\sigma} \brk{w_j}}}^{1 - \frac{1}{n}} \le \sum_{j = 1}^J
\brk[\bigg]{\sum_{\sigma \in \Sigma} \abs{\deg_{\sigma} \brk{w_j}}}^{1 - \frac{1}{n}}\text{.}
\end{equation*}
By the choice of \(\rho_0\) and of \(\varepsilon\), for every \(j \in \set{1, \dotsc, J}\), there is \(i \in \set{1, \dotsc, I}\) such that
\begin{equation}
	\label{eq:InclusionuBallInGi}
u \brk{B_{\rho_j} \brk{a_j}}
\subseteq
 B_{\varepsilon} \brk{u \brk{a_j}}
 \subseteq G_{i}\text{.}
\end{equation}
Therefore,
\[
  w_j \brk{B_{\rho_j/2} \brk{a_j} \setminus B_{\rho_j/4} \brk{a_j}}
  \subseteq
  B_{2\varepsilon} \brk{u \brk{a_j}}
  \subseteq G_i\text{,}
\]
where we have used assertions~\ref{item:Locationbj} and~\ref{item:ImageAnnulusw} in Proposition~\ref{proposition_bubble_decomposition}.
It follows thus from Proposition~\ref{proposition_conical_joint_degree_estimate} that
\begin{equation*}
\begin{split}
\brk[\Big]{
 \sum_{\sigma \in \Sigma} \abs{\deg_\sigma \brk{w_j}}}^{1 - \frac{1}{n}}
 &\le\Cl{cst:AppPropDegEst} \int_{w_j^{-1} \brk{\manifold{C}_i + \Sigma}} \abs{\Deriv w_j}^{n-1}\\
 &= \Cr{cst:AppPropDegEst} \int_{B_{\rho_j/4} \brk{a_j} \cap w_j^{-1} \brk{\manifold{C}_i + \Sigma}} \abs{\Deriv w_j}^{n-1}\\
 &\le \Cl{cst:AppPropDegEstBis} \int_{B_{\rho_j} \brk{a_j} \cap v^{-1} \brk{\manifold{C}_i + \Sigma}} \abs{\Deriv v}^{n-1}\text{,}
 \end{split}
\end{equation*}
since
\begin{enumerate*}[label=(\(\alph*\))]
\item \( w_j \brk{B_{\rho_j} \brk{a_j} \setminus B_{\rho_j/4} \brk{a_j}} \subseteq G_{i} \), 
\item \( G_i \cap \brk{\manifold{C}_i + \Sigma} = \emptyset \), and 
\item \(v\brk{4x} = w_j\brk{x}\) on \(B_{\rho_j/4} \brk{a_j}\).
\end{enumerate*}
We note that actually, in order to apply Proposition~\ref{proposition_conical_joint_degree_estimate}, one needs to view the map \( w_{j} \), which is originally defined as a map on the disk \( B_{\rho_j/2}\brk{a_j} \) with constant value on the boundary, as a map defined on the sphere \( \Sset^{n-1} \), 
which can be done extending \( w_{j} \) to the exterior of the ball by a constant.
We also note that \( \Cr{cst:AppPropDegEst} \) depends on the measure of \( \manifold{C}_{i} \cap \Sset^{n-1} \) through the use of Proposition~\ref{proposition_conical_joint_degree_estimate}, while \( \Cr{cst:AppPropDegEstBis} \) is a purely geometric constant that depends only on \( n \).

It follows thus, using also~\eqref{eq:InclusionuBallInGi}, that 
\[
\begin{split}
  \sum_{\substack{j \in \set{1, \dotsc, J}\\ B_{\varepsilon} \brk{u\brk{a_j}}\subseteq G_i}
  }\brk[\Big]{
  \sum_{\sigma \in \Sigma} \abs{\deg_\sigma \brk{w_j}}}^{1 -\frac{1}{n}}
  &\le \Cr{cst:AppPropDegEstBis} \sum_{j = 1}^J \int_{u^{-1} \brk{G_i} \cap B_{\rho_j} \brk{a_j} \cap v^{-1} \brk{\manifold{C}_i + \Sigma}} \abs{\Deriv v}^{n-1}\\
  &\le \Cr{cst:AppPropDegEstBis} \int_{u^{-1}\brk{G_{i}} \cap v^{-1} \brk{\manifold{C}_i + \Sigma}} \abs{\Deriv v}^{n-1}\text{.}
\end{split}
\]
The conclusion follows by summing the above estimate over all \( i \in \set{1,\dotsc,I} \).
\resetconstant
\end{proof}

\begin{figure}
\includegraphics{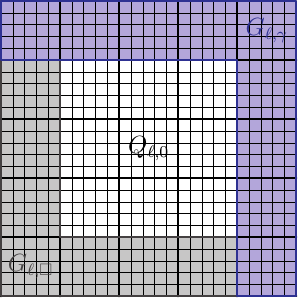}
\caption{The sets \(Q_{\ell, 0}\), \(G_{\ell, \square}\), and \(G_{\ell, \gamma}\) for \(\gamma = \brk{-1, -1}\) for \(\ell = 5\) and \(q = 2\).
The colored cubes on the boundary form the set \( G_{\ell,\square} \), and the cubes that are further colored in blue form the set \( G_{\ell,\gamma} \) for \( \gamma = \brk{-1,-1} \).} 
\label{fig:Glgamma}
\end{figure}
In order to get a lower estimate on sequences of maps on spheres related to our counterexample,
we will choose some specific sets \(\mathcal{C}_i\) and \(G_i\) in Proposition~\ref{proposition_sphere_joint_degree_bubbled}.
For \( \ell \in \Nset \), define the cube
\begin{equation*}
  Q_{\ell}\defeq \intvc{0}{\ell}^{n}\text{.}
\end{equation*}
Given \(q \in \Nset\), for every \(\alpha \in \mathbf{A} \defeq \set{-q, -\brk{q - 1}, \dotsc, q}^{n}\), we set
\begin{equation*}
 Q_{\ell, \alpha}
 \defeq Q_{\ell} + \ell \alpha + \brk{q\ell, \dotsc, q\ell}\text{,}
\end{equation*}
so that
\begin{equation}
\label{eq_nuijaiDaofoKahbiebee0aix}
 Q_{\brk{2q+1}\ell} = \bigcup_{\alpha \in \mathbf{A}} Q_{\ell, \alpha}\text{,}
\end{equation}
and the interiors of the \(\brk{Q_{\ell, \alpha}}_{\alpha \in \mathbf{A}}\) are mutually disjoint.
We also define, for every \(\gamma \in \mathbf{\Gamma}
  \defeq \set{-1, 1}^{n}\), the cone
\begin{equation*}
 \mathcal{C}_{\gamma}
 \defeq  \set{\brk{x_1, \dotsc, x_{n}} \in \Rset^{n}
 \st \text{for every \(i \in \set{1, \dotsc, n}\), \( \gamma_i x_i > 0 \)}}\text{.}
\end{equation*}
Setting
\begin{align*}
  \mathbf{A}_{\square} &\defeq
  \set{\alpha \in \mathbf{A} \st \max_{1 \le i \le n} \abs{\alpha_i} = q}\text{,}&
 G_{\ell, \square}
 &\defeq \operatorname{int}\;
 \brk[\Big]{\bigcup_{\alpha \in A_{\square}} Q_{\ell, \alpha}}\text{,}\\
  \mathbf{A}_{\gamma} &\defeq
  \set{\alpha \in \mathbf{A} \st \min_{1 \le i \le n} \alpha_i \gamma_i = - q}\text{,}&
 G_{\ell, \gamma}
 &\defeq \operatorname{int}\; \brk[\Big]{
 \bigcup_{\alpha \in A_{\gamma}} Q_{\ell, \alpha}}\text{,}
\end{align*}
and viewing \( Q_{\brk{2q+1}\ell} \) as made of \( \brk{2q+1}^{n} \) cubes of side length \( \ell \) according to~\eqref{eq_nuijaiDaofoKahbiebee0aix}, we observe that  \( \mathbf{A}_{\square} \) is the set of all indices \( \alpha \) such that the cube \( Q_{\ell, \alpha} \) intersects the boundary of \( Q_{\brk{2q+1}\ell} \), while \( G_{\ell, \square} \) is made of the interior of the union of all such cubes.
The set of indices \( \mathbf{A}_{\gamma} \) is a refinement of \( \mathbf{A}_{\square} \), corresponding to cubes that intersect on the boundary of \( Q_{\brk{2q+1}\ell} \) on some specific faces, indicated by the multi-index \( \gamma \) --- more precisely, \( \mathbf{A}_{\gamma} \) selects the cubes on the boundary of \( Q_{\brk{2q+1}\ell} \) that lie opposite to at least one face indicated by \( \gamma \) (see Figure~\ref{fig:Glgamma}).
We observe that
\begin{equation*}
 G_{\ell, \square} = \bigcup_{\gamma \in \mathbf{\Gamma}} G_{\ell, \gamma}\text{.}
\end{equation*}
Indeed, given \(x \in G_{\ell, \square} \cap Q_{\ell, \alpha}\), we have \(x \in G_{\ell, \gamma}\)
with any \(\gamma \in \mathbf{\Gamma}\) such that for every \(i \in \set{1, \dotsc, n}\), \(\alpha_i \gamma_i \ne q\).

We now introduce some notation for the centers of the cubes of our decomposition.
More precisely, we define
\[
\Sigma_{\ell}  \defeq Q_{\ell,0} \cap \brk{\Zset+1/2}^{n}\text{.}
\]
Otherwise stated, if we consider the standard decomposition of \( \Rset^{n} \) into unit cubes, such that the origin is a vertex of a cube, then \( \Sigma_{\ell} \) is the set of all centers of those cubes that lie inside \( Q_{\ell,0} \).
We observe that, for every \( \gamma \in \mathbf{\Gamma} \) and every \( \alpha \in \mathbf{A}_{\gamma} \), it holds that
\begin{equation}
\label{eq_dist_cone_boundary_cubes}
\dist_{\infty} \brk{y, Q_{\ell, \alpha}} \geq \brk{q - 1}\ell
\quad
\text{for every \( y \in \manifold{C}_{\gamma} + \Sigma_{\ell} \).}
\end{equation}
Indeed, by definition, there exists an index \( i \in \set{1, \dotsc, n} \) such that \( \alpha_{i}\gamma_{i} = -q \).
Assume without loss of generality that \( \gamma_{i} = 1 \) and thus \(\alpha_i = -q\).
Then, for every \( y \in \manifold{C}_{\gamma} + \Sigma_{\ell} \) and \( z \in Q_{\ell,\alpha} \), we have
\( y_{i} \geq q \ell + \frac{1}{2} \) while \( z_{i} \leq \ell \), so that
\[
\abs{y-z}_{\infty}
\geq
y_{i}-z_{i}
\geq
\brk{q - 1}\ell\text{,}
\]
which proves our claim.

Finally, we define a retraction from \( \Rset^{n} \) into \( Q_{\ell, \alpha} \).
More specifically, given any  \(\alpha \in \mathbf{A}\) and \(x \in \Rset^{n} \),
we let
\begin{equation*}
 \Theta_{\ell, \alpha} \brk{x}
 \defeq 
 \begin{cases}
   x&\text{if \(x \in Q_{\ell, \alpha}\),}\\
   c_{\ell, \alpha} + \frac{\ell \brk{x - c_{\ell, \alpha}}}{2 \abs{x - c_{\ell, \alpha}}_{\infty}}&\text{otherwise,}
 \end{cases}
\end{equation*}
where \(c_{\ell, \alpha}\) is the center of the cube \(Q_{\ell, \alpha}\).

\begin{lemma}
\label{lemma_Theta}
For every \(y \in \Rset^n\), one has
\[
 \lim_{z \to y} \frac{\abs{\Theta_{n, \alpha} \brk{y} - \Theta_{n, \alpha} \brk{z}}_2}
 {\abs{y - z}_2} \le \frac{\sqrt{n}}{1 + 2 \dist_{\infty}(y, Q_{\ell, \alpha})/\ell}.
\]
\end{lemma}
\begin{proof}
Assuming without loss of generality that \(c_{\ell, \alpha} = 0\), this follows from the fact that if \(y = \brk{y_1, y''}, z=\brk{z_1, z''} \in \Rset^n\) satisfy \(\abs{y}_{\infty} = y_1 \ge \ell/2\) and \(\abs{z}_{\infty} = z_2 \ge \ell/2\), then
\[
\begin{split}
 \abs[\Big]{\Theta \brk{y} - \Theta \brk{z}}
 &= \abs[\Big]{\frac{\ell y}{2 y_1} - \frac{\ell z}{2 z_1}}
 = \frac{\ell}{2}\abs[\Big]{\frac{y''}{y_1} - \frac{z''}{z_1}}\\
 &\le \frac{\ell}{4}\brk[\Big]{\frac{1}{y_1} + \frac{1}{z_1}} \abs{y'' - z''} + \frac{\ell}{4}\abs[\Big]{\frac{1}{y_1} - \frac{1}{z_1}} \abs{y'' + z''}\\
 & \le \frac{\ell}{4}\brk[\Big]{\frac{1}{y_1} + \frac{1}{z_1}} \abs{y'' - z''} + \sqrt{n - 1}\frac{\ell}{4}\abs[\Big]{\frac{1}{y_1} - \frac{1}{z_1}} \brk{y_1 + z_1}\\
 &= \frac{\ell}{4}\brk[\Big]{\frac{1}{y_1} + \frac{1}{z_1}}
 \brk{\abs{y''-z''} + \sqrt{n - 1} \abs{y_1 - z_1}}\\
 &\le \frac{\sqrt{n}}{2} \brk[\Big]{\frac{1}{1 + 2 \dist_{\infty}(y, Q_{n, \alpha})/\ell} + \frac{1}{1 + 2 \dist_{\infty} (z, Q_{n, \alpha})/\ell}}  \abs{y - z}\text{,}
\end{split}
\]
where the inequality \( \abs{y''-z''} + \sqrt{n - 1} \abs{y_1 - z_1} \leq \sqrt{n}\abs{y - z} \) follows by applying the Cauchy--Schwarz inequality to the vectors \( \brk{\abs{y''-z''},\abs{y_1 - z_1}} \) and \( \brk{1,\sqrt{n - 1}} \).
Letting \(z \to y\), we reach the conclusion.
\end{proof}

We can now state and prove the main result of this section, which is a lower bound on the energy gap of approximating sequences that will be the key ingredient in order to prove the required estimate to strengthen inequality~\eqref{eq_OogaiSh3neiy7ahl6eeF7eiV}, according to the strategy described in the introduction.

\begin{proposition}
\label{proposition_bound_degree_gap_by_projection}
Assume that
\begin{equation}
\label{eq_eif7leizee5aeloh1mohSh1i}
  n < \brk{2 q - 1}^2.
\end{equation}
Given \( u \in W^{1,n-1}\brk{\Sset^{n-1}, \Rset^{n} \setminus \brk{B_{1/2}+\Sigma_{\ell}}} \cap C\brk{\Sset^{n-1}, \Rset^{n} \setminus \brk{B_{1/2}+\Sigma_{\ell}}} \) such that
	\[
	u\brk{\Sset^{n-1}}
	\subseteq
	G_{\ell, \square}\text{,}
	\]
	\( M \in \intvo{0}{\infty} \), and \(\brk{E_\alpha}_{\alpha \in \mathbf{A}_{\square}}\) a family of subsets of \(\Sset^{n-1}\) such that for every \(\gamma \in \mathbf{\Gamma}\),
	\begin{equation}
		\label{eq:FamilyEalpha}
		u^{-1} \brk{G_{\ell, \gamma}}
		\subseteq \bigcup_{\alpha \in \mathbf{A}_{\gamma}} E_\alpha\text{,}
	\end{equation}
	there exists \( \eta \in \intvo{0}{\infty} \) such that, for every \( v \in W^{1,n-1}\brk{\Sset^{n-1}, \Rset^{n} \setminus \brk{B_{1/2}+\Sigma_{\ell}}} \cap C\brk{\Sset^{n-1}, \Rset^{n} \setminus \brk{B_{1/2}+\Sigma_{\ell}}} \) satisfying
	\[
	\int_{\Sset^{n-1}} \abs{\Deriv v}^{n-1}
	\leq
	M
	\]
	and  
	\[
	\int_{\Sset^{n-1}}\abs{u-v}
	\leq
	\eta\text{,}
	\]
	we have
	\begin{equation*}
		\brk[\bigg]{\sum_{\sigma \in \Sigma_{\ell}}
			\abs{\deg_\sigma \brk{u} - \deg_{\sigma} \brk{v}}}^{1 - \frac{1}{n}}
		\le
		C
		\sum_{\alpha \in \mathbf{A}_{\square}} \int_{E_\alpha} \abs{\Deriv v}^{n-1} -
		\abs{\Deriv \brk{\Theta_{\ell,\alpha} \compose v}}^{n-1}\text{,}
	\end{equation*}
	where the constant \( C > 0 \) depends only on \( n \).
	Moreover, the above statement remains valid if the domain \( \Sset^{n-1} \) is replaced by the boundary \( \partial Q \) of any cube \( Q \subset \Rset^{n} \), with a constant \( C > 0 \) independent of the choice of the cube.
\end{proposition}

\begin{proof}
	Given an open set \(\Omega \subseteq \Sset^{n-1}\) and \(v \in W^{1, n-1} \brk{\Sset^{n-1}, \Rset^{n}}\), we have for every \( x \in \Omega \), by \cref{lemma_Theta}
	\begin{equation*}
		\abs{\Deriv \brk{\Theta_{\ell,\alpha} \compose v} \brk{x}}
		\le
		\frac{\sqrt{n}\abs{\Deriv v \brk{x}}}{1 + 2\dist_{\infty} \brk{v\brk{x}, Q_{\ell, \alpha}}/\ell}\text{,}
	\end{equation*}
	and therefore
	\begin{equation}
		\label{eq:DevDerivvAndProj}
		\begin{split}
			\abs{\Deriv v \brk{x}}^{n-1} -
			\abs{\Deriv \brk{\Theta_{\ell,\alpha} \compose v} \brk{x}}^{n-1}
			&\ge
			\brk[\Big]{1 -
				\frac{n^{\brk{n-1}/2}}{\brk{1 + 2\dist_{\infty} \brk{v\brk{x}, Q_{\ell, \alpha}}/\ell}^{n-1}}}
			\abs{\Deriv v \brk{x}}^{n-1}\text{.}
		\end{split}
	\end{equation}
	Hence, for
 every \(\alpha \in \mathbf{A}_{\gamma}\), we deduce from~\eqref{eq_eif7leizee5aeloh1mohSh1i}, \eqref{eq_dist_cone_boundary_cubes} and~\eqref{eq:DevDerivvAndProj} that
	\begin{equation}
		\label{eq:DevDerivvAndProjFinal}
		\begin{split}
		&\int_{E_\alpha \cap v^{-1} \brk{\manifold{C}_\gamma+\Sigma_{\ell}} }
		\abs{\Deriv v}^{n-1}\\
		&\qquad \le
		\frac{1}{1 - (\sqrt{n}/\brk{2q - 1})^{n - 1}}
		\int_{E_\alpha \cap v^{-1} \brk{\manifold{C}_\gamma+\Sigma_{\ell}}} \abs{\Deriv v}^{n-1} -
		\abs{\Deriv \brk{\Theta_{\ell,\alpha} \compose v}}^{n-1}\text{.}
		\end{split}
	\end{equation}
	
	On the other hand, Proposition~\ref{proposition_sphere_joint_degree_bubbled} ensures that we can choose \( \eta \in \intvo{0}{\infty} \) such that 
	\begin{equation}
	\label{eq_thiPhu7Thimooc9dai6ka9va}
	\begin{split}
		\brk[\bigg]{\sum_{\sigma \in \Sigma_{\ell}}
			\abs{\deg_\sigma \brk{u} - \deg_{\sigma} \brk{v}}}^{1 - \frac{1}{n}}
		\le
		\Cl{cst:SplitUnion}
		\sum_{\gamma \in \mathbf{\Gamma}} \int_{u^{-1}\brk{G_{\ell,\gamma}} \cap v^{-1}\brk{\manifold{C}_\gamma+\Sigma_{\ell}}} \abs{\Deriv v}^{n-1}\\
		\le 
		\Cr{cst:SplitUnion} \sum_{\gamma \in \mathbf{\Gamma}} \sum_{\alpha \in \mathbf{A}_{\gamma}}
		\int_{E_\alpha \cap v^{-1}\brk{\manifold{C}_\gamma+\Sigma_{\ell}}} \abs{\Deriv v}^{n-1}
		\text{,}
      \end{split}
	\end{equation}
	in view of \eqref{eq:FamilyEalpha}.
	We deduce from \eqref{eq:DevDerivvAndProjFinal} and~\eqref{eq_thiPhu7Thimooc9dai6ka9va} that
	\begin{equation*}
		\begin{split}
		\brk[\bigg]{\sum_{\sigma \in \Sigma_{\ell}}
			\abs{\deg_\sigma \brk{u} - \deg_{\sigma} \brk{v}}}^{1 - \frac{1}{n}}
		&\le
		\Cl{cst_oixai4ooweiJe6Ee8aimuo4u}
		 \sum_{\gamma \in \mathbf{\Gamma}} \sum_{\alpha \in \mathbf{A}_{\square}} \int_{E_\alpha} \abs{\Deriv v}^{n-1}  -
		\abs{\Deriv \brk{\Theta_{\ell,\alpha} \compose v}}^{n-1}\\
		&\le \Cr{cst_oixai4ooweiJe6Ee8aimuo4u} 2^n 
		\sum_{\alpha \in \mathbf{A}_{\square}} 
		\int_{E_\alpha} \abs{\Deriv v}^{n-1}  -
		\abs{\Deriv \brk{\Theta_{\ell,\alpha} \compose v}}^{n-1}
		\text{,}
		\end{split}
	\end{equation*}
	since \(\# \mathbf{\Gamma} = 2^n\).
	Moreover, since \( \mathcal{H}^{n-1}\brk{\manifold{C}_{\gamma} \cap \Sset^{n-1}} \) only depends on \( n \), the constant \( \Cr{cst_oixai4ooweiJe6Ee8aimuo4u} \) only depends  on \( n \).

	This proves the proposition when the domain is \( \Sset^{n-1} \).
	The fact that the statement remains valid if the domain is instead the boundary of a cube follows from the fact that there is a bi-Lipschitz transformation between \( \Sset^{n-1} \) and the boundary of a unit cube in \( \Rset^{n} \), along with the scaling invariance of the \( W^{1,n-1} \) energy in dimension \( n-1 \).
	\resetconstant
\end{proof}

\subsection{Lower bound on the relaxed energy}

In this section, we give the final and key estimate that allows us to prove that the maps we construct indeed lead to an analytical obstruction to weak density.
Given \( n \geq 2 \), we define the set 
\begin{equation*}
	\smash{\widetilde{\manifold{N}}_{0}}
	\defeq
	\set{\brk{x_1, \dotsc, x_n} \in \Rset^{n} \st \text{there exists \( i \in \set{1,\dotsc,n} \) such that \( x_{i} \in \Zset \)}}\text{.}
\end{equation*}
We observe that \( \smash{\widetilde{\manifold{N}}_{0}} \) is actually the \( \brk{n-1} \)-skeleton of the standard decomposition of \( \Rset^{n} \) into unit cubes, such that the origin is a vertex of one cube.

\begin{proposition}
\label{proposition_gap_v_and_proj}
Let \(u \in W^{1, n-1} \brk{G_{\ell, \square}, \smash{\widetilde{\manifold{N}}_{0}}} \cap C\brk{G_{\ell, \square} \setminus \brk{\Zset+1/2}^{n}, \smash{\widetilde{\manifold{N}}_{0}}} \) be such that, for every \(\alpha \in \mathbf{A}_{\square}\),
\begin{equation}
\label{eq_keleopheuvahsuuhu0hei3Ee}
 u \brk{Q_{\ell, \alpha}} \subseteq Q_{\ell, \alpha}\text{.}
\end{equation}
Assume also that \(v_{k} \in C \brk{Q_{\brk{2q+1}\ell}, \smash{\widetilde{\manifold{N}}_{0}}}\) and that \(\brk{v_{k} \restr{G_{\ell, \square}}}_{k \in \Nset}\) converges weakly to \(u\) in \( W^{1,n-1} (G_{\ell, \square},\smash{\widetilde{\manifold{N}}_{0}})\).
If
\(
  n < \brk{2 q - 1}^2
\),
then
\begin{equation}
\label{eq_ahDi3iegeinoobe8om5shee8}
 \liminf_{k \to \infty} \sum_{\alpha \in \mathbf{A}_{\square}}
 \int_{Q_{\ell, \alpha}} \abs{\Deriv v_{k}}^{n-1}
 - \abs{\Deriv \brk{\Theta_{\ell, \alpha} \compose v_{k}}}^{n-1}
 \ge C \ell^{n}\text{,}
\end{equation}
for some constant \( C > 0 \) depending only on \( n \).
\end{proposition}
\begin{proof}
We consider \(Q^t\) to be the cube of edge length \(t \in \intvo{0}{\brk{2q+1}\ell}\) and same center as \(Q_{\brk{2q+1}\ell}\).
For almost every \(t \in \brk{\brk{2 q - 1}\ell, \brk{2q+1}\ell}\), the map
\(u \restr{\partial Q^t}\) is continuous,
and by~\eqref{eq_keleopheuvahsuuhu0hei3Ee}, for every \(\sigma \in \Sigma_{\ell} \), \(\deg_{\sigma} \brk{u \restr{\partial Q^{t}}} = 1\).
On the other hand, we know that \( \deg_{\sigma} \brk{v_{k} \restr{\partial Q^t}} = 0 \).
Hence,
\begin{equation*}
 \sum_{\sigma \in \Sigma_{\ell}} \abs{\deg_{\sigma} \brk{u \restr{\partial Q^t}} - \deg_{\sigma} \brk{v_{k} \restr{\partial Q^t}}}
 = \ell^{n}\text{.}
\end{equation*}

By Fubini--Tonelli’s theorem and Fatou’s lemma, we can assume up to a subsequence that
\begin{multline}
\label{eq_paejohchilije9ew3inge3Ei}
  \liminf_{k \to \infty} \sum_{\alpha \in \mathbf{A}_{\square}} \int_{\partial Q^t \cap Q_{\ell, \alpha}} \abs{\Deriv v_{k}}^{n-1}
 - \abs{\Deriv \brk{\Theta_{\ell, \alpha} \compose v_{k}}}^{n-1}\\
 \le \frac{1}{\ell} \liminf_{k \to \infty} \sum_{\alpha \in \mathbf{A}_{\square}} \int_{Q_{\ell, \alpha}} \abs{\Deriv v_{k}}^{n-1}
 - \abs{\Deriv \brk{\Theta_{\ell, \alpha} \compose v_{k}}}^{n-1}\text{,}
\end{multline}
that
\begin{equation*}
 \sup_{k \in \Nset} \int_{\partial Q^t} \abs{\Deriv v_{k}}^{n-1} < \infty\text{,}
\end{equation*}
and that
\begin{equation*}
 \lim_{k \to \infty} \int_{\partial Q^t} \abs{v_{k} - v} = 0\text{.}
\end{equation*}

Taking \(E_\alpha \defeq \partial Q^t \cap Q_{\ell, \alpha}\),
we observe that~\eqref{eq:FamilyEalpha} is indeed satisfied.
Therefore, if \( k \) is sufficiently large depending on \( \eta \), we deduce that Proposition~\ref{proposition_bound_degree_gap_by_projection} applies and yields
\begin{equation*}
 \ell^{n-1} \le \C \liminf_{k \to \infty} \sum_{\alpha \in \mathbf{A}_{\square}} \int_{\partial Q^t \cap Q_{\ell, \alpha}} \abs{\Deriv v_{k}}^{n-1}
 - \abs{\Deriv \brk{\Theta_{\ell, \alpha} \compose v_{k}}}^{n-1}\text{.}
\end{equation*}
The conclusion~\eqref{eq_ahDi3iegeinoobe8om5shee8} follows then from~\eqref{eq_paejohchilije9ew3inge3Ei}.
\resetconstant
\end{proof}

We recall the definition of the relaxed energy of a map \(u \in W^{1, p} \brk{\manifold{M}, \lifting{\manifold{N}}_0}\) as
\begin{equation*}
\begin{split}
  &\mathcal{E}^{1, p}_{\textnormal{rel}} \brk{u, \manifold{M}}\\
  &\qquad \defeq \inf\set[\Big]{\liminf_{k \to \infty} \mathcal{E}^{1,p}\brk{v_{k}, \manifold{M}} \st v_{k} \to u \text{ a.e.} \text{ and } v_{k} \in W^{1,p}\brk{\manifold{M}, \lifting{\manifold{N}}_0} \cap C^{0} \brk{\manifold{M}, \lifting{\manifold{N}}_0}}\text{.}
\end{split}
\end{equation*}
Using the tools constructed before, we may finally give a lower bound on the relaxed energy of maps with values into \( \smash{\widetilde{\manifold{N}}_{0}} \).
From now on, we work with the map \( u \) defined in the introduction as the retraction from \( \Rset^{n} \) to \( \smash{\widetilde{\manifold{N}}_{0}} \).
More precisely, \( u \) is defined as
\begin{equation}
\label{eq_25f11db66dd9864e}
	u\brk{x}
	=
	\sigma + \frac{x-\sigma}{2\abs{x-\sigma}_{\infty}}
	\quad
	\text{whenever \( x \in \sigma + \intvc{-1/2}{1/2}^{n} \).}
\end{equation}

\begin{proposition}
\label{prop:RelaxedEnergyuell}
Let \( u \) be defined as in \eqref{eq_25f11db66dd9864e}.
For every \(n \geq 2 \), 
\begin{equation*}
 \liminf_{\ell \to \infty} \frac{\mathcal{E}^{1, n-1}_{\textnormal{rel}} \brk{u, Q_{\ell}}}{\ell^{n} \ln \ell} > 0\text{.}
\end{equation*}
\end{proposition}

\begin{proof}
If \(v \in W^{1, n-1} \brk{Q_{\brk{2q+1}\ell}, \smash{\widetilde{\manifold{N}}_{0}}}\) and \( v_{k} \in W^{1,n-1}\brk{Q_{\brk{2q+1}\ell}, \smash{\widetilde{\manifold{N}}_{0}}} \cap C^{0} \brk{Q_{\brk{2q+1}\ell}, \smash{\widetilde{\manifold{N}}_{0}}} \) are such that \(v_{k} \to v\) almost everywhere, then
\begin{equation*}
\begin{split}
  \int_{Q_{\brk{2q+1}\ell}} \abs{\Deriv v_{k}}^{n-1}
  &= \sum_{\alpha \in \mathbf{A}} \int_{Q_{\ell, \alpha}} \abs{\Deriv v_{k}}^{n-1}\\
  &= \sum_{\alpha \in \mathbf{A} \setminus \mathbf{A}_{\square}}
  \int_{Q_{\ell, \alpha}} \abs{\Deriv v_{k}}^{n-1}\\
  &\qquad
  + \sum_{\alpha \in \mathbf{A}_{\square}}  \int_{Q_{\ell, \alpha}} \abs{\Deriv \brk{\Theta_{\ell, \alpha} \compose v_{k}}}^{n-1}\\
  &\qquad
  + \sum_{\alpha \in \mathbf{A}_{\square}} \int_{Q_{\ell, \alpha}}
  \abs{\Deriv v_{k}}^{n-1}-
  \abs{\Deriv \brk{\Theta_{\ell, \alpha} \compose v_{k}}}^{n-1}\text{.}
\end{split}
\end{equation*}
If \(v \brk{Q_{\ell, \alpha}}\subseteq Q_{\ell, \alpha}\), we have
\(\Theta_{\ell, \alpha} \compose v_{k} \restr{Q_{\ell, \alpha}} \to v\restr{Q_{\ell, \alpha}} \) almost everywhere, and thus, assuming that \(n < \brk{2 q - 1}^2\) and invoking Proposition~\ref{proposition_gap_v_and_proj}, we deduce that
\begin{equation*}
  \mathcal{E}^{1, n-1}_{\textnormal{rel}} \brk{v, Q_{\brk{2q+1}\ell}}
  \ge \sum_{\alpha \in \mathbf{A}} \mathcal{E}^{1, n-1}_{\mathrm{rel}} \brk{v, Q_{\ell, \alpha}} + C \ell^{n}\text{.}
\end{equation*}

In particular, if we take \( v = u \), we have
\begin{equation*}
 \mathcal{E}^{1, n-1}_{\textnormal{rel}} \brk{u, Q_{\brk{2q+1}\ell}}
 \ge \brk{2q+1}^{n} \mathcal{E}^{1, n-1}_{\textnormal{rel}} \brk{u, Q_{\ell}} + C \ell^{n}\text{,}
\end{equation*}
or equivalently,
\begin{equation*}
\frac{\mathcal{E}^{1, n-1}_{\textnormal{rel}} \brk{u, Q_{\brk{2q+1}\ell}}}{\brk{\brk{2q+1}\ell}^{n}}
 \ge \frac{\mathcal{E}^{1, n-1}_{\textnormal{rel}} \brk{u, Q_{\ell}}}{\ell^{n}} + \frac{C}{\brk{2q+1}^{n}}\text{.}
\end{equation*}
It follows by induction that
\begin{equation*}
 \frac{\mathcal{E}^{1, n-1}_{\textnormal{rel}} \brk{u, Q_{\brk{2q+1}^{m}}}}{\brk{2q+1}^{mn}}
 \ge \frac{C m}{\brk{2q+1}^{n}}\text{,}
\end{equation*}
and the conclusion follows.
\end{proof}

As explained in the introduction, if \( \smash{\widetilde{\manifold{N}}_{0}} \) were a compact Riemannian manifold, we would be done with the proof of our main result.
In the next section, we explain how to remedy the non compactness issue, by taking a suitable quotient space of \( \smash{\widetilde{\manifold{N}}_{0}} \).

\subsection{Analytical obstruction with values into a skeleton}
\label{subsection:AnalyticalObstructionSkeleton}

In this section, we explain how to construct a sequence of Sobolev mappings whose relaxed energy grows superlinearly with respect to its Sobolev energy when the target is not a manifold but merely a skeleton of the \( n \)-dimensional torus \( \Tset^{n} \).
Later on, in Section~\ref{section_from_skeleton_to_manifold}, we will explain how to move from a skeleton to a legitimate compact manifold without boundary, while preserving all the required properties of our construction.

We define \( \manifold{N}_{0} \) to be the \(\brk{n-1}\)-skeleton of the torus \( \Tset^{n} \), or equivalently the quotient
\[
  \manifold{N}_{0} = \lifting{\manifold{N}}_0/\Zset^n.
\]
In particular, \( \manifold{N}_{0} \) can be isometrically embedded --- up to a dilation --- into a Euclidean space, as a subset of the compact Riemannian manifold \( \Tset^{n} \subseteq \Rset^{2n}\).
Here, we define \( \Tset^{n} \) as
\begin{equation}
\label{eq_df154aa9c515bb26}
	\Tset^{n}
	=
	\Sset^{1} \times \dotsb \times \Sset^{1}
	=
	\set{x \in \Rset^{2n} \st \text{\( x_{2i-1}^{2} + x_{2i}^{2} = 1 \) for every \( 1 \leq i \leq n \)}}\text{.}
\end{equation}
The manifold \( \manifold{N}_{0} \) enjoys the following properties.

\begin{proposition}
\label{proposition_topological_properties_skeleton}
Assume that \( n \geq 3 \).
The quotient map \(\pi\colon \smash{\widetilde{\manifold{N}}_{0}} \to \manifold{N}_{0} \) is a universal cover, and
\[ \pi_{1}\brk{\manifold{N}_{0}} \simeq \Zset^{n}\text{,}
\]
\[ \pi_{n-1}\brk{\manifold{N}_{0}} \simeq \textstyle\bigoplus_{\Zset^{n}}\Zset\text{,} \]and
\[ 
\pi_{2}\brk{\manifold{N}_{0}} \simeq \dotsb \simeq \pi_{n-2}\brk{\manifold{N}_{0}} \simeq \set{0} \text{.}
\]
In addition, \( \pi_{n - 1}\brk{\manifold{N}_{0}} \) is finitely generated as a module over the action of \( \pi_{1}\brk{\manifold{N}_{0}} \).
Finally, the higher order homotopy groups of \( \manifold{N}_{0} \), \( \pi_{k}\brk{\manifold{N}_{0}} \) with \( k > n-1 \), may be computed as the corresponding homotopy groups of a bouquet of infinitely many \( \brk{n - 1} \)-spheres, one for each element of \( \Zset^{n} \).
\end{proposition}
\begin{proof}
	First of all, it is straightforward from the natural cell complex structure of \( \smash{\widetilde{\manifold{N}}_{0}} \)  that its integer homology groups are given by
	\begin{equation*}
		H_{j}\brk{\smash{\widetilde{\manifold{N}}_{0}}}
		\simeq
		\begin{cases}
			\set{0} & \text{if \( j \neq n-1 \),} \\
			\bigoplus_{\Zset^{n}}\Zset & \text{if \( j = n-1 \).}
		\end{cases}
	\end{equation*}
	Actually, there is one copy of \( \Zset \) in \( H_{n-1}\brk{\smash{\widetilde{\manifold{N}}_{0}}} \) for every boundary of an \( n \)-cube in \( \smash{\widetilde{\manifold{N}}_{0}} \).

	Therefore, Hurewicz's theorem or a direct argument implies that \( \pi_{1}\brk{\smash{\widetilde{\manifold{N}}_{0}}} \simeq \dotsb \simeq \pi_{n-2}\brk{\smash{\widetilde{\manifold{N}}_{0}}} \simeq \set{0} \), while \( \pi_{n-1}\brk{\smash{\widetilde{\manifold{N}}_{0}}} \simeq \bigoplus_{\Zset^{n}}\Zset \).
	Since \( \pi \colon \smash{\widetilde{\manifold{N}}_{0}} \to \manifold{N}_{0} \) is the covering associated to the action of \( \Zset^{n} \) on \( \smash{\widetilde{\manifold{N}}_{0}} \), we deduce that \( \pi_{1}\brk{\manifold{N}_{0}} = \Zset^{n} \), \( \pi_{n-1}\brk{\manifold{N}_{0}} \simeq \bigoplus_{\Zset^{n}}\Zset \), and \( \pi_{2}\brk{\manifold{N}_{0}} \simeq \dotsb \simeq \pi_{n-2}\brk{\manifold{N}_{0}} \simeq \set{0} \).
	Moreover, we know that \( \pi_{1}\brk{\manifold{N}_{0}} \) acts as \( \Zset^{n} \) by translation on \( \smash{\widetilde{\manifold{N}}_{0}} \); see e.g.~\cite{Hatcher_2002}*{Proposition~1.40}.
	Let us investigate the effect of this action on \( \pi_{n-1}\brk{\manifold{N}_{0}} \simeq \pi_{n-1}\brk{\smash{\widetilde{\manifold{N}}_{0}}} \).
	Assume without loss of generality that the origin \(0\) has been chosen as the basepoint for the homotopy groups of \(\smash{\widetilde{\manifold{N}}_{0}} \).
	Then, any loop in \( \gamma \in \pi_{1}\brk{\manifold{N}_{0}} \simeq \Zset^{n} \) can be lifted to a path in \( \smash{\widetilde{\manifold{N}}_{0}} \) connecting the basepoint \(0\) to the associated endpoint \( z_{\gamma} \in \Zset^{n} \).
	Let \( a \) be the element of \( \pi_{n-1}\brk{\manifold{N}_{0}} \) obtained by projecting the element of \( \pi_{n-1}\brk{\smash{\widetilde{\manifold{N}}_{0}}} \) that covers once the cycle \( \partial\intvc{0}{1}^{n} \).
	By construction, the element \( \gamma a \in \pi_{n-1}\brk{\manifold{N}_{0}} \) lifts to the element of \( \pi_{n-1}\brk{\smash{\widetilde{\manifold{N}}_{0}}} \) that covers once the cycle \( z_{\gamma} +  \partial\intvc{0}{1}^{n} \); see also~\cite{Hatcher_2002}*{\S4.1, Exercise~4}.
	(The action can also be described as a particular case of a general bijection between pointed homotopy classes and free homotopy classes in the covering space~\citelist{\cite{Spanier_1966}*{Corollary 7.3.7}\cite{Whitehead_1978}*{III~(1.15)}}.)
	When \( \gamma \) runs over \( \pi_{1}\brk{\manifold{N}_{0}} \), it is readily checked that we obtain the generating family of \( \pi_{n-1}\brk{\smash{\widetilde{\manifold{N}}_{0}}} \) associated with the homology basis shown above, one element for each boundary of a unit \( n \)-cube in \( \smash{\widetilde{\manifold{N}}_{0}} \).
	Therefore, \(  \pi_{n-1}\brk{\manifold{N}_{0}} \) is generated by one element over the action of \( \pi_{1}\brk{\manifold{N}_{0}} \).

	Finally, the affirmation concerning the higher order homotopy groups of \( \manifold{N}_{0} \) comes from the fact that \( \smash{\widetilde{\manifold{N}}_{0}} \) is homotopically equivalent to a bouquet of infinitely many \( \brk{n-1} \)-spheres, one for each cube in \(\smash{\widetilde{\manifold{N}}_{0}} \).
	This can be proved by a hand-made construction.
	However, for the sake of completeness, we sketch a proof, relying on a more general argument.
	We start with the following exact sequence, which is a part of the long exact sequence relating the homology of a space, a subspace, and the associated quotient space, see e.g.~\cite{Hatcher_2002}*{Theorem~2.13}:
	\begin{equation}
	\label{eq_exact_sequence_quotient_space}
		\begin{tikzcd}
			\set{0} \simeq H_{n-1}\brk{\lifting{\manifold{N}}_{0}^{n-2}} \arrow[r] & H_{n-1}\brk{\lifting{\manifold{N}}_{0}} \arrow[r, "j_{\sharp}"] & H_{n-1}\brk{\lifting{\manifold{N}}_{0}/\lifting{\manifold{N}}_{0}^{n-2}} \arrow[lld, "\partial" above] \\[1em] H_{n-2}\brk{\lifting{\manifold{N}}_{0}^{n-2}} \arrow[r] & H_{n-2}\brk{\lifting{\manifold{N}}_{0}} \simeq \set{0}\text{,} &
		\end{tikzcd}
	\end{equation}
	where \( \lifting{\manifold{N}}_{0}^{n-2} \) is the \( \brk{n-2} \)-skeleton of \( \lifting{\manifold{N}}_{0} \).
	The quotient space \( \lifting{\manifold{N}}_{0}/\lifting{\manifold{N}}_{0}^{n-2} \) is a bouquet of \( \brk{n-1} \)-spheres, one for each \( \brk{n-1} \)-cell in \( \lifting{\manifold{N}}_{0} \), the map \( j \) is the quotient map that sends each cell to the corresponding sphere, and the map \( \partial \) sends every sphere to the boundary of the associated cell, which is an \( \brk{n-2} \)-cycle.
	The exact sequence~\eqref{eq_exact_sequence_quotient_space} implies that the map \( \partial \) is an epimorphism. Since \( H_{n-2}\brk{\lifting{\manifold{N}}_{0}^{n-2}} \) is free, this map admits a section \( \sigma \), and moreover, it can be constructed so that its image \( \operatorname{Im} \sigma \) is the submodule generated by a subset of the spheres constituting \( \lifting{\manifold{N}}_{0}/\lifting{\manifold{N}}_{0}^{n-2} \).
	This section induces the direct sum decomposition%
	\[
	H_{n-1}\brk{\lifting{\manifold{N}}_{0}/\lifting{\manifold{N}}_{0}^{n-2}} \simeq H_{n-1}\brk{\lifting{\manifold{N}}_{0}} \oplus \operatorname{Im} \sigma\text{.}
	\]
	Collapsing the spheres which generate \( \operatorname{Im}\sigma \), the map \( j \) induces a map \( h \) from \( \lifting{\manifold{N}}_{0} \) to a (possibly smaller) bouquet of spheres, and our argument shows that this map induces an isomorphism in homology in degree \( n-1 \).
	It also induces an isomorphism in homology in all the other degrees, since all the other homology groups of both spaces involved are trivial (except for the \( H_{0} \), which are both \( \Zset \), and on which this map also induces an isomorphism).
	A corollary of Whitehead's and Hurewicz's theorems, see e.g.~\cite{Hatcher_2002}*{Corollary~4.33}, allows us to conclude that \( h \) is a homotopy equivalence.
\end{proof}

We note importantly that, although there is a correspondence between the cubes in \(\smash{\widetilde{\manifold{N}}_{0}} \) and the spheres of the bouquet to which it is homotopic, the homotopy equivalence between the bouquet of spheres and \(\smash{\widetilde{\manifold{N}}_{0}} \) \emph{does not} map each generator associated to a sphere to the generator associated to the corresponding cube.
Instead, a generator associated to a sphere may be mapped to a generator associated to a cycle containing an arbitrary large number of cubes in \(\smash{\widetilde{\manifold{N}}_{0}} \).
This can be intuitively seen from the proof: \( \lifting{\manifold{N}}_{0} \) is sent to a bouquet of spheres by collapsing its \( \brk{n-2} \)-skeleton, sending all \( \brk{n-1} \)-faces to spheres, and then collapsing a certain number of those spheres.
Therefore, the spheres of the bouquet are in correspondence with a family of cycles that form a basis of \( H_{n-1}\brk{\lifting{\manifold{N}}_{0}} \) chosen such that the faces that have not been collapsed are part of one and only one cycle.
However, it is readily seen that it is impossible to select some faces of \( \lifting{\manifold{N}}_{0} \) in such a way that the corresponding family of cycles can be chosen to be all boundaries of unit cubes.

This can be put into the framework of the third strategy for proving weak approximation explained in the introduction, based on the construction and elimination of connections \citelist{\cite{BrezisCoronLieb1986}\cite{AlmgrenBrowderLieb1988}\cite{Bethuel_1990}\cite{AlbertiBaldoOrlandi2003}\cite{Canevari_Orlandi_2019}\cite{pakzad_riviere_2003}}.
Indeed, by composing an \( \manifold{N}_{0} \)-valued \( W^{1,n-1} \) mapping \( u \) with the homotopy equivalence \( h \colon \lifting{\manifold{N}}_{0} \to \bigvee_{\Zset^{n}} \Sset^{n-1} \) constructed above --- which can be taken to be Lipschitz, as it merely amounts to collapse some cells of \( \lifting{\manifold{N}}_{0} \) and map the other ones to spheres --- one can proceed with the first part of the strategy, namely construct and analyze the topological singular set of the map \( u \) and obtain a minimal connection for this set, whose length is controlled by the Sobolev energy of \( u \), identifying separately the contributions of the different spheres of the bouquet.
However, one cannot proceed with the second step of the strategy, as eliminating the singularities corresponding to the different spheres of the bouquet cannot be done with uniformly controlled maps for creating the dipoles, due to the fact that some of these spheres are associated with arbitrarily large cycles in \( \lifting{\manifold{N}}_{0} \).
This explains why this strategy cannot be applied to obtain weak approximation in our context.

We now provide an instance of a sequence of Sobolev mappings with values into the \emph{compact} skeleton \( \manifold{N}_{0} \) whose relaxed energy grows superlinearly with respect to the Sobolev energy.
We define \( v = \pi \circ u \), where \( u \) is the map from Proposition~\ref{prop:RelaxedEnergyuell}.
In particular, we note that the map \( v \) is \( \Zset^{n} \)-periodic.

\begin{proposition}
\label{prop_counterexample_skeleton}
	For every \( n \geq 3 \), 
	\begin{equation*}
		\liminf_{\ell \to \infty} \frac{\mathcal{E}^{1, n-1}_{\textnormal{rel}} \brk{v, Q_{\ell}}}{\ell^{n} \ln \ell} > 0\text{.}
	\end{equation*}
\end{proposition}

The map \( v \) being \( \Zset^{n} \)-periodic, it satisfies in particular \( \mathcal{E}^{1,n-1}\brk{v,Q_{\ell}} \lesssim \ell^{n} \), so that the above theorem indeed provides a superlinear growth of the relaxed energy with respect to the Sobolev energy.

\begin{proof}[Proof of Proposition~\ref{prop_counterexample_skeleton}]
	Assume that we are given a sequence \( \brk{w_{k}}_{k \in \Nset} \) of mappings in \( W^{1,n-1}\brk{Q_{\ell}, \manifold{N}_{0}} \cap C\brk{Q_{\ell}, \manifold{N}_{0}} \) that converges weakly to \( v \) in \( W^{1,n-1} \) on \( Q_{\ell} \).
	We perform a lifting in the spirit of \citelist{\cite{Bourgain_Brezis_Mironescu_2000}\cite{Bethuel_Chiron_2007}\cite{Bethuel_Zheng_1988}}.
	Since \( Q_{\ell} \) is simply connected, the maps \( w_{k} \) may be lifted as maps \( \widetilde{w}_{k} \in W^{1,n-1}\brk{Q_{\ell}, \smash{\widetilde{\manifold{N}}}_{0}} \cap C\brk{Q_{\ell}, \smash{\widetilde{\manifold{N}}}_{0}} \).
	In addition, since \( \pi \) is a local isometry, it holds that
	\begin{equation*}
		\int_{Q_{\ell}} \abs{\Deriv w_{k}}^{n-1}
		=
		\int_{Q_{\ell}} \abs{\Deriv \widetilde{w}_{k}}^{n-1}\text{.}
	\end{equation*}
	Moreover, taking profit of the fact that \( \Zset^{n} \) acts on \( \smash{\widetilde{\manifold{N}}_{0}} \) and that those deck transformations are translations, we may choose the liftings \( \widetilde{w}_{k} \) so that
	\begin{equation*}
		\brk[\bigg]{\fint_{Q_{\ell}}\widetilde{w}_{k}}_{k \in \Nset}
		\quad
		\text{is bounded.}
	\end{equation*}
	Therefore, the Poincaré--Wirtinger inequality ensures that the sequence \( \brk{\widetilde{w}_{k}}_{k \in \Nset} \) is bounded in \( L^{n-1} \), and hence, up to extraction of a subsequence, it converges weakly to some map \( w \in W^{1,n-1}\brk{Q_{\ell}, \smash{\widetilde{\manifold{N}}_{0}}} \).
	Here we use the fact that \( n-1 > 1 \).
	By the continuity of the covering map \( \pi \), we have almost everywhere
	\begin{equation}
	\label{eq:PointwiseConvergenceLiftings}
		\pi \compose w 
		= 
		\lim_{k \to \infty} \pi \compose \widetilde{w}_{k}
		=
		\lim_{k \to \infty} w_{k}
		=
		v
		=
		\pi \compose u\text{.}
	\end{equation}
	Since the action of \( \Zset^{n} \) by translation is transitive, by the uniqueness of liftings in Sobolev spaces (see e.g.~\citelist{\cite{Bourgain_Brezis_Mironescu_2000}\cite{Bethuel_Chiron_2007}\cite{Bethuel_Zheng_1988}\cite{Mironescu_VanSchaftingen_2021_APDE}*{Proposition~4.2}}), it follows that \( w = u + a \) for some \( a \in \Zset^{n} \).
	Hence, we deduce that
	\begin{equation*}
		\liminf_{k \to \infty} 	\int_{Q_{\ell}} \abs{\Deriv w_{k}}^{n-1}
		=
		\liminf_{k \to \infty} \int_{Q_{\ell}} \abs{\Deriv \widetilde{w}_{k}}^{n-1}
		\geq
		\mathcal{E}^{1,n-1}_{\textnormal{rel}}\brk{u, Q_{\ell}}\text{.}
	\end{equation*}
	Taking the infimum over all sequences weakly converging to \( v \), we obtain
	\begin{equation*}
		\mathcal{E}^{1,n-1}_{\textnormal{rel}}\brk{v, Q_{\ell}} \geq \mathcal{E}^{1,n-1}_{\textnormal{rel}}\brk{u, Q_{\ell}}\text{.}
	\end{equation*}
	The conclusion follows from Proposition~\ref{prop:RelaxedEnergyuell}.
	\resetconstant
\end{proof}

We point out that all our analysis before Proposition~\ref{prop_counterexample_skeleton} did not rely on the fact that \( n \ge 3 \).
Since weak density always holds in \( W^{1,1} \), a uniform linear bound on the relaxed energy must hold and thus Proposition~\ref{prop_counterexample_skeleton} cannot be extended to that case. 
Let us examine why our construction does not contradict weak density in \( W^{1,1} \). 
Proposition~\ref{prop:RelaxedEnergyuell} yields
\begin{equation*}
	 \liminf_{\ell \to \infty} \frac{\mathcal{E}^{1, 1}_{\textnormal{rel}} \brk{u, Q_{\ell}}}{\ell^{2} \ln \ell} > 0\text{.}
\end{equation*}
When \( n = 2 \), \( \smash{\widetilde{\manifold{N}}_{0}} \) is the \( 1 \)-skeleton of \( \Rset^{2} \), and \( \manifold{N}_{0} \) is a bouquet of two circles.
In particular, \( \pi \colon \smash{\widetilde{\manifold{N}}_{0}} \to \manifold{N}_{0} \) is no longer the universal cover, but it is nevertheless a covering.
Actually, the effect of passing from \( \manifold{N}_{0} \) to \( \smash{\widetilde{\manifold{N}}_{0}} \) is to abelianize the fundamental group \( \pi_{1} \brk{\manifold{N}_0} \): the covering map \(\pi  \colon \lifting{\manifold{N}}_{0} \to \manifold{N}_{0} \) is normal, maps \(\pi_1 \brk{\lifting{\manifold{N}}_0}\) to the commutator group of \(\pi_1 \brk{\manifold{N}_0}\), and has the abelianization of \(\pi_1 \brk{\manifold{N}_0}\) as a group of deck transformations.
This is however no obstruction to constructing the map \( v \), and to proceed to the above reasoning up to~\eqref{eq:PointwiseConvergenceLiftings} included.
However, the main difference is that here, \( n-1 = 1 \), and hence, boundedness in \( W^{1,1} \) only implies weak convergence to a BV map, not to a \( W^{1,1} \) map.
Since there is no uniqueness of lifting in BV, one cannot conclude that \( w = u + a \) and hence transfer the energy gap estimate for \( u \) to \( v \).
Actually, it is not difficult to construct by hand a BV lifting of \( v \) such that there indeed exists a sequence of continuous maps, bounded in \( W^{1,1} \), that converges to this lifting almost everywhere.
Of course, it is therefore not \( W^{1,1} \) itself.

Another way to understand the difference between the case \( n = 2 \) and \( n \geq 3 \) is via the fact that, in contrast with the case \( n \geq 3 \), when \( n = 2 \), \( \pi_{n - 1}\brk{\manifold{N}_{0}} = \pi_{1}\brk{\manifold{N}_{0}} = \Zset \ast \Zset \) \emph{is} finitely generated, as the free group over two generators.
In the context of the third strategy for proving weak density results that we presented in the introduction, via analysis of the topological singularities and elimination by a dipole insertion, this implies that all the generators of \( \pi_{1}\brk{\manifold{N}_{0}} \) can be eliminated with uniformly bounded \( W^{1,1} \) energy, while for \( n \geq 3 \), the obstruction to the weak approximation property that we obtain precisely comes from the fact that \( \pi_{n - 1}\brk{\manifold{N}_{0}} \) has no generating set whose elements can be eliminated with uniformly bounded \( W^{1,n-1} \) energy, as we already mentioned.
Let us however mention that turning this observation into a rigorous proof of the weak approximation property in \( W^{1,1} \) is a non-trivial task, especially when the fundamental group \( \pi_{1} \brk{\manifold{N}_0}\) of the target is non-abelian --- a difficulty specific to \( \pi_{1} \), higher order homotopy groups always being abelian.
This was done by Pakzad \cite{pakzad_2003}.

\subsection{From the skeleton to a manifold}
\label{section_from_skeleton_to_manifold}
We now transfer our construction to a manifold.

Given \(m \in \Nset \) and \(\lambda \in \intvr{0}{\infty}\), we define the set
\[
	\manifold{N}_{\lambda}
	\defeq V^{-1} \brk{\set{\lambda}}\text{,}
\]
where the function \(V \colon \Tset^{n} \times \Rset^{m} \to \Rset\) is given for \(x \in \Tset^{n} \times \Rset^{m} \subseteq \Rset^\nu\) by
\begin{equation*}
V \brk{x} \defeq \prod_{j = 1}^{n} \frac{1 + x_{2j - 1}}{2}
  + \sum_{j = 1}^m \abs{x_{2n +j}}^2\text{.}
\end{equation*}
We observe that 
\[
 \manifold{N}_0 =  V^{-1} \brk{\set{0}}
 = \set{ x \in \Tset^{n} \st \min_{1 \le j \le n} x_{2j - 1} = -1}
 \times \set{0}
\]
is the set that was used in the previous section, along with its universal covering \(\lifting{\manifold{N}}_{0} \).
From now on, and recording the definition of \( \Tset^{n} \) from \eqref{eq_df154aa9c515bb26}, we will also use angular coordinates, denoting a point \( x \in \Tset^{n} \times \Rset^{m} \) as
\[
	x 
	= 
	\brk{\cos \theta_1, \sin \theta_1, \dotsc,
	\cos \theta_{n}, \sin \theta_{n}, z_1, \dotsc, z_{m}}\text{.}
\]
In this notation, \( \manifold{N}_{0} \) is identified with the sets of points such that \( \theta_{i} \in \pi + 2\pi\Zset \) for some \( 1 \leq i \leq n \).
For the sake of conciseness, we will often also omit the use of trigonometric functions, identifying \( \Tset^{n} \) with (the quotient of) \( \intvc{-\pi}{\pi}^{n} \) via the map
\[
	\theta \in \intvc{-\pi}{\pi}^{n}
	\mapsto
	\brk{\cos \theta_1, \sin \theta_1, \dotsc,
	\cos \theta_{n}, \sin \theta_{n}}
	\in \Tset^{n}\text{.}
\]

\begin{proposition}
\label{proposition_manifold}
For every \(\lambda \in \intvo{0}{1}\),
\(\manifold{N}_\lambda\) is a smooth manifold and there exists a Lipschitz map from \(\manifold{N}_\lambda\) to \(\manifold{N}_0\).
If moreover \(m > 0 \),
\[
 \manifold{N}_{\lambda} \supseteq
 \manifold{N}_{0} \times \Sset^{m - 1}_{\sqrt\lambda}\text{,}
\]
and the restriction of the Lipschitz map \( \manifold{N}_{\lambda} \to \manifold{N}_{0} \) above to \( \manifold{N}_{0} \times \Sset^{m - 1}_{\sqrt\lambda} \) coincides with \( \id_{\manifold{N}_{0}} \times \; 0 \).
If \(i < m\), then
\[
  \pi_i \brk{\manifold{N}_\lambda} \simeq \pi_i \brk{\manifold{N}_0}
  \simeq
  \begin{cases}
     \Zset^{n} & \text{if \(i = 1\)},\\
     \set{0} & \text{if \(2 \le i \le n-2\)},\\
     \pi_i \brk{\bigvee_{\Zset^{n}} \Sset^{n-1}}
     &\text{if \(i \ge n - 1\)}\text{.}
  \end{cases}
\]
Moreover, if \( n-1 < m \), then \( \pi_{n-1}\brk{\manifold{N}_{\lambda}} \) is finitely generated over the action of \( \pi_{1}\brk{\manifold{N}_{\lambda}} \).
\end{proposition}

\begin{proof}
In angular coordinates
\(
 x = \brk{\cos \theta_1, \sin \theta_1, \dotsc,
 \cos \theta_{n}, \sin \theta_{n}, z_1, \dotsc, z_{m}}
\)
,
we have 
\[
 V \brk{x} = \prod_{j = 1}^{n} \frac{1 + \cos \theta_j}{2} + \abs{z}^2 = \prod_{j = 1}^{n} \brk[\big]{\cos \tfrac{\theta_j}{2}}^2 + \abs{z}^2\text{.}
\]
In particular, if \(V \brk{x} < 1\), we have \(\theta \ne 0\).
Taking the gradient, we have, if \(0 < V \brk{x} < 1\),
\[
\abs{\nabla V \brk{x}}^{2} = 
\brk[\bigg]{\sum_{j = 1}^{n} \tan \brk[\big]{\tfrac{\theta_{j}}{2}}^{2}}
\prod_{j = 1}^{n} \brk[\big]{\cos \tfrac{\theta_j}{2}}^4
+ 4 \abs{z}^2 > 0\text{.}
\]
Hence, for every \(\lambda \in \intvo{0}{1}\), the set \(\manifold{N}_{\lambda} \defeq V^{-1} \brk{\set{\lambda}}\) is a smooth submanifold of \( \Tset^{n} \times \Rset^{m} \).

Next, we define 
for \(\brk{\theta, z} \in V^{-1} \brk{\intvr{0}{1}} \subseteq \brk{\intvc{-\pi}{\pi}^n \setminus \set{0}} \times \Rset^m \),
\[
 \Theta \brk{t, \theta, z}
 \defeq \brk[\Big]{\brk{1 + t \brk{\tfrac{\pi}{\abs{\theta}_\infty} - 1}} \theta,\brk{1 - t}z}
  = \brk{1 - t} \brk{\theta, z}
 + t \Theta \brk{1,\theta, z}\text{.}
\]
We observe that if \(t_0, t_1 \in \intvc{0}{1}\) satisfy \(t_0 < t_1\) and if \(x \not \in \manifold{N}_0\), then
\[
 V \brk{\Theta \brk{t_0, x}} > V \brk{\Theta \brk{t_1, x}}\text{.}
\]

It follows from the facts that
\begin{enumerate*}[label=(\roman*)]
\item \(\Theta\) is continuous,
\item
\(
  \Theta \brk{\intvc{0}{1} \times V^{-1} \brk{\intvc{0}{\lambda}}}
  \subseteq V^{-1} \brk{\intvc{0}{\lambda}}
\),
\item \(\Theta \brk{\set{1} \times V^{-1} \brk{\intvc{0}{\lambda}}}
\subseteq V^{-1} \brk{\set{0}}\), and
\item for every \(x \in V^{-1} \brk{\set{0}}\),
\(\Theta \brk{t, x} = x \),
\end{enumerate*}
that the set \(\manifold{N}_0 = V^{-1} \brk{\set{0}}\) is a (strong) deformation retract of \(V^{-1} \brk{\intvc{0}{\lambda}}\).
In particular, for every \(j \in \Nset\), we have 
\[
 \pi_{j} \brk{V^{-1} \brk{\intvc{0}{\lambda}}} = \pi_{j} \brk{\manifold{N}_0}\text{.}
\]
Moreover, it is readily checked that the map \( \Theta\brk{1,\cdot}\restr{\manifold{N}_{\lambda}} \) is the required Lipschitz map \( \manifold{N}_{\lambda} \to \manifold{N}_{0} \).

The map 
\(
 \Theta \restr{\intvr{0}{1} \times \manifold{N}_\lambda} \colon \intvr{0}{1} \times \manifold{N}_\lambda \to V^{-1} \brk{\intvl{0}{\lambda}}
\) is a bijection.
An inspection of its definition shows its injectivity. For the surjectivity, defining for \(\brk{\theta, z} \in V^{-1} \brk{\intvr{0}{1}} \subseteq \brk{\intvc{-\pi}{\pi}^n \setminus \set{0}} \times \Rset^m \),
\begin{equation*}
\Lambda \brk{t, \theta, z}
\defeq \brk[\bigg]{\frac{1 -t \pi/\abs{\theta}_\infty}{1 - t} \theta, \frac{z}{1 - t}},
\end{equation*}
we observe that if \(t < \abs{\theta}_{\infty}/\pi\), \(\Lambda \brk{t, \theta, z}
\in V^{-1} \brk{\intvr{0}{1}} \subseteq \brk{\intvc{-\pi}{\pi}^n \setminus \set{0}} \times \Rset^m \) and \(\Theta \brk{t, \Lambda\brk{t, \theta, z}} = \brk{\theta, z}\).
Since \(t \in \intvo{0}{\abs{\theta}_{\infty}/\pi} \mapsto V \brk{\Lambda \brk{t, \theta, z}}\) is increasing and continuous and since
\(\lim_{t \to \abs{\theta}_{\infty}/\pi} V \brk{\Lambda \brk{t, \theta, z}} > 1\), \(\Theta\) has the required surjectivity property.

For every \(\varepsilon \in \intvo{0}{\lambda}\), the set \(\Theta^{-1} \brk{V^{-1} \brk{\intvc{\varepsilon}{\lambda}}} \subseteq \intvr{0}{1} \times \manifold{N}_{\lambda} \subseteq \intvc{0}{1}  \times \manifold{N}_{\lambda}\) is closed and thus \(\smash{\Theta \restr{\intvr{0}{1} \times \manifold{N}_\lambda} ^{-1} \restr{V^{-1} \brk{\intvc{\varepsilon}{\lambda}}}}\) is continuous.
Since \(V\) is continuous, 
\(\smash{\Theta\restr{\intvr{0}{1} \times \manifold{N}_\lambda}^{-1}}\) is continuous on \(V^{-1} \brk{\intvl{0}{\lambda}}\).
Writing \(\brk{T, R} \defeq  \smash{\Theta \restr{\intvr{0}{1} \times \manifold{N}_\lambda}^{-1}}\) with \(T \in  C \brk{V^{-1} \brk{\intvl{0}{\lambda}}, \intvr{0}{1}}\) and \(R \in C \brk{V^{-1} \brk{\intvl{0}{\lambda}}, \manifold{N}_{\lambda}}\), we note that the map
\[
 \Xi \colon \brk{t, x} \in \intvc{0}{1} \times V^{-1} \brk{\intvl{0}{\lambda}}
 \mapsto \Theta \brk{\brk{1 -t} T \brk{x}, R \brk{x}}
\]
is continuous, that \(\Xi \brk{0, \cdot} = \id\), and that \(\Xi \brk{t, \cdot}\restr{\manifold{N}_\lambda} = \operatorname{id}_{\manifold{N}_\lambda}\).
Therefore, \(V^{-1} \brk{\set{\lambda}}\) is a (strong) deformation retract of \(V^{-1} \brk{\intvl{0}{\lambda}}\), and thus for every \(j \in \Nset\),
\[
 \pi_{j} \brk{V^{-1} \brk{\intvl{0}{\lambda}}} \simeq \pi_{j} \brk{\manifold{N}_{\lambda}}\text{.}
\]
By a cellular or simplicial approximation (see~\cite{Hatcher_2002}*{Theorem 2C.1 or 4.8}), since \(\manifold{N}_0 = V^{-1}\brk{\set{0}}\) and \(\dim \manifold{N}_0 = n - 1\), we have if
\(j \in \set{1, \dotsc,m - 1}\),
\[
 \pi_{j} \brk{V^{-1} \brk{\intvl{0}{\lambda}}} \simeq
 \pi_{j} \brk{V^{-1} \brk{\intvc{0}{\lambda}}}\text{.}
\]

It remains to prove that, if \( n-1 < m \), then \( \pi_{n-1}\brk{\manifold{N}_{\lambda}} \) is finitely generated as a module over the action of \( \pi_{1}\brk{\manifold{N}_{\lambda}} \).
Since \( \manifold{N}_{0} \simeq V^{-1} \brk{\intvc{0}{\lambda}} \) and \( V^{-1} \brk{\intvl{0}{\lambda}} \simeq \manifold{N}_{\lambda} \), it suffices to prove that this property is preserved when passing from \( V^{-1} \brk{\intvc{0}{\lambda}} \) to \( V^{-1} \brk{\intvl{0}{\lambda}} \).
But this follows from the same general position argument.
Indeed, let \( g \) be a generator of \( \pi_{n-1}\brk{V^{-1} \brk{\intvc{0}{\lambda}}} \simeq \pi_{n-1}\brk{V^{-1} \brk{\intvl{0}{\lambda}}} \) as a module over the action of \( \pi_{1}\brk{V^{-1} \brk{\intvc{0}{\lambda}}} \), where the equality has been proved above.
Let also \( a \) be any element of \( \pi_{n-1}\brk{V^{-1} \brk{\intvl{0}{\lambda}}} \simeq \pi_{n-1}\brk{V^{-1} \brk{\intvc{0}{\lambda}}} \).
By assumption, there exists \( \gamma \in \pi_{1}\brk{V^{-1} \brk{\intvc{0}{\lambda}}} \simeq \pi_{1}\brk{V^{-1} \brk{\intvl{0}{\lambda}}} \) such that \( \gamma g = a \) in \( \pi_{n-1}\brk{V^{-1} \brk{\intvc{0}{\lambda}}} \).
But an additional application of the general position argument shows that this equality then also holds in \( \pi_{n-1}\brk{V^{-1} \brk{\intvl{0}{\lambda}}} \), which establishes our claim and concludes the proof.
\end{proof}

\begin{remark}
\label{rmk:target_form}
Let us write
\[
\manifold{N}_{\lambda}
= \manifold{N}_{\lambda}^0 \cup \manifold{N}_{\lambda}^1\text{,}
\]
where
\begin{align*}
 \manifold{N}_{\lambda}^0 &= \manifold{N}_\lambda \setminus \brk{\Rset^{2n}\times B_{\sqrt{\lambda}/2} \brk{0}}&\text{and}&&
 \manifold{N}_{\lambda}^1 &= \manifold{N}_\lambda \cap \brk{\Rset^{2n}\times \Bar{B}_{\sqrt{\lambda}/2} \brk{0}}\text{.}
\end{align*}
Let us emphasize that here, \( \smash{B_{\sqrt{\lambda}/2} \brk{0}} \) is the open ball with center \( 0 \) and radius \( \sqrt{\lambda}/2 \), so that \( \smash{\manifold{N}_{\lambda}^0 \cap \manifold{N}_{\lambda}^1 }= \smash{\manifold{N}_\lambda \cap \brk{\Rset^{2n}\times \partial B_{\sqrt{\lambda}/2} \brk{0}}} \).
We observe that
\begin{align*}
 \manifold{N}_{\lambda}^0 &\simeq \brk{\Tset^n \setminus \Bset^n} \times \Sset^{m - 1}\text{,}&
 \manifold{N}_{\lambda}^1 &\simeq \Sset^{n - 1} \times \Bar{\Bset}^m\text{,}&
 &\text{and}&
 \manifold{N}_{\lambda}^0 \cap \manifold{N}_{\lambda}^1
 \simeq \Sset^{n - 1} \times \Sset^{m - 1}\text{.}
\end{align*}
Without giving detailed arguments, these claims can be justified as follows.
In the sequel, we denote an element of \( \manifold{N}_{\lambda} \subset \Rset^{2n} \times \Rset^{m} \) as \( \brk{x',z} \).
For the description of \( \manifold{N}_{\lambda}^{0} \), we first note that, for a fixed \( x' \in \Rset^{2n} \), the set of those \( z \in \Rset^{m} \) such that \( \brk{x',z} \in \manifold{N}_{\lambda} \) forms a sphere centered at the origin, unless it is empty.
In order to enforce the additional constraint that this sphere should have a radius at least \( \sqrt{\lambda}/2 \), we have to remove those \( x \) such that
\[
\prod_{j=1}^{n}\frac{1+x'_{2j-1}}{2} \geq \sqrt{3\lambda}/2 \text{,}
\]
 which corresponds to removing from \( \Tset^{n} \) a set homeomorphic to a ball, centered at the point \( \brk{1,\dotsc,1} \), hence yielding \( \manifold{N}_{\lambda}^0 \simeq \brk{\Tset^n \setminus \Bset^n} \times \Sset^{m - 1} \).

For the description of \( \manifold{N}_{\lambda}^1 \), we argue in the other direction, by first observing from the definition of \( V \) that, for a fixed \( z \in \Rset^{m} \), the set of those \( x' \in \Rset^{2n} \) such that \( \brk{x',z} \in \manifold{N}_{\lambda} \) is homeomorphic to the sphere \( \Sset^{n-1} \), unless it is empty.
The fact that we restrict to \( z \in \Bar{B}_{\sqrt{\lambda}/2} \brk{0} \) yields the homeomorphism \( \manifold{N}_{\lambda}^1 \simeq \Sset^{n - 1} \times \Bset^m \).

The description of \( \manifold{N}_{\lambda}^0 \cap \manifold{N}_{\lambda}^1 \) follows from a careful combination of the considerations above for both \( \manifold{N}_{\lambda}^{0} \) and \( \manifold{N}_{\lambda}^1 \).

This shows that
\[
	\manifold{N}_{\lambda} \simeq \brk{\Tset^n \setminus \Bset^n} \times \Sset^{m - 1} \cup_{\partial} \Sset^{n - 1} \times \Bset^m\text{,}
\]
where \( \cup_{\partial} \) denotes the attachment along the boundary (see e.g.~\cite{Lee2013}*{Theorem~9.29}).

In the special case where \( m = 1 \), then \( \brk{\Tset^n \setminus \Bset^n} \times \Sset^{m - 1} = \brk{\Tset^n \setminus \Bset^n} \times \set{-1,1} \) consists of two disjoint copies of \( \brk{\Tset^n \setminus \Bset^n} \), and \( \Sset^{n - 1} \times \Bset^m = \Sset^{n - 1} \times \intvo{-1}{1} \) is a cylinder that collapses after attaching, so that 
\[
	\manifold{N}_{\lambda} \simeq \Tset^{n} \,\#\, \Tset^{n}\text{.} 
\]
\end{remark}

We now extend Proposition~\ref{prop_counterexample_skeleton} to a map in a manifold.

\begin{proposition}
\label{prop_nonestimate_manifold}
	For every \( n \geq 3 \) and \(\lambda \in \intvo{0}{1}\),
	\begin{equation*}
		\liminf_{\ell \to \infty} \frac{\mathcal{E}^{1, n-1}_{\textnormal{rel}} \brk{v, Q_{\ell}}}{\ell^{n} \ln \ell} > 0\text{.}
	\end{equation*}
\end{proposition}

The difference between Propositions~\ref{prop_counterexample_skeleton} and \ref{prop_nonestimate_manifold} is that the relaxed energy is taken with respect to sequences in \(\manifold{N}_{0} \) and \(\manifold{N}_\lambda\) respectively.
\begin{proof}

	We view the map \( v \) from Proposition~\ref{prop_counterexample_skeleton} as a Sobolev mapping into \( \manifold{N}_{\lambda} \), taking advantage of the inclusion \( \manifold{N}_{0} \subset \manifold{N}_{\lambda} \) exhibited in Proposition~\ref{proposition_manifold}.
	The conclusion hence follows from Proposition~\ref{prop_counterexample_skeleton} and the fact that any weakly converging sequence of continuous Sobolev maps into \( \manifold{N}_{\lambda} \) can be projected to a weakly converging sequence of continuous Sobolev maps into \( \manifold{N}_{0} \)  through the Lipschitz map \( \Phi \colon \manifold{N}_{\lambda} \to \manifold{N}_{0} \) provided by Proposition~\ref{proposition_manifold}.
	We note that this relies on the fact that \( \Phi \) is a left-inverse of the embedding of \( \manifold{N}_{0} \) into \( \manifold{N}_{\lambda} \).
\end{proof}

Our proof will rely on the following nonlinear uniform boundedness principle.

\begin{theorem}[Nonlinear uniform boundedness principle]
\label{thm:NonlinUnifBddPrinciple}
	Let \( \manifold{M} \) be an open domain of a Euclidean space or a compact Riemannian manifold, and let \( \manifold{N} \) be a connected Riemannian manifold.
	If \( 1 \leq p < m \) and if there exists a sequence \( \brk{u_{\ell}}_{\ell \in \Nset} \) in \( H^{1,p}_W
	\brk{\Bset^m, \manifold{N}} \) such that
	\begin{equation}
	\label{eq_d537ac00f9c8ab35}
		\sup_{\ell \in \Nset} \frac{\mathcal{E}^{1,p}_{\mathrm{rel}} \brk{u_{\ell}, \Bset^{m}}}{\mathcal{E}^{1,p} \brk{u_{\ell}, \Bset^{m}}}
		=
		+\infty\text{,}
	\end{equation}
	then for every \(b \in \manifold{N}\), there exists a map \( v \in \overline{H^{1,p}_W \brk{\Bset^m, \manifold{N}}} \setminus H^{1,p}_{W}\brk{\Bset^m, \manifold{N}} \) such that \(\operatorname{tr}_{\partial \Bset^m} v = b\).
\end{theorem}
\begin{proof}
The classical statement of \cite{Hang_Lin_2003_III}*{Th.\ 9.6} (see also~\cite{Monteil_VanSchaftingen_2019}) gives \(v \in W^{1,p} \brk{\Bset^m, \manifold{N}}\);
an inspection of the proof shows that \(v\) is obtained as a strong limit of reparametrisations of the \(v_\ell\) and lies  thus in \(\overline{H^{1,p}_W \brk{\Bset^m, \manifold{N}}}\).
\end{proof}

\Cref{theorem_main} will follow from the slightly stronger statement.

\begin{theorem}
\label{theorem_non_strongly_closed}
For every \(p \in \Nset \setminus \set{1}\), there exists a compact manifold \(\manifold{N}\)
such that if \(\dim \manifold{M} > p\), then
\[
 H^{1, p}_W \brk{\manifold{M}, \manifold{N}}
 \subsetneq
 \smash{\overline{H^{1, p}_W \brk{\manifold{M}, \manifold{N}}}}
\]
\end{theorem}

\begin{proof}
	It suffices to apply the nonlinear uniform boundedness principle for the weak approximation of \cref{thm:NonlinUnifBddPrinciple} and a local chart argument, noting that each of the maps \( u_{\ell} \) is actually smooth outside of a finite number of point singularities so that it can can always be approximated weakly with smooth maps \cite{Bethuel_1991,Hang_Lin_2003_III}.
\end{proof}

We observe that this proof yields a slightly stronger conclusion than required: it shows that the obstruction to weak approximation already arises if one considers continuous Sobolev maps rather than smooth maps.

With a little more effort, the failure of weak approximation can be shown to be generic in the sense of Baire category.

\begin{theorem}
\label{theorem_meagre}
For every \(p \in \Nset \setminus \set{1}\), there exists a compact manifold \(\manifold{N}\)
such that if \(\dim \manifold{M} > p\), then
\(
 H^{1, p}_W \brk{\manifold{M}, \manifold{N}}
\)
is meager in \(\smash{\overline{H^{1, p}_W \brk{\manifold{M}, \manifold{N}}}}\) with respect to the strong topology.
\end{theorem}
\begin{proof}
For every \(t \in \intvo{0}{\infty}\), we define the set
\[
 \mathcal{F}_t \defeq
 \set{u \in W^{1, p} \brk{\manifold{M}, \manifold{N}} \st
\mathcal{E}^{1,p}_{\mathrm{rel}} \brk{u, \manifold{M}} \le t }.
\]
In particular, \( \mathcal{F}_{t} \subset H^{1, p}_W \brk{\manifold{M}, \manifold{N}} \).
By the lower semicontinuity of \(\mathcal{E}^{1,p}_{\mathrm{rel}}\), the set \(\mathcal{F}_t\) is closed. Next, given \(u \in \mathcal{F}_t\), by an opening construction \cite{Brezis_Li_2001}, there exists \(a \in \manifold{M}\) and a family \(u_\varepsilon \in W^{1, p} \brk{\manifold{M}, \manifold{N}}\) such that \(u_\varepsilon  = u\brk{a}\) on \(B_\varepsilon \brk{a}\) and \( u_{\varepsilon} = u \) outside of \( B_{2\varepsilon}\brk{a} \). 
An inspection of the proof of the opening technique shows that \(u_\varepsilon  \) is a strong limit of maps in \(H^{1, p}_W \brk{\manifold{M}, \manifold{N}}\) that are constant on \(B_\varepsilon \brk{a}\).
Letting \(v\) be given by \cref{thm:NonlinUnifBddPrinciple} with \( \tr_{\partial\Bset^{m}}v = u\brk{a} \), we observe that the proof in \cite{Hang_Lin_2003_III}*{Th.\ 9.6} (see also~\cite{Monteil_VanSchaftingen_2019}) shows that \(v\) is a strong limit of maps in \(H^{1, p}_W \brk{\Bset^m, \manifold{N}}\) that are constant on their boundary.
Defining in a suitable local chart
\[
 v_\varepsilon \brk{x}
 \defeq
 \begin{cases}
 u_{\varepsilon} \brk{x} &\text{if \(x \not \in B_{2\varepsilon}\brk{a}\),}\\
 v \brk{2x/\varepsilon} &\text{if \(x \in B_{2\varepsilon}\brk{a}\),}
 \end{cases}
\]
we have \(v_\varepsilon \in \overline{H^{1, p}_W \brk{\manifold{M}, \manifold{N}}} \setminus H^{1, p}_W \brk{\manifold{M}, \manifold{N}} \) and \(v_\varepsilon \to u\) strongly as \(\varepsilon \to 0\).
We have thus showed that \(\mathcal{F}_t\) is rare.

Since
\[
   H^{1, p}_W \brk{\manifold{M}, \manifold{N}}
   = \bigcup_{n \in \Nset} \mathcal{F}_n,
\]
this shows that \(H^{1, p}_W \brk{\manifold{M}, \manifold{N}}\) is meagre.
\end{proof}

\begin{remark}
\label{rmk:target_connected_sum}
The essential features of the manifold \(\manifold{N}_{\lambda}\) given by \cref{proposition_manifold} used in the proofs of \cref{prop_nonestimate_manifold} which yields \cref{theorem_main,theorem_meagre,theorem_non_strongly_closed} are
\begin{enumerate}[label=(\alph*)]
 \item
 \label{it_coojofoocooNeiph4aeb8rau}
 \(\mathcal{N}_\lambda \supseteq \mathcal{N}_0\) up to a bi-Lipschitz embedding of \(\mathcal{N}_0\) into \(\mathcal{N}\),
 \item
 \label{it_eeGaingeic8Seomi0oYoh2Ke}
 there is a Lipschitz retraction of from \(\mathcal{N}_\lambda\) to \(\mathcal{N}_0\).
\end{enumerate}
This will be the case if \(\mathcal{N}_{\lambda}\) is replaced by the connected sum \(\mathcal{N}_{\lambda}\,\#\, \manifold{N}_{\sharp}\) where \(\manifold{N}_{\sharp}\) is a compact \(n\)-dimensional manifold.
Indeed, since \(\manifold{N}_0\) is a \(p\)-dimensional object and since \(\dim \manifold{N}_\lambda = n \ge p+1\), up to a diffeomorphism, we can assume that the connected sum is performed along a ball that does not intersect \(\manifold{N}_0\) so that \ref{it_coojofoocooNeiph4aeb8rau} holds;
moreover the homotopy extension property allows us to consider that the Lipschitz retraction given by \ref{it_eeGaingeic8Seomi0oYoh2Ke} is constant on the ball used in the gluing and thus to define it for \(\mathcal{N}_{\lambda}\,\# \,\manifold{N}_{\sharp}\), showing that \ref{it_eeGaingeic8Seomi0oYoh2Ke} is satisfied.
In particular, we can take as a target
\[
 \Tset^{p + 1}\,\#\,\Tset^{p + 1}\,\#\, \manifold{N}_{\sharp},
\]
where \(\manifold{N}_{\sharp}\) is any compact \(\brk{p+1}\)-dimensional Riemannian manifold.
\end{remark}

\section{Whitehead products and spheres}
\label{section_whitehead}

\subsection{Construction of mappings}

We now move towards our second main result, Theorem~\ref{theorem_counterexample_hopf}.
We first use the Whitehead product construction~\cite{Whitehead_1941} to construct a periodic map \( \Rset^{4n} \to \Sset^{2n} \) with singularities located at every point in \( \Zset^{4n} \) and Hopf degree \( 2 \) around each singularity.

\begin{proposition}
\label{proposition_whitehead_periodic_link}
For every \( n \geq 1 \), there exists a mapping 
\(u \colon \Rset^{4n} \to \Sset^{2n}\) such that 
\begin{enumerate}[label=(\roman*)]
 \item 
for every \(h \in \Zset^{4n}\),
\[
 u \brk{x + h} = u \brk{x}\text{,}
\]
\item 
\(
 \deg_H u\restr{\partial \intvc{-1/2}{1/2}^{4n}} = 2
\),
\item 
\(u \in W^{1,4n-1}_{\textnormal{loc}}\brk{\Rset^{4n}, \Sset^{2n}} \cap C \brk{\Rset^{4n} \setminus \Zset^{4n}, \Sset^{2n}}\)
and 
\(\abs{\Deriv u \brk{x}} \dist \brk{x, \Zset^{4n}} \le C\) for a.e.\ \( x \in \Rset^{4n} \setminus \Zset^{4n} \).
\end{enumerate}
\end{proposition}
\begin{proof}
We consider a mapping 
\(f \in C^{\infty} \brk{\Rset^{2n}, \Sset^{2n}}\) such that
\(f = b\) in \(\Rset^{2n} \setminus \intvc{-1/2}{1/2}^{2n}\) and \(\deg f = 1\).
We note that, since \(f\) is constant outside a compact set, it can be associated thanks to a stereographic projection to a map from \(\Sset^{2n}\) to \(\Sset^{2n}\).
Therefore, \( \deg f \) is well-defined. 
In fact, such a map \(f\) can actually be defined as a truncated inverse stereographic projection.

We then define \(v \colon \partial\brk{\intvc{-1/2}{1/2}^{4n}} \to \Sset^{2n} \) as the Whitehead product of \(f\) with itself~\cite{Whitehead_1941}. More precisely we set, for
\(x = \brk{x', x''} \in \partial \brk{\intvc{-1/2}{1/2}^{4n}} \subseteq \Rset^{4n} = \Rset^{2n} \times \Rset^{2n}\),
\begin{equation}
\label{eq_Eich4AiPhebahshaichoa7te}
 v \brk{x',x''}
 \defeq 
 \begin{cases}
  f \brk{x'} &\text{if \(\abs{x'}_{\infty} < 1/2\),}\\
  f \brk{x''} &\text{if \(\abs{x''}_{\infty} < 1/2\),}\\
  b & \text{otherwise.}
 \end{cases}
\end{equation}
The Hopf invariant of \(v\) can be computed by classical properties of the Whitehead product as
\(
 \deg_H \brk{v} = 2 \brk{\deg \brk{f}}^2 = 2\);
see e.g.~\cite{Whitehead_1978}*{XI~(2.5)}.
Here we use the fact that we work with a sphere of even dimension.
Indeed, the exact same construction could have been performed to obtain a map \(v \colon \partial\intvc{-1/2}{1/2}^{4n+2} \to \Sset^{2n+1} \), but one would then have had \( \deg_H \brk{v} = 0 \).
More generally, the Hopf degree of \emph{any} continuous \( \Sset^{4n+1} \to \Sset^{2n+1} \) map is zero; see e.g.~\cite{Whitehead_1978}*{XI~(2.4)}.

We then extend \(v\) homogeneously to \(\intvc{-1/2}{1/2}^{4n} \setminus \set{0}\) by setting \(v \brk{x} \defeq v\brk{x/\brk{2 \abs{x}_\infty}}\).
It follows from \eqref{eq_Eich4AiPhebahshaichoa7te} that, 
if \(x\), \(y \in \intvc{-1/2}{1/2}^{4n} \setminus \Zset^{4n}\) and
\(x - y \in \Zset^{4n}\),
then \(v \brk{x} = v \brk{y}\), so that \(v\) can be extended periodically to \(\Rset^{4n}\).
\end{proof}

\subsection{Lower estimate on the relaxed energy}

Below, the target \( \manifold{N} \) is assumed to be a compact manifold without boundary.
In the application we have in mind, \( \manifold{N} \) will be a sphere.
We start by describing a basic cylinder-type construction.
Given \( m \geq 1 \), we let \( Q^{m} = \intvc{0}{1}^{m} \) be the unit \( m \)-dimensional cube.
If \( u \), \( v \in W^{1,p}\brk{Q^{m-1}, \manifold{N}} \cap C\brk{Q^{m-1}, \manifold{N}} \), then there exists \( \delta \in \intvo{0}{\infty} \) such that, if \( \norm{u-v}_{L^{\infty}\brk{\partial Q^{m-1}}} \leq \delta \), then the map \( w \colon \partial Q^{m} \to \manifold{N} \) given by
\begin{equation*}
	w\brk{x',x_{m}}
	=
	\begin{cases}
		u\brk{x'} & \text{if \( x_{m} = 0 \),} \\
		\Pi_{\manifold{N}}\brk{\brk{1-x_{m}}u\brk{x'} + x_{m}v\brk{x'}} & \text{if \( x' \in \partial Q^{m-1} \),} \\
		v\brk{x'} & \text{if \( x_{m} = 1 \)}
	\end{cases}
\end{equation*}
is well-defined and belongs to \( W^{1,p}\brk{\partial Q^{m}, \manifold{N}} \cap C\brk{\partial Q^{m}, \manifold{N}} \).
Moreover, for every such \( \delta \), a straightforward computation shows that
\begin{equation}
\label{eq:CylinderEstimate}
	\int_{\partial Q^{m}} \abs{\Deriv w}^{p}
	\leq
	\int_{Q^{m-1}} \abs{\Deriv u}^{p} + \int_{Q^{m-1}} \abs{\Deriv v}^{p} + C\int_{\partial Q^{m-1}} \abs{\Deriv u}^{p} + C\int_{\partial Q^{m-1}} \abs{\Deriv v}^{p} + C\delta^{p}\text{.}
\end{equation}

With this construction at hand, we prove the following lower estimate on the relaxed energy.
In order to state the following proposition, we recall that any (subset of a) hyperplane can be endowed with two different orientations.
For a face \( \sigma \) of \( Q^{m} \subset \Rset^{m} \), we consider the natural orientation induced by the outward normal vector with respect to \( Q^{m} \). 
Of course, if \( \sigma \) is now viewed as a face of a skeleton made of several cubes, then it may receive both orientations, depending on the cube with respect to which it is considered.
We will view the two copies of the same face with reverse orientations as two distinct oriented faces.
Moreover, we will denote by \( \abs{\sigma}  \) the unoriented face \( \sigma \).

\begin{proposition}
\label{proposition_grid_branched_transport}
Let \(L_{\ell}^{j}\) be the \(j\)-dimensional skeleton of the cube \(Q_{\ell} = \intvc{0}{\ell}^{4n}\).
Assume that \((u_{k})_{k \in \Nset}\) is a sequence in \(C\brk{L_{\ell}^{4n-1}, \Sset^{2n}} \cap W^{1, 4n-1} \brk{L_{\ell}^{4n-1}, \Sset^{2n}}\) converging weakly in \( W^{1,4n-1} \) to some \(u \in C\brk{L_{\ell}^{4n-1}, \Sset^{2n}} \cap W^{1, 4n-1} \brk{L_{\ell}^{4n-1}, \Sset^{2n}}\)
and such that the traces on \(L_{\ell}^{4n-2}\) also converge weakly in \(W^{1, 4n-1}\).
Then, for every oriented face \(\sigma\) of \(L_{\ell}^{4n-1}\), there exists some \(d_{\sigma} \in \Zset\) such that \( d_{-\sigma} = -d_{\sigma} \), where \( -\sigma \) denotes the face \( \sigma \) with reverse orientation, such that, up to a subsequence,
\[
\sum_{\sigma \textnormal{ face of } L^{4n - 1}} \abs{d_{\sigma}}^{1 - \frac{1}{4n}}
\le \liminf_{k \to \infty} C\brk[\bigg]{\int_{L_{\ell}^{4n-1}} \abs{\Deriv u_{k}}^{4n-1}
	+ \int_{L_{\ell}^{4n-2}} \abs{\Deriv u_{k}}^{4n-1}}\text{.}
\]
Moreover, there exists some \(k_* \in \Nset\) such that for every cube \(Q\) of \(L_{\ell}^{4n}\),
\[
 \sum_{\sigma \textnormal{ face of }Q} d_{\sigma} = \deg_{H} \brk{ u\restr{\partial Q}} - \deg_{H} \brk{u_{k_*} \restr{\partial Q}}\text{.}
\]
\end{proposition}

The apparent notational complexity of the statement should not obscure the intuition behind it: the difference of degrees between the maps \( u_{k} \) for large \( k \) and the limiting map \( u \) can be attributed in a coherent way to the different faces of \( L^{4n-1}_{\ell} \), so that the difference of degrees on one cube is the sum of the contributions of all its oriented faces.
Moreover, we have an estimate available for these partial degrees on faces by the Sobolev energy of the approximating maps, involving a power \( 1 - \frac{1}{4n} \), whose importance will become apparent later on, in relation with branched optimal transportation considerations.

Proposition~\ref{proposition_grid_branched_transport} leaves some freedom to the \(d_\sigma\) on the boundary.
Indeed, whereas interior faces all belong to two cubes and are involved, since \(d_\sigma = -d_{-\sigma}\), in two degree conservation constraints, the boundary faces are merely involved in a single one, so that they are much less constrained.

\begin{proof}[Proof of Proposition~\ref{proposition_grid_branched_transport}]
By compactness of the embedding \( W^{1,4n-1} \hookrightarrow L^{\infty} \) in dimension \( 4n-2 \), we assume without loss of generality that
\[\norm{u_k - u}_{L^\infty \brk{L^{4n - 2}_\ell}} \le \delta\text{.}
\]
For every face \( \sigma \) of \( L^{4n-1}_{\ell} \), we consider \( u_{\sigma} \) and \( u_{\sigma,k} \) the restrictions of \( u \) and \( u_{k} \) to \( \sigma \), and we let \( w_{\sigma,k} \) be the map obtained using the above cylinder construction with respect to \( u_{\sigma,k} \) and \( u_{\sigma} \), with the following orientations: \( \sigma \) is placed in \( \Rset^{4n-1} \times \set{0} \) so that the orienting normal vector points upwards, the map \( u_{\sigma,k} \) is placed at the bottom of the cylinder, and the map \( u_{\sigma} \) is placed at the top of the cylinder.
We define \( d_{\sigma,k} \) to be the Hopf degree of the map \( w_{\sigma,k} \).
With our orientation convention, we have \( d_{-\sigma,k} = -d_{\sigma,k} \):
Indeed, reversing the orientation in our construction amounts to reflect the domain with respect to the hyperplane \( \Rset^{4n-1} \times \set{1/2} \), which, in terms of homotopy, corresponds to taking the inverse.

We first prove that 
\[
\sum_{\sigma \textnormal{ face of }Q} d_{\sigma,k} = \deg_{H} \brk{u\restr{\partial Q}} - \deg_{H} \brk{u_{k} \restr{\partial Q}}\text{.}
\]
For this purpose, we consider the set \( \widetilde{Q} \) obtained by taking the union of \( Q \) and all the cubes congruent to \( Q \) and sharing exactly one face with it.
We observe that \( \partial\widetilde{Q} \) is homeomorphic to the boundary of a cube.
We define a map \( w \) on the \( \brk{4n-1} \)-skeleton of \( \widetilde{Q} \) by letting \( w \) be equal to \( u_{k} \) on \( \partial Q \), and to \( w_{\sigma,k} \) on the boundary of the cube that has been glued to \( Q \) along the face \( \abs{\sigma} \) (after a suitable rotation so that both maps coincide on \( \abs{\sigma} \)).
It is straightforward to see that \( w\restr{\partial\widetilde{Q}} \) is homotopic to \( u\restr{\partial Q} \), since the maps \( w_{\sigma,k} \) are constructed on the vertical faces as the projection of a linear interpolation between \( u_{\sigma,k} \) and \( u_{\sigma} \) on \( \partial\abs{\sigma} \).
On the other hand, we have the identity
\[
	\deg_{H} \brk{w\restr{\partial\widetilde{Q}}} = \sum_{\sigma \textnormal{ face of }Q} d_{\sigma,k} + \deg_{H} \brk{u_{k} \restr{\partial Q}}\text{.}
\]
This follows directly from the construction of \( w \), since gluing two maps defined on the boundary of two neighboring cubes corresponds exactly to taking the sum of the associated homotopy classes.
Combining both these pieces of information, we deduce that
\[
	\deg_{H} \brk{u\restr{\partial Q}} = \deg_{H} \brk{w\restr{\partial\widetilde{Q}}} = \sum_{\sigma \textnormal{ face of }Q} d_{\sigma,k} + \deg_{H} \brk{u_{k} \restr{\partial Q}}\text{,}
\]
which proves the required formula.

It remains to prove the integral estimate.
For this purpose, we recall Rivière’s estimate on the Hopf invariant~\cite{Riviere_1998}
\[
 	\abs{\deg_H \brk{f}}^{1 - \frac{1}{4n}}\le 
 	\C
 	\int_{\Sset^{4n - 1}} \abs{\Deriv f}^{4n-1}\text{.}
\]
Applied to the maps \( w_{\sigma,k} \), this yields
\[
	\abs{d_{\sigma,k}}^{1 - \frac{1}{4n}}
	\leq
	\C
	\int_{\partial Q^{4n}} \abs{\Deriv w_{\sigma,k}}^{4n-1}\text{.}
\]
Invoking~\eqref{eq:CylinderEstimate}, the compactness of the embedding \( W^{1,4n-1} \hookrightarrow L^{\infty} \) in dimension \( 4n-2 \),  and the weak lower semicontinuity of the norm of the gradient, we obtain
\[
  \limsup_{k \to \infty}  \sum_{\sigma \textnormal{ face of } L^{4n - 1}} \abs{d_{\sigma, k}}^{1 - \frac{1}{4n}}
 \le \C \liminf_{k \to \infty} \brk[\bigg]{\int_{L_{\ell}^{4n-1}} \abs{\Deriv u_{k}}^{4n-1}
 + \int_{L_{\ell}^{4n-2}} \abs{\Deriv u_{k}}^{4n-1} }\text{.}
\]
The conclusion then holds provided \( k_* \in \Nset \) is taken sufficiently large, letting \( d_{\sigma} \coloneqq d_{\sigma,k_{\ast}} \).
\resetconstant
\end{proof}

\begin{corollary}
\label{corollary_grid_branched_transport}
Under the assumptions of Proposition~\ref{proposition_grid_branched_transport}, if for every cube \(Q\) of \(L_{\ell}^{4n}\),
\[
 \deg_{H} \brk{u\restr{\partial Q}} = 2 
 \quad
 \text{and}
 \quad
 \deg_{H} \brk{{u_{k}}\restr{\partial Q}} = 0
 \quad
 \text{for every \( k \in \Nset \),}
\]
then 
\[
 \liminf_{k \to \infty} \int_{L_{\ell}^{4n-1}} \abs{\Deriv u_{k}}^{4n-1}
 + \int_{L_{\ell}^{4n - 2}} \abs{\Deriv u_{k}}^{4n-1}
 \ge C\ell^{4n} \ln \ell\text{.}
\]
\end{corollary}

Let us comment that the assumption that \( \deg_{H} \brk{{u_{k}}\restr{\partial Q}} = 0 \) for every \( k \in \Nset \) holds in particular if the \( u_{k} \) are defined and continuous on the whole \( Q_{\ell} \), which is the context in which we are going to apply it to obtain our final counterexample.

\begin{proof}[Proof of \cref{corollary_grid_branched_transport}]
This follows from the lower estimate on branched transportation~\cite{Bethuel_2020}*{Theorem A.1}, 
 and is linked to the fact that the exponent \( \alpha = 1 - 1/N \) is the critical exponent for the irrigability of the Lebesgue measure in dimension \( N \); see~\cite{Bernot_Caselles_Morel_2009} and the references therein for a detailed discussion about branched optimal transportation problems.
 Indeed, the Hopf degree assumptions on the boundary of every cube \( Q \) of \( L^{4n}_{\ell} \) imply that
 \begin{equation}
 \label{eq:KirchoffLaw}
 	\sum_{\sigma \textnormal{ face of }Q} d_{\sigma,k} = 2\text{.}
 \end{equation}
 We define a transport plan to the boundary for the measure that consists of one Dirac mass with weight equal to \( 2 \) at the center of each cube of \( L_{\ell}^{4n} \) by connecting the center of each cube to the centers of all the neighboring cubes by a segment, where each segment is weighted by the number \( d_{\sigma,k} \) corresponding to the only face \( \sigma \) that it crosses.
 The orientation of the segments is determined by the sign of \( d_{\sigma,k} \).
 By the fact that \( d_{-\sigma,k} = -d_{\sigma,k} \) and~\eqref{eq:KirchoffLaw}, this indeed defines a transport plan from the above described measure to the boundary, and the mass of this transport plan with respect to the branched optimal transportation with power \( \alpha = 1-\frac{1}{4n} \) is exactly given by
 \[
 	\sum_{\abs{\sigma}} \abs{d_{\sigma}}^{1 - \frac{1}{4n}}\text{.}
 \]
 We conclude by combining the lower estimate on branched transportation and Proposition~\ref{proposition_grid_branched_transport}.
\end{proof}

\begin{remark}
Strictly speaking, our lower bound just uses sublinear transport (without branching), which means that a self-contained analysis would not require to take into account the appearance of additional points.
\end{remark}

\subsection{The counterexample}

We take \(u\) given by Proposition~\ref{proposition_whitehead_periodic_link}.
Thanks to Fubini--Tonelli and Fatou, we can apply Corollary~\ref{corollary_grid_branched_transport}, and we get the conclusion by the uniform boundedness principle for the weak approximation.

\section{Obstruction for higher order Sobolev energies}
\label{section_higher_order}

In this final section, we briefly explain how a higher order counterpart can be deduced from our main results
by adapting our previous constructions so that the higher order nonlinear uniform boundedness principle \cite{Monteil_VanSchaftingen_2019}*{Theorem~6.1} can be applied.
In order to do this, we need maps with improved regularity and with a control on the energy on an enlarged domain;
the explicit construction that we made in Sections~\ref{section_analytical_obstruction} and~\ref{section_whitehead} can be adapted with reasonable effort.

We start with the construction in Section~\ref{section_analytical_obstruction} of the map into the skeleton of the torus.
Let \( n \geq 3 \), let \( s \geq 1 \), and let \( 1 \leq p < \infty \) be such that \( sp = n-1 \).
We keep denoting by \( L^{j} \) the \( j \)-dimensional skeleton of \( \Rset^n \).
In Section~\ref{section_from_skeleton_to_manifold}, our counterexample was built up from \(v \defeq \pi \compose u\) defined in Proposition~\ref{prop_counterexample_skeleton},  where \(\pi \colon \lifting{\manifold{N}}_0 \to \manifold{N}_0\) was the universal covering and \(u\)
was the homogeneous extension of the inclusion
\( i \colon L^{n-1} \to \lifting{\manifold{N}}_{0} \).
This construction can be adapted to higher order regularity.

\begin{lemma}
\label{lemma_higher_order_construction}
	Given \( \varepsilon > 0 \), there exists a map \( v^{\sm} \in C^{\infty}\brk{\Rset^n \setminus \Sigma, \manifold{N}_{\lambda}} \), with \(\Sigma = \brk{\Zset+1/2}^{n}\), such that
	\begin{enumerate}[label=(\roman*)]
	\item \( v^{\sm} \) is periodic under the action of \( \Zset^{n} \),
	\item\label{it_estimate_derivatives_thickening} for every \(k \in \Nset \setminus\set{0}\),
	\(\sup_{x \in \Rset^n \setminus \Sigma} \abs{\Deriv^k v^{\sm} \brk{x}} \dist \brk{x, \Sigma}^k < \infty\),
	\item
	\label{it_uniform_distance_smoothing}
	\(
		\norm{v - v^{\sm}}_{L^{\infty}\brk{\Rset^n}} \le \varepsilon\text{.}
	\)
	\end{enumerate}
\end{lemma}
\begin{proof}
By convolution of the map \(v\) with a suitable mollifier and  projection on  \( \manifold{N}_{\lambda} \) via the smooth nearest point projection, we obtain a smooth periodic map \( v^{\ext} \colon L^{n-1} + B_{\delta} \to \manifold{N}_{\lambda} \).
	
	Then, we rely on a thickening procedure to extend \( v^{\ext}\) smoothly to \( \Rset^n \setminus \Sigma \).
	More precisely, on each cube of \( L^{n} \), that we identify with the cube \( \intvc{-1}{1}^{n} \), instead of precomposing \( v^{\ext} \) with \( x/\abs{x}_\infty \) as we would do for homogeneous extension, we precompose it with a map of the form \( \lambda\brk{\abs{x}_{2q}}x \), where \( \lambda \colon \intvl{0}{\infty} \to \intvr{1}{\infty} \) is a suitable smooth nonincreasing map which is constantly equal to \( 1 \) in a neighborhood of \( \intvr{1}{\infty} \) and \( q \in \Nset\) is chosen sufficiently large so that
	\(\set{x \in \Rset^n \st \abs{x}_{2q} \ge 1} \subseteq \partial \brk{\intvc{-1}{1}^n} + B_\delta\); see~\cite{Bousquet_Ponce_VanSchaftingen_2015}*{Section~4}.
	Let us call \( v^{\sm} \) the map obtained through this process.
	
	Since the thickening procedure was applied to the periodic map \( v^{\ext} \), the map \( v^{\sm}\) is periodic.
	In addition, a suitable choice of the number \( \delta \), the convolution parameter, and the function \( \lambda \), ensures that~\ref{it_uniform_distance_smoothing} holds.
	Property~\ref{it_estimate_derivatives_thickening} follows from the estimates for thickening.
\end{proof}

We may now compose \( v^{\sm} \) with the Lipschitz map \( \Phi \colon \manifold{N}_{\lambda} \to \manifold{N}_{0} \) provided by Proposition~\ref{proposition_manifold}, which yields a map \( v^{\lip} \defeq  \Phi \compose v^{\sm} \colon \Rset^n \to \manifold{N}_{0} \) with the following properties:
\begin{enumerate}[label=(\roman*)]
  \item
	\label{item:vlip_regular} 
	\( v^{\lip} \in W^{1,n-1}_{\mathrm{loc}} \brk{\Rset^n, \manifold{N}_{0}} \cap C\brk{\Rset^n \setminus \Sigma, \manifold{N}_{0}} \);
  \item
	\label{item:vlip_periodic} 
	\( v^{\lip} \) is periodic under the action of \( \Zset^{n} \);
  \item\label{item:dist_vlip_v} \( \norm{v^{\lip} - v}_{L^{\infty}\brk{\Rset^n}} \le \Cl{cst_proj}\varepsilon \).
\end{enumerate}

Assertion~\ref{item:dist_vlip_v} follows from the fact that \( \Phi \colon \manifold{N}_{\lambda} \to \manifold{N}_{0} \) is a left-inverse of the embedding of \( \manifold{N}_{0} \) into \( \manifold{N}_{\lambda} \).

If \(Q_\ell \defeq \intvc{0}{\ell}^n\), the map \( v^{\lip} \restr{Q_\ell}\) can then be lifted to a map \( u \in W^{1,n-1}\brk{Q_{\ell}, \smash{\widetilde{\manifold{N}}_{0}}}  \cap C\brk{Q_{\ell} \setminus \Sigma, \smash{\widetilde{\manifold{N}}_{0}}} \), which can easily be checked to satisfy the assumptions of Proposition~\ref{proposition_gap_v_and_proj}, with the exception that~\eqref{eq_keleopheuvahsuuhu0hei3Ee} should be replaced by \( u\brk{Q_{\ell,\alpha}} \subset Q_{\ell,\alpha} + B_{\Cr{cst_proj}\varepsilon} \).
However, it is readily verified that the conclusion of Proposition~\ref{proposition_gap_v_and_proj} remains valid under this slightly weaker assumption, provided that \( \varepsilon > 0 \) is chosen sufficiently small.
Therefore, as the reasoning carried out in Section~\ref{subsection:AnalyticalObstructionSkeleton} does not depend on the specific form of the map, we deduce that the map \( v^{\lip} \) also satisfies the estimate
\begin{equation*}
	\liminf_{\ell \to \infty}\frac{\mathcal{E}^{1,n-1}_{\textnormal{rel}}\brk{v^{\lip}, Q_{\ell}}}{\ell^{n}\ln{\ell}} > 0\text{.}
\end{equation*}
This also implies that 
\begin{equation}
\label{eq_ohph1icip2eitei2gei0Leez}
	\liminf_{\ell \to \infty}\frac{\mathcal{E}^{1,n-1}_{\textnormal{rel}}\brk{v^{\sm}, Q_{\ell}}}{\ell^{n}\ln{\ell}} > 0\text{,}
\end{equation}
since any sequence of smooth maps that weakly converges to \( v^{\sm} \) can be projected to a sequence of Lipschitz maps weakly converging to \( v^{\lip} \) through the Lipschitz map \( \Phi \colon \manifold{N}_{\lambda} \to \manifold{N}_{0} \) provided by Proposition~\ref{proposition_manifold}.

By compactness of \( \manifold{N}_{\lambda} \),
we have \(W^{s,p}\brk{\manifold{M}, \manifold{N}_{\lambda}}\subseteq L^\infty\) and thus, by the Gagliardo--Nirenberg interpolation inequality \( W^{s,p} \cap L^{\infty} \subset W^{1,sp} \), see e.g.~\cite{Runst1986}*{Lemma~3.1 p.~329} or~\cite{BrezisMironescu2001}*{Corollary~3.2}, for every \( w \in W^{s,p}\brk{Q_\ell, \manifold{N}_{\lambda}} \), we have
\begin{equation}
\label{eq_tuLooTadaoya9Ceilie2ieta}
\begin{split}
\mathcal{E}^{1,sp}\brk{w, Q_{\ell}}
\le \C \norm{w}_{L^\infty}^{p\brk{s - 1}}
\brk{
\mathcal{E}^{s,p}\brk{w, Q_{\ell}} + \ell^{-sp} \norm{w}_{L^p}^p}
	\leq
	\C \brk{\mathcal{E}^{s,p}\brk{w, Q_{\ell}} + \ell}\text{,}
\end{split}
\end{equation}
so that by definition of relaxed energy and \eqref{eq_ohph1icip2eitei2gei0Leez},
\begin{equation*}
	\liminf_{\ell \to \infty}\frac{\mathcal{E}^{s, p}_{\textnormal{rel}}\brk{v^{\sm}, Q_{\ell}}}{\ell^{n}\ln{\ell}} > 0\text{.}
\end{equation*}

On the other hand, by the periodic character of the map \( v^{\sm} \), it is straightforward to check that, if \( 2Q_{\ell} \) denotes the cube with the same center as \( Q_{\ell} \) and double edge length,  we have
\begin{align*}
	\mathcal{E}^{s,p}\brk{v^{\sm}, 2Q_{\ell}} &\leq \C\ell^{n}&
	&\text{and} &
	\mathcal{E}^{1,p}\brk{v^{\sm}, 2Q_{\ell}} &\leq \C\ell^{n}\text{.}
\end{align*}
The case where \( s \notin \Nset \) might be somehow more subtle, since the Gagliardo seminorm is not additive.
Working with mixed integrals, writing again \( Q^{n} = \intvc{0}{1}^{n} \), one has 
\[
 \iint_{Q^n \times \Rset^n} \frac{\abs{\Deriv^{\floor{s}}v^{\sm}\brk{x} - \Deriv^{\floor{s}}v^{\sm}\brk{y}}^{p}}{\abs{x - y}^{n + \brk{s-\floor{s}}p}} \dif x \dif y < \infty\text{,}
\]
and one can then conclude by additivity and periodicity (see~\cite{Detaille2023}*{Lemma~2.1} for a precise statement and a brief history of such techniques).

We are thus in position to apply the higher order nonlinear uniform boundedness principle \cite{Monteil_VanSchaftingen_2019}*{Theorem~6.1}, which proves the existence of a map \( u \in W^{s,p}\brk{\manifold{M}, \manifold{N}_{\lambda}} \) which cannot be weakly approximated by smooth maps in \( W^{s,p}\brk{\manifold{M}, \manifold{N}_{\lambda}} \), hence completing the proof of Theorem~\ref{theorem_counterexample_higher_order}.
The proof of the meagreness also uses the opening construction for higher-order Sobolev mappings \citelist{\cite{Bousquet_Ponce_VanSchaftingen_2015}\cite{Detaille2023}}.

\resetconstant

A similar process applied to our construction involving the Whitehead product yields a higher order counterpart to Theorem~\ref{theorem_counterexample_hopf}.
Since the procedure is actually simpler than for Theorem~\ref{theorem_main}, as there is no lifting or passing from a cell complex to a manifold involved, the target being a sphere, we omit the details.
However, we importantly mention that there is indeed a smoothing procedure to be performed.
Indeed, even though the map \( f \) in the proof of Proposition~\ref{proposition_whitehead_periodic_link} is taken to be smooth, the resulting map \( u \) does not need to be smooth, not only because of the homogeneous extension procedure, but already in the definition of \( v \), which need not be constant near the boundaries of the faces of \( \partial \brk{\intvc{-1/2}{1/2}^{4n}} \).
(The image of the construction on \( \partial \brk{\intvc{-1/2}{1/2}^{2}}\) is misleading as it is the only dimension where
\( v \) is constant near the boundary of the faces.)
Therefore, we need to argue as for the higher order counterpart of Theorem~\ref{theorem_main}: build a smooth variant of the original construction by first extending continuously to a neighborhood of the faces, then use regularization by convolution and thickening to replace homogeneous extension, and finally study the relaxed energy thanks to the Gagliardo--Nirenberg interpolation inequality to get Theorem~\ref{theorem_counterexample_hopf_higher_order}.

Let us close this section with some thoughts about the case \(0 < s < 1\).
In this case, there is no smoothing needed for our constructions to be in \(W^{s, p}\), but the Gagliardo--Nirenberg interpolation inequality cannot be used to estimate the relaxed energy.
Although it should be possible to adapt the essential part of the slicing and bubbling arguments, the crucial lower estimates on energies would require more work.
In order to extend Theorem~\ref{theorem_counterexample_higher_order} to \(0 < s < 1\), one would need a fractional version of the localized degree estimate of Proposition~\ref{proposition_conical_joint_degree_estimate};
for Theorem~\ref{theorem_counterexample_hopf_higher_order}, Proposition~\ref{proposition_grid_branched_transport} relies on Rivière’s Hopf invariant estimate~\cite{Riviere_1998}, which has only been partially extended to the fractional case for \(s \ge 1 - \frac{1}{4n}\)~\cite{Schikorra_VanSchaftingen_2020}.

\end{document}